\documentclass[11pt]{article}%
\usepackage{amsfonts}
\usepackage{amsthm}
\usepackage{enumerate}
\usepackage{natbib}
\usepackage{amsmath}
\usepackage{amssymb}
\usepackage{graphicx}
\usepackage{subfig}
\usepackage{color}
\usepackage{multirow}
\usepackage{booktabs}
\usepackage{tabularray}
\usepackage{enumitem}
\usepackage[bottom]{footmisc}
\usepackage{algorithm}
\usepackage{algorithmic}
\usepackage{utfsym}
\usepackage[bookmarks=false,colorlinks,linkcolor=black,anchorcolor=black,citecolor=black]%
{hyperref}%
\providecommand{\U}[1]{\protect\rule{.1in}{.1in}}
\newtheorem{theorem}{Theorem}
\newtheorem{corollary}{Corollary}

\newtheorem{lemma}{Lemma}
\newtheorem{proposition}{Proposition}
\newtheorem{assumption}{Assumption}
\theoremstyle{definition}

\newcommand{\crossmark}{\scalebox{0.75}{\usym{2613}}}

\DeclareMathOperator*{\sign}{sgn}

\DeclareMathOperator*{\argmin}{arg\,min}
\allowdisplaybreaks
\begin{document}
% Keywords command
\providecommand{\keywords}[1]
{
 \small	
 \textbf{\textit{Keywords---}} #1
}

\title{Shape-Constrained Distributional Optimization via Importance-Weighted Sample Average Approximation}
\date{}
\author{Henry Lam\thanks{Department of Industrial
Engineering and Operations Research, Columbia University.}
\and Zhenyuan Liu\footnotemark[1]
\and Dashi I. Singham\thanks{Operations Research Department, Naval Postgraduate School.}
}
\maketitle

\begin{abstract}
Shape-constrained optimization arises in a wide range of problems including distributionally robust optimization (DRO) that has surging popularity in recent years. In the DRO literature, these problems are usually solved via reduction into moment-constrained problems using the Choquet representation. While powerful, such an approach could face tractability challenges arising from the geometries and the compatibility between the shape and the objective function and moment constraints. In this paper, we propose an alternative methodology to solve shape-constrained optimization problems by integrating sample average approximation with importance sampling, the latter used to convert the distributional optimization into an optimization problem over the likelihood ratio with respect to a sampling distribution. We demonstrate how our approach, which relies on finite-dimensional linear programs, can handle a range of shape-constrained problems beyond the reach of previous Choquet-based reformulations, and entails vanishing and quantifiable optimality gaps. Moreover, our theoretical analyses based on strong duality and empirical processes reveal the critical role of shape constraints in guaranteeing desirable consistency and convergence rates.
\end{abstract}


\section{Introduction}\label{sec:introform}
% We first introduce our main formulation in Section \ref{sec:introform}.  Section \ref{sec:literature} reviews the literature on distributionally robust optimization relevant to our problem.  Section \ref{sec:SAA_DRO_contributions} presents our main contributions.
% \subsection{Formulation}
This paper studies optimization problems in the form%
\begin{align}
\sup_{f}\text{ }  &  \mathbb{E}_{f}[\phi_{0}(X)]\nonumber\\
\text{subject to }  &  \mathbb{E}_{f}[\phi_{j}(X)]\leq\mu_{j},j=1,\ldots,m\label{DRO_original}\\
&  \mathbb{E}_{f}[\phi_{j}(X)]=\mu_{j},j=m+1,\ldots,l\nonumber\\
&  f\in\mathcal{F}\nonumber
\end{align}
where the decision variable is a probability density $f$ that controls the expectation $\mathbb{E}_{f}[\cdot]$ on the random variable $X\in\mathcal{X}\subset\mathbb{R}^{n}$.  The functions $\phi_{j},j=0,\ldots,l$ are known functions from $\mathcal X$ to $\mathbb R$ that give rise to the objective and moment constraints. In particular, one of these $\phi_{j}$ in the moment constraints is chosen to be the indicator function $I(x\in\mathcal{X})$ to assert a total mass. Importantly, and as our main focus, the function class $\mathcal{F}$ contains shape information on $f$, such as monotonicity, convexity and various forms of unimodality. Note that one could generalize formulation \eqref{DRO_original} to allow probability distributions without densities, but at the expense of more technicality that we avoid digressing to in this paper.

% Our goal in this paper is to create a general-purpose methodology that is computationally efficient to solve the infinite-dimensional problem \eqref{DRO_original}. More precisely, we propose a sampling-based approach that integrates sample average approximation (SAA) with importance sampling, and study the 

% Formulation \eqref{DRO_original} and the methodology to be studied can be potentially extended to more general distributions that do not necessarily have a density, but for the purposes of clarity and applications that we will consider, we focus on \eqref{DRO_original}.

% (DO WE NEED DENSITY? SHOULD COMMENT. ZHENYUAN: YES. UNDER THE CURRENT FRAMEWORK/DERIVATION, WE CAN ONLY HANDEL DISTRIBUTIONS WITH DENSITIES.).

% shape information \DS{that represent shape constraints}. We assume our formulation is defined on the region : \DSQ{[DS:$n$ is defined later as the number of samples but I think here you mean the number of dimensions?]} $\phi_{j}$ and all functions in $\mathcal{F}$ are defined on $\mathcal{X}$ and the shape information is also imposed on $\mathcal{X}$. \sout{The first two constraints in (\ref{DRO_original}) are called moment constraints and the last constraint in (\ref{DRO_original}) is recognized as the shape constraint.}

Our primary motivation for studying \eqref{DRO_original} is that it gives rise to sharp worst-case bounds on expectation-type performance measures $\mathbb{E}_{f}[\phi_{0}(X)]$ when only partial information on $f$ is known. This problem arises generically in distributionally robust optimization (DRO) (\cite{delage2010distributionally,goh2010distributionally,kuhn2019wasserstein}), a methodology to tackle optimization under uncertainty that can be viewed as a generalization of classical robust optimization (\cite{ben2009robust,bertsimas2011theory}) when the uncertain parameter refers to the underlying probability distribution in a stochastic problem. Like classical robust optimization, DRO advocates decision-making that optimizes the worst-case scenario, and thus leads to a minimax problem where the outer minimization is over the decision and the inner maximization is over the uncertain probability distribution. In this regard, formulation \eqref{DRO_original} is precisely the inner maximization of a DRO, or can simply represent a worst-case estimation without reference to a decision. Here, the feasible region, known as the ambiguity set or uncertainty set, captures partial information on the uncertain distribution via moment and distributional shape constraints.

% Obviously, formulation \eqref{DRO_original} also applies for problems where the modeler is simply interested in understanding the worst-case performance measure subject to the imposed partial information, without necessarily a needing to make a decision per se. 

% (ZHENYUAN: ALSO MAYBE TOO CONSERVATIVE FOR FURTHER APPLICATIONS OF THE DRO OPTIMAL VALUE.)

% and thus is overly conservative. 
\subsection{Background and Existing Challenges}\label{sec:literature}
% In this regard, key ingredients of DRO are the constraints imposed on the distributions, which comprise the so-called ambiguity set or uncertainty set that aims to specify the problem in the form \eqref{DRO_original}, where the worst case is within the so-called ambiguity set or uncertainty set. The perspective of making decisions that optimize the worst-case has gathered surging attention in recent years. It advocates a minimax perspective that  is a quickly growing methodology to handle optimization under uncertainty. When facing distributional uncertainty, DRO seeks the worst-case scenario by optimizing over a class of distributions that contains the true distribution with high probability \cite{delage2010distributionally,goh2010distributionally,kuhn2019wasserstein}. Therefore, DRO provides a bound on the ground-truth value (objective under the true distribution) and can be thought of as robust optimization \cite{ben2009robust,bertsimas2011theory}. However, since its feasible region (or so-called uncertainty set) is composed of probability measures or probability densities, DRO is actually an infinite-dimensional optimization problem, making it difficult to solve. our considered problem, namely a distributional optimization problem \eqref{DRO_original} that includes both shape and moment constraints. 

Let us elaborate further and explain the challenges on the aforementioned problem that motivates this paper. First of all, the uncertainty sets in the DRO literature can be generally categorized into two major types. The first type is a neighborhood ball surrounding a baseline distribution, where the ball size is measured via a statistical distance. Common choices of distance include the class of $\phi$-divergence (\cite{ben2013robust,bertsimas2018robust,bayraksan2015data,Iyengar2005,hu2013kullback}) which also covers particular cases like the Renyi divergence (\cite{atar2015robust,dey2010entropy}) and total variation distance (\cite{jiang2018risk}), and the Wasserstein metric (\cite{esfahani2018data,blanchet2019quantifying,gao2023distributionally,xie2021distributionally,chen2018robust}). This type of DRO originates from stochastic control (\cite{petersen2000minimax}), and since then has found applications in various disciplines in economics (\cite{hansen2008robustness}), finance (\cite{glasserman2014robust}), queueing (\cite{jain2010optimality}) and simulation output analysis (\cite{Lam2018sensitivity}). More recently, it has been shown to have close connections with regularization and solution generalization behaviors (\cite{LAM2017301,duchi2021statistics,gotoh2018robust,Lam2016robust,lam2019recovering,gupta2019near,blanchet2021sample,blanchet2022confidence,gao2022finite,shafieezadeh2019regularization}), and as such gains surging popularity in machine learning and statistics. 

% (\cite{ghaoui2003worst,bertsimas2005optimal,delage2010distributionally,goh2010distributionally,wiesemann2014distributionally}) have received growing attention
The second major type of uncertainty sets, which is the focus of this paper, represents partial information on the distribution through moments and support (\cite{delage2010distributionally,bertsimas2005optimal,wiesemann2014distributionally,goh2010distributionally,ghaoui2003worst,hanasusanto2015distributionally,zhang2018ambiguous,xie2018deterministic}), marginal constraints (\cite{doan2015robustness,dhara2021worst}), and shape constraints (\cite{van2016generalized,li2019ambiguous,lam2017tail,chen2021discrete}). These uncertainty sets are especially useful in limited data situations or when data are partially unobservable due to limitations in the data collection process. Moment constraints are intuitive and easy to calibrate from either expert knowledge or data. On the other hand, shape constraints acts as an important complement to moments, one reason being that moment constraints alone can lead to unrealistic worst-case distributions, namely with sparsely discrete support and thus becoming overly conservative. Another reason is that shape constraints need no calibration and hence can be used with little or no data. Thanks to these advantages, shape constraints have been demonstrably useful in, for instance, extremal estimation where the objective function denotes a tail probability or related risk quantity (\cite{van2016generalized,lam2017tail}). In this situation, by the very definition of a tail, data are scarce and shape constraints offer plausible extrapolations that balance conservativeness and obtainable statistical information. It is worth noting that both moments and shape are important: Using moments alone can lead to unrealistically pessimistic distributions concentrated on a small handful of points (\cite{popescu2005semidefinite}), while using shape alone can lead to the density being too ``free" and does not conform to the natural boundary of the considered problem (\cite{lam2017tail}). In our main problem \eqref{DRO_original}, we thus consider the shape and the moment constraints both as integral parts of the formulation.

% Our investigation in this paper is more closely related 

% The ball sizes using these distances are calibrated from either density and entropy estimation (\cite{jiang2016data}), using goodness-of-fit statistics (\cite{ben2013robust,bertsimas2018robust}), or employing or developing nonparametric empirical likelihood theory (\cite{LAM2017301,duchi2021statistics,lam2019recovering,Blanchet2019,blanchet2022confidence}). 

% \DSQ{[DS: This also seems like a lot of extra references to me, but I didn't want to rewrite this paragraph given my unfamiliarity with the references.]}

% Next, we review the state-of-the-art methods in tackling \eqref{DRO_original}. 
To this end, various types of shape constraints have been considered in the literature, including monotonicity, unimodality and convexity (\cite{popescu2005semidefinite}) in the one-dimensional case, and $\alpha$-unimodality (with star-unimodality as a special case; \cite{dharmadhikari1988unimodality,van2016generalized,li2019ambiguous}), block unimodality (\cite{dharmadhikari1988unimodality}) and orthounimodality (\cite{lam2021orthounimodal}) in the multidimensional case. 
% \DSQ{[DS: The previous sentence implies Choquet addresses shape constraints, but the next sentence suggest moment constraints are key? Maybe this is ok.]} , similar to the representation of the points in a convex set in the Euclidean space with respect to a proper mixing distribution
In terms of methodologies to solve these shape-constrained problems, a major technique is to use the Choquet representation theorem. This theorem asserts that probability measures with a shape constraint can be represented as a mixture of more ``elementary'' distributions. More concretely, given a cumulative distribution function $F$ in a convex shape-constrained class of distributions, the Choquet representation establishes that:
% the following expression of $F$:
\begin{equation}
F(x) = \int U_z(x)d\mu(z),\label{choquet_rep}
\end{equation}
where $U_z(x)$ are the more ``elementary'' distributions (indexed by $z$) in this class and $\mu(z)$ is the mixing distribution. In this representation, $U_z(x)$ are fixed (as the extreme points of this convex class) and the mixing distribution varies as $F$ varies. Such a representation helps to reformulate the objective function and moment constraints in \eqref{DRO_original} as mixtures of moments over these elementary distributions, i.e.,
% in the following way:
\begin{equation*}
\mathbb{E}_F[\phi(X)]=\mathbb{E}_Z[\mathbb{E}_{U_Z}[\phi(X)]],
\end{equation*}
which turns the decision variable into the mixing distribution. When the inner expectation $\mathbb{E}_{U_Z}[\phi(X)]$ can be easily computed, the resulting optimization with respect to the mixing distribution (i.e., the distribution of $Z$) could be reducible to a solvable optimization problem (e.g., semidefinite program) and hence the exact optimal value can be obtained. 

 % With arbitrary choices of moments, the mixture of $\mathbb{E}_{U_Z}[\phi(X)]$ can be too complicated to be handled. One is that the resulting reduction may not necessarily lead to a solvable problem. 
While powerful, there are two main challenges of the Choquet representation that can hammer ultimate solvability. First is that the index $z$ in (\ref{choquet_rep}) can still be infinite-dimensional depending on the shape constraint (e.g., orthounimodality; \cite{lam2021orthounimodal}), and thus the optimization problem with respect to the mixing distribution is still infinite-dimensional. In this case, unless a further finite-dimensional reduction can be found, the problem can remain intractable. Another more important challenge is that the moments in the objective function and constraints must be in certain restrictive forms so that the inner expectation $\mathbb{E}_{U_Z}[\phi(X)]$ can be easily computed. Consequently, in \cite{popescu2005semidefinite} which focuses on one-dimensional shapes including convexity and unimodality, the functions $\phi_j(x)$ in the objective and moment constraints must be piecewise polynomial. In \cite{lam2021orthounimodal} that considers orthounimodality, the constraints must be expressed as probabilities but not general moments, and the objective function must be the probability of a set that is expressed as an epigraph of a function. Additionally, both \cite{van2016generalized} and \cite{li2019ambiguous} can only handle the first and second moment constraints for $\alpha$-unimodality, with the objective function in \cite{van2016generalized} restricted to the probability of the random vector falling outside a prescribed polyhedron, and the objective function in \cite{li2019ambiguous} being either the probability of the random vector lying in a hyperplane or a special function derived from the conditional value-at-risk. The left half of Table \ref{table_conditions} summarizes the requirements in this existing literature for all the shape constraints that we contribute to in this paper.
% , including all the discussed shapes in the one-dimensional case and orthomodality in the multivariate case. (POSSIBLE TO ADD ONE MORE COLUMN OR INDICATE WHICH CASE IS ONE OR HIGHER DIMENSION?)(ZHENYUAN: DONE.)

% where  and demonstrates the mild conditions in our paper. 
% Similarly, In 

% (USE TABLE 2 HERE TO EXPLAIN THE CHALLENGES IN THE LITERATURE, AND BE MORE SPECIFIC, E.G., WE SHOULD EXPLAIN THE POLYNOMIAL REQUIREMENT ETC., AND IN YOUR EXPLANATION ABOVE THERE'S NO DESCRIPTIONS ON WHAT SHAPES THE PAPERS ARE ABOUT (WE DON'T EXPECT THE READER TO GO INTO THESE PAPERS SO WE SHOULD BE CLEAR HERE..); CORRESPONDINGLY, CITE THE RELEVANT PAPERS IN THE TABLE ENTRIES IN TABLE 2; ALSO, VISUALLY, IT MIGHT BE BETTER TO PUT VERTICAL LINES BETWEEN THE MAIN COLUMNS OF SHAPE CONSTRAINTS, EXISTING LITERATURE AND OUR WORK)

% (ZHENYUAN: DONE.)

\begin{table}
\centering
\caption{Comparison between existing literature and our approach in terms of conditions on $\phi_j(x)$ in the objective and additional moment constraints. $*$ denotes one-dimensional shapes and $\dagger$ denotes multidimensional shapes.}
% \caption{Conditions on the objective function and moment constraints for tractable importance-weighted SAA}
\label{table_conditions}
\resizebox{\columnwidth}{!}{
\begin{tblr}{
  cells = {c},
  cell{1}{2} = {c=2}{},
  cell{1}{4} = {c=2}{},
  cell{3}{2} = {r=3}{},
  cell{3}{3} = {r=3}{},
  cell{3}{4} = {r=4}{},
  cell{3}{5} = {r=4}{},
  vline{2,4} = {1-6}{},
  hline{1,7} = {-}{0.08em},
  hline{2} = {2-5}{},
  hline{3} = {-}{},
  hline{6} = {1-3}{},
}
                           & Existing literature                       &                                             & IW-SAA                         &                                                       \\
Shape constraints          & Objective                                 & {Additional moment\\constraints}             & Objective                        & {Additional moment\\constraints}                       \\
monotonicity\textsuperscript{*}               & {piecewise polynomial \\ (\cite{popescu2005semidefinite})}                      & {piecewise polynomial\\\& Slater's condition \\ (\cite{popescu2005semidefinite})} & {mild integrability\\condition \\(Corollaries \ref{1D_compact}-\ref{OU_unbounded})} & {mild integrability\\condition\\\& Slater's condition \\(Corollaries \ref{1D_compact}-\ref{OU_unbounded})} \\
convexity\textsuperscript{*} (bounded domain) &                                           &                                             &                                  &                                                       \\
unimodality\textsuperscript{*}                &                                           &                                             &                                  &                                                       \\
{\\orthounimodality\textsuperscript{\textdagger}}           & {indicator of the\\epigraph of a function \\ (\cite{lam2021orthounimodal})
} & {indicator of \\hyperrectangles \\ (\cite{lam2021orthounimodal})
}      &                                  &       \\
\end{tblr}
}
\end{table}

\subsection{Contributions}\label{sec:SAA_DRO_contributions}

Motivated by the above challenges, we propose an alternative sampling-based approach to solve shape-constrained distributional optimization. Rather than using the mixture idea in the Choquet representation, our approach first conducts a \emph{change of measure} from the underlying distribution to a pre-selected sampling distribution. Then, by drawing Monte Carlo samples from this latter distribution, we formulate a suitable sample average approximation (SAA) (\cite{shapiro2021lectures}) where the decision variable is now the importance weights between the original and the sampling distributions. 

% (\cite{tokdar2010importance}) ,kloek1978bayesian,geweke1989bayesian
We call our approach importance-weighted (IW) SAA. The change of measure idea underpins the technique of importance sampling, which is especially useful as a variance reduction method (\cite{bucklew2004introduction,juneja2006rare,blanchet2012state}) and in Bayesian inference (\cite{liu2001monte}). In our setting, importance sampling is used to translate the decision variables from the underlying distribution in \eqref{DRO_original} to the importance weights, so that we can now sample from a \emph{known} distribution to conduct SAA. In this regard, the closest works to our approach are \cite{glasserman2014robust} and \cite{ghosh2019robust}, which like us consider worst-case distributional optimization problems. In particular, \cite{glasserman2014robust} considers uncertainty sets constructed as divergence-based neighborhood balls surrounding a baseline distribution, while \cite{ghosh2019robust} considers additionally moment-based uncertainty sets and essentially discrete distributions. Nonetheless, when shape constraints are present, challenges arise regarding the representability of these shapes in terms of SAA. Moreover, the SAA solution may violate the feasibility of the additional moment constraints in \eqref{DRO_original}, which adds complications in convergence analyses as it necessitates the consideration of both optimality and feasibility. These challenges motivate us to develop strong duality results for both the original and SAA counterparts in order to remove the need to check feasibility for the moment constraints. With this, we study statistical consistency and convergence rates of our IW-SAA, by leveraging tools from empirical processes (\cite{van1996weak}) that critically hinge on the functional complexities of the shape constraints. Our study on this interplay of duality and empirical process machinery to analyze shape-constrained optimization appears to be the first in the literature. It ultimately leads to solution schemes for problems that are beyond the reach of existing Choquet-based approaches; see the right half of Table \ref{table_conditions}.

 % A closely-related field to shape-constrained DRO is the The general formulation of shape-constrained nonparametric $M$-estimation is to 
 \subsection{Other Related Works}
We close this introduction by comparing our study with several lines of related works. First is shape-constrained nonparametric $M$-estimation in statistics, which concerns the optimization of  likelihood function or empirical loss function subject to shape constraints as a means to obtain estimators for statistical quantities. \cite{royset2015fusion}, \cite{pavlides2012nonparametric}, and \cite{sager1982nonparametric} and \cite{polonik1998silhouette} study density estimation under one-dimensional shape constraints, the so-called scale mixture of uniform class of distributions, and orthounimodal constraints respectively. \cite{royset2020variational} and \cite{royset2020approximations} investigate consistency and convergence rates. \cite{seijo2011nonparametric,guntuboyina2015global,mukherjee2024least} study (quasi)convex least squares estimation. \cite{fang2021multivariate} study multivariate extensions of isotonic regression and total variation denoising. Compared with this literature, our shape-constrained distributional optimization bears two key differences. First is our focus on the optimal value instead of the optimal solution, which arises from our disparate motivation from nonparametric $M$-estimation. In DRO, the optimal value of \eqref{DRO_original} is the worst-case value of the objective function, which signifies and provides an upper bound on the actual decision performance. In contrast, the optimal solution in nonparametric $M$-estimation corresponds to the statistical estimator and is thus of primary interest in statistics. The second difference is the lack of moment constraints in nonparametric $M$-estimation, compared to our formulation \eqref{DRO_original} that critically comprises both moments and shape constraints as discussed above. These differences lead us to develop new results on the interaction of duality with empirical processes and convergence guarantees for optimal values that are not studied in the nonparametric $M$-estimation literature.

Our work is also related to functional optimization that, like our problem, is infinite-dimensional in nature, but its decision variable is not necessarily a density. In contrast to deterministic calculus of variation problems (\cite{gelfand2000calculus}), some recent work studies stochastic counterparts. \cite{singham2017sample} first propose to use SAA to approximate functional optimization formulations motivated from principal-agent problems that are analytically intractable. \cite{singham2019sample} further studies bootstrapping to estimate the optimal value, though without consistency guarantees. \cite{singham2020sample} investigates bounding techniques (\cite{mak1999monte,bayraksan2006assessing}) for optimal values and theoretical convergence. Similar to our work, these studies involve SAA problems with decision variables of growing dimensions. On the other hand, they focus on essentially one-dimensional monotonic functions as decision variables, and do not include auxiliary moment constraints or distributional variables that necessitate changes of measure. As such, they do not investigate duality theory and the elaborate usage of empirical processes for general shapes. Lastly, \cite{zhou2022sample} studies the convergence of SAA in functional optimization that is driven by Gaussian processes. However, their sampling is on the underlying Gaussian variable that is different from our change of measure and subsequent developments to address shape-constrained problems.

% Finally, comprehensive materials on SAA can be found in 

The remainder of this paper is organized as follows.
% Section \ref{sec:formulation} presents the formulation of the DRO problem with shape and moment constraints. 
Section \ref{sec:SAA main} gives an overview of our IW-SAA approach and theoretical guarantees. Section \ref{sec:example} specializes our main guarantees to various shape constraints, and discusses our scope of applicability and limitations. Section \ref{sec:SAA_DRO_complexity} discusses computational challenge and remedies associated with the number of constraints. Section \ref{sec:num} presents numerical results to validate our theorems and demonstrates our empirical performances. The Appendix contains additional technical details and proofs for all results in this paper.
\section{IW-SAA and General Theoretical Guarantees\label{sec:SAA main}}

 % (WOULDN'T THIS REQUIREMENT BE STATED IN THE ORIGINAL FORMULATION \eqref{DRO_original} ALREADY?)(ZHENYUAN: I MOVED THE EXPLANATION ON $I(x\in\mathcal{X})$ TO \eqref{DRO_original}) 
We present our IW-SAA to solve (\ref{DRO_original}). First, we select a suitable known and simulatable (i.e., allows the generation of random copies) positive probability density $g$ on $\mathcal{X}$, and rewrite (\ref{DRO_original}) via a change of measure from $f$ to $g$ as%
\begin{align}
(\mathcal{P}):\quad\quad\quad\sup_{L}\text{ }  &  \mathbb{E}_{g}[\phi_{0}(X)L(X)]\nonumber\\
\text{subject to }  &  \mathbb{E}_{g}[\phi_{j}(X)L(X)]\leq\mu_{j},j=1,\ldots,m\label{DRO_g}\\
&  \mathbb{E}_{g}[\phi_{j}(X)L(X)]=\mu_{j},j=m+1,\ldots,l\nonumber\\
&  L\in\mathcal{L}\nonumber
\end{align}
where $L=f/g$ is the likelihood ratio and $\mathcal{L}=\{f/g:f\in\mathcal{F}\}$ contains the corresponding shape constraints for $L$. We require the positivity of $g$ to ensure the existence of the likelihood ratio. Now, since the probability density $g$ is known and simulatable, we can draw $n$ i.i.d.\ samples $X_{1},\ldots,X_{n}$ from $g$ and formulate the SAA counterpart of ($\mathcal{P}$) as%
\begin{align}
(\mathcal{P}_{n}):\quad\quad\quad\sup_{L}\text{ } &  \frac{1}{n}\sum_{i=1}^{n}\phi_{0}(X_{i})L(X_{i})\nonumber\\
\text{subject to } &  \frac{1}{n}\sum_{i=1}^{n}\phi_{j}(X_{i})L(X_{i})\leq\mu_{j},j=1,\ldots,m\label{SAA_DRO}\\
&  \frac{1}{n}\sum_{i=1}^{n}\phi_{j}(X_{i})L(X_{i})=\mu_{j},j=m+1,\ldots,l\nonumber\\
&  L\in\mathcal{L}\nonumber
\end{align}

Note that the SAA problem ($\mathcal{P}_{n}$) is potentially infinite-dimensional, since its decision space $\mathcal L$ is the same as that of problem ($\mathcal{P}$). To obtain a finite-dimensional reduction of ($\mathcal{P}_{n}$), we need to discretize the shape constraint $L\in\mathcal{L}$ in terms of samples $X_1,\ldots,X_n$ that leads to what we call \emph{finite-dimensional reducibility}. To explain, consider the example where $\mathcal{F}$ is the class of monotonically increasing functions. The space $\mathcal{L}$ can be written as $\mathcal{L}=\{L:L(x_1)g(x_1)\le L(x_2)g(x_2)\text{ for }x_1\le x_2\}$. A corresponding discretized version, denoted by $\mathcal{L}_n$, can be written as $\mathcal{L}_n=\{(L(X_1), \ldots,L(X_n)):L(X_i)g(X_i)\le L(X_j)g(X_j)\text{ for }X_i\le X_j\}$ which is a subset of $\mathbb{R}^n$. This discretized version is invertible in the sense that, given any discrete values $(L_1, \ldots,L_n)\in \mathcal{L}_n$, we can recover a function $L\in\mathcal{L}$ such that $L(X_i)=L_i$ for all $i$ by interpolation:
\[
L(x)=\left\{
\begin{array}[c]{l}%
L_{i_1}g(X_{i_1})/g(x),x\in(-\infty,X_{i_1}]\\
L_{i_2}g(X_{i_2})/g(x),x\in(X_{i_1},X_{i_2}]\\
\cdots\\
L_{i_n}g(X_{i_n})/g(x),x\in(X_{i_{n-1}},\infty)
\end{array}
\right.,
\]
where $(i_1,\ldots,i_n)$ is a permutation of $(1,\ldots,n)$ such that $X_{i_1}\leq\cdots\leq X_{i_n}$. Consequently, we can replace $L\in\mathcal{L}$ by $(L(X_1), \ldots,L(X_n))\in \mathcal{L}_n$ in ($\mathcal{P}_{n}$) without changing its optimal value, and the resulting optimization problem is a finite-dimensional linear program with decision variables $(L(X_1), \ldots,L(X_n))$. 
In general, we say the SAA problem ($\mathcal{P}_{n}$) has a \emph{finite-dimensional reduction} if the discretization of $L\in\mathcal{L}$  like above is invertible, in which case the decision variable can be chosen as $(L(X_1),\ldots,L(X_n))$ without loss of generality. The left half of Table \ref{checkbox} displays, for various shape constraints, whether ($\mathcal{P}_{n}$) has a finite-dimensional reduction. In particular, the four constraints listed in Table \ref{table_conditions} are finite-dimensionally reducible and result in readily solvable linear programs. Details of these reformulations will be presented in Section \ref{sec:example}.
 % and whether the SAA optimal value has statistical guarantees for estimating the true DRO optimal value
% in $L(X_{i}),i=1,\ldots,n$.
% and thus can be solved easily. Details of the reformulation will be presented in Section \ref{sec:example}.
%relies on whether there is a tractable equivalent reformulation of the shape constraint $L\in\mathcal{L}$ using only the discrete values $X_{i},i=1,\ldots,n$. \DSQ{[DS:What exactly do you mean by tractable? \red{ZL will describe the two ways here.} Also discuss how interpolation works as related to previous paragraph.]} \DS{For example, it is easy to impose monotonically increasing densities in the SAA problem by requiring that $L(X_i)g(X_i)\le L(X_j)g(X_j)$ for $X_i\le X_j$.} , and thus is solvable using off-the-shelf algorithms

\begin{table}
\centering
\caption{Finite-dimensional reducibility and statistical guarantees of IW-SAA. $\checkmark$ means it has this property. $\protect\crossmark$ means there is a counterexample showing that it does not have this property. Blank means this is an open question.}
\label{checkbox}
\resizebox{\columnwidth}{!}{
\begin{tabular}{cccc}
\toprule
\textbf{Shape constraints}              & \textbf{Finite-dimensional reducibility} & \textbf{Consistency} & \textbf{Canonical convergence rate}  \\
\midrule
monotonicity                 & $\checkmark$           & $\checkmark$         & $\checkmark$                         \\
convexity (bounded domain)   & $\checkmark$           & $\checkmark$         & $\checkmark$                         \\
convexity (unbounded domain) &                        & $\checkmark$         & $\checkmark$                         \\
unimodality                  & $\checkmark$           & $\checkmark$         & $\checkmark$                         \\
$\alpha$-unimodality                     &                        & $\crossmark$                     & $\crossmark$                                      \\
block unimodality            &                        & $\checkmark$         &                                      \\
orthounimodality             & $\checkmark$           & $\checkmark$         &                                      \\
\bottomrule
\end{tabular}
}
\end{table}

In the next two subsections, we present our main theoretical results on IW-SAA, including strong duality (Section \ref{sec:strong duality}) and statistical guarantees (Section \ref{sec:guarantee}).

\subsection{Strong Duality}\label{sec:strong duality}

We prove strong duality to transform ($\mathcal{P}$) and ($\mathcal{P}_{n}$) into unconstrained optimization problems. The Lagrangian dual problems of ($\mathcal{P}$) and ($\mathcal{P}_{n}$) are given by%
\begin{equation}
(\mathcal{D}):\inf_{\lambda\in\mathbb{R}_{+}^{m}\times\mathbb{R}^{l-m}}\sup_{L\in\mathcal{L}}\text{ }\left\{  \mathbb{E}_{g}[\phi_{0}(X)L(X)]-\sum_{j=1}^{l}\lambda_{j}(\mathbb{E}_{g}[\phi_{j}(X)L(X)]-\mu_{j})\right\}  ,
\label{DRO_Lagrange}%
\end{equation}%
\begin{equation}
(\mathcal{D}_{n}):\inf_{\lambda\in\mathbb{R}_{+}^{m}\times\mathbb{R}^{l-m}}\sup_{L\in\mathcal{L}}\text{ }\frac{1}{n}\sum_{i=1}^{n}\left(  \phi_{0}(X_{i})L(X_{i})-\sum_{j=1}^{l}\lambda_{j}(\phi_{j}(X_{i})L(X_{i})-\mu_{j})\right)  . \label{SAA_DRO_Lagrange}%
\end{equation}
The following assumptions are needed to establish strong duality.

\begin{assumption}
\label{integrability_condition}For any $L\in\mathcal{L}$, the random variables $\phi_{j}(X)L(X),j=0,\ldots,l$ are integrable with respect to $g$.
\end{assumption}

\begin{assumption}
\label{convex_feasible_set}$\mathcal{L}$ is a convex set.
\end{assumption}

\begin{assumption}
\label{interior_point}There exists a feasible function $L_{0}\in\mathcal{L}$ s.t.
\[
\mathbb{E}_{g}[\phi_{j}(X)L_{0}(X)]<\mu_{j},j=1,\ldots,m,\quad \mathbb{E}_{g}[\phi_{j}(X)L_{0}(X)]=\mu_{j},j=m+1,\ldots,l.
\]
Additionally, $(\mu_{m+1},\ldots,\mu_{l})$ is an interior point of the following set%
\[
\{(\mathbb{E}_{g}[\phi_{m+1}(X)L(X)],\ldots,\mathbb{E}_{g}[\phi_{l}(X)L(X)]):L\in\mathcal{L\}.}%
\]

\end{assumption}

 % (ZHENYUAN: THESE REFERENCES ARE FROM Mottet C, Lam H (2017) On optimization over tail distributions)
Assumption \ref{integrability_condition} ensures the dual problem is well-defined. Assumption \ref{convex_feasible_set} guarantees the optimization problem is convex, which is the usual assumption when establishing strong duality. Assumption \ref{interior_point} is a Slater condition for the distributional optimization problem, similar to the conditions in, e.g., \cite{karlin1966tchebycheff,smith1995generalized,shapiro2001duality,bertsimas2005optimal,popescu2005semidefinite}. Note that whether Assumptions \ref{integrability_condition}-\ref{interior_point} hold or not depends only on $\phi_{j}(X)$ and $\mathcal{F}$ but not the choice of sampling distribution $g$. 
% \QR{We state them in the current form because they are easy to work with in the following.}

% Under these assumptions, 
Under these assumptions, we develop strong Lagrangian duality in both the distributional optimization problem and its SAA counterpart. 
We use $val(\cdot)$ to denote the optimal value of optimization problems. Since the SAA optimal value is the supremum of possibly an uncountable number of random variables, it may not be measurable. As such we employ outer (with a superscript $*$) and inner (with a subscript $*$) probabilities and expectations to handle the measurability issue. Stochastic convergence is understood under the outer expectation. We delegate these technical measurability details to Appendix \ref{sec:remedy}. 
% (DISCUSS NOTATION * BELOW?)(ZHENYUAN: DEFINED $*$ HERE)
% The following establishes strong duality for the original distributional optimization problem.

\begin{theorem}[Strong duality for $\mathcal{P}$]
\label{strong_duality}If Assumption \ref{integrability_condition} holds, then weak duality holds, i.e., $val(\mathcal{P})\leq val(\mathcal{D})$. Furthermore, if Assumptions \ref{convex_feasible_set} and \ref{interior_point} also hold, then strong duality holds, i.e., $val(\mathcal{P})=val(\mathcal{D})$.
\end{theorem}

% Next we show strong duality of importance-weighted SAA.

\begin{theorem}[Strong duality for $\mathcal{P}_n$]
\label{strong_duality_SAA}For any sample size $n$, weak duality for IW-SAA holds, i.e., $val(\mathcal{P}_{n})\leq val(\mathcal{D}_{n})$ for any realization of $X_{1},\ldots,X_{n}$. If Assumptions \ref{integrability_condition}-\ref{interior_point} hold, then strong duality holds at all but finitely often $n$, i.e.,
\begin{align*}
\mathbb{P}_{\ast}\left(\bigcup_{n=1}^{\infty}\bigcap_{k=n}^{\infty}\{val(\mathcal{P}_{k})=val(\mathcal{D}_{k})\}\right)  =\mathbb{P}_{\ast}(\exists N_{0}\in\mathbb{N}\text{ s.t. }val(\mathcal{P}_{n})=val(\mathcal{D}_{n}),\forall n\geq N_{0})=1
\end{align*}
as well as asymptotically, i.e., as $n\to\infty$, $\mathbb{P}_{\ast}(val(\mathcal{P}_{n})=val(\mathcal{D}_{n}))\rightarrow1.$

\end{theorem}

% We use inner probability to address the measurability issue, and from Remark \ref{ref:remark1} the two results in Theorem \ref{strong_duality_SAA} must be proved separately ((\ref{eq:SD2}) is not implied by (\ref{eq:SD1})).

\subsection{Statistical Guarantees}\label{sec:guarantee}
We prove the consistency and canonical convergence rate of IW-SAA.
% under general conditions. 
We first introduce some notation. We write $P_{g}$ as the probability measure induced by $g$ and define the corresponding empirical distribution $P_{n}$ based on $n$ observations from $g$ as $P_{n}=(1/n)\sum_{i=1}^{n}\delta_{X_{i}}$, where $\delta_{X_{i}}$ is the point mass at $X_i$. 
% \DSQ{[I changed the function from $f(X)$ to $\phi(X)$ to align with our formulation where $f$ is a specific decision variable density.]} \DS{ in empirical process theory
Following the convention in empirical processes, the expectations of $h(X)$ under $P_{g}$ and $P_{n}$ are written as $P_{g}h:=\mathbb{E}_{g}[h(X)]$ and $P_{n}h:=(1/n)\sum_{i=1}^{n}h(X_{i})$. An envelope of a function class $\mathcal{F}$ is any function $F(x)$ such that $\sup_{h\in\mathcal{F}}|h(x)|\leq F(x)$. We write $||Q||_{\mathcal{F}}:=\sup\{|Qh|:h\in\mathcal{F\}}$ for any signed measure $Q$ and function class $\mathcal{F}$. We define the empirical process $G_{n}$ as $G_{n}=\sqrt{n}(P_{n}-P_{{g}})$.
% Furthermore, a uniform law of large numbers holds for $\mathcal{F}$ if $||P_n-P_f||_{\mathcal{F}}:=\sup_{\phi\in\mathcal{F}}|P_n\phi-P_f\phi|\stackrel{P}{\to}0.$

Now we state our assumptions. We begin with the finiteness of the optimal value.

\begin{assumption}
% [{\textbf{Finite optimal}}]
\label{finite_opt}$val(\mathcal{P})$ is finite.
\end{assumption}

The next assumption states that the function in the objective and constraints in ($\mathcal{P}$) are so-called $P_g$\textit{-Glivenko-Cantelli} ($P_{g}$-GC), i.e., they satisfy the uniform law of large numbers. 

\begin{assumption}
\label{GC_class}The function classes $\mathcal{F}_{j}:=\{\phi_{j}(x)L(x):L\in\mathcal{L}\},j=0,\ldots,l$ have integrable envelopes and are $P_{g}$-GC, i.e., $||P_{n}-P_{g}||_{\mathcal{F}_{j}}\overset{\text{a.s.*}}{\rightarrow}0$, for any $j=0,\ldots,l$.
\end{assumption}

With these, we have the consistency of IW-SAA.
% We next present a theorem establishing consistency of importance-weighted SAA.  The proof is based on strong duality and the uniform law of large numbers.

\begin{theorem}[Consistency]
\label{consistency}Suppose Assumptions \ref{integrability_condition}-\ref{GC_class} hold. Then we have $|val(\mathcal{P}_{n})-val(\mathcal{P})|\overset{\text{a.s.*}}{\rightarrow}0.$
\end{theorem}

% \subsection{Convergence Rate}\label{sec:canonical}

Establishing canonical convergence rates requires an additional maximal inequality holds for $G_{n}$, which helps bound the probability $\mathbb{P}^{\ast}(|\sqrt{n}(val(\mathcal{P}_{n})-val(\mathcal{P}))|\geq K)$.
% We prove $val(\mathcal{P}_{n})-val(\mathcal{P})$ has the canonical $\sqrt{n}$ convergence rate under an additional assumption.  We assume the following maximal inequality holds for $G_{n}$.

\begin{assumption}
\label{maximal_inequality}$\exists$ a constant $C>0$ s.t. $\mathbb{E}^{\ast}||G_{n}||_{\mathcal{F}_{j}}\leq C,\forall j=0,\ldots,l,\forall n\in\mathbb{N}.$
\end{assumption}

% The maximal inequality helps to bound the probability $\mathbb{P}^{\ast}(|\sqrt{n}(val(\mathcal{P}_{n})-val(\mathcal{P}))|\geq K)$ for large $K$, which gives the $\sqrt{n}$ convergence rate:

\begin{theorem}[Convergence rate]
\label{convergence_rate}Suppose Assumptions \ref{integrability_condition}-\ref{maximal_inequality} hold. Then we have $\sqrt{n}(val(\mathcal{P}_{n})-val(\mathcal{P}))=O_{\mathbb{P}^{\ast}}(1)$, i.e., $\forall\varepsilon>0$, $\exists~K>0$ s.t.%
\[
\limsup_{n\rightarrow\infty}\mathbb{P}^{\ast}(|\sqrt{n}(val(\mathcal{P}_{n})-val(\mathcal{P}))|\geq K)\leq\varepsilon.
\]

\end{theorem}

 % and leads to Theorem \ref{consistency} and Theorem \ref{convergence_rate}  via the strong duality and triangular inequality With a focus on the optimal values and the need to handle the additional moment constraints via strong duality, 

% WHEN DISCUSSING THE ABOVE THEOREMS, ADD DISCUSSION ON WHY STRONG DUALITY IS NEEDED. FOR THIS, WE CAN DISCUSS HOW NOT USING DUALITY CAUSES CHALLENGES. A CONCRETE DISCUSSION ON THIS POINT IS IMPORTANT BUT CURRENTLY MISSING IN THE PAPER. 

% ZHENYUAN: REVISED THE FOLLOWING PARAGRAPH.

%In deriving Theorems \ref{consistency} and \ref{convergence_rate}, we leverage tools from empirical processes to bound the uniform discrepancy between expectations and their sample averages, and strong duality that translates it into the error $val(\mathcal{P}_{n})-val(\mathcal{P})$. In our analyses, shape constraints are critical in reducing the complexity of the feasible decision space and rendering Assumption \ref{GC_class} and \ref{maximal_inequality} verifiable. 

In deriving Theorems \ref{consistency} and \ref{convergence_rate}, we leverage tools from empirical processes to bound the uniform discrepancy between expectations and their sample averages over a class of functions signified by the likelihood ratios. This class of functions are defined by the feasible region specified by both the moment and shape constraints. However, an issue arises that the moment constraints in ($\mathcal{P}$) and ($\mathcal{P}_n$) may not be the same due to the discretization in ($\mathcal{P}_n$), and in fact the moment constraints in ($\mathcal{P}_n$) can even change with the sample size $n$. Thus, a likelihood ratio that satisfies the moment constraints in ($\mathcal{P}$) may not satisfy the ones in ($\mathcal{P}_n$) and vice versa. To this end, strong duality helps move the moment constraints to the objective function and thus unifies the feasible region to be $L\in\mathcal{L}$ in both ($\mathcal{D}$) and ($\mathcal{D}_n$), which makes the empirical process theory applicable to our problem. With this remedy, the shape constraints are instrumental in reducing the complexity of the feasible decision space and rendering Assumptions \ref{GC_class} and \ref{maximal_inequality} verifiable. 

In the next section, we specialize our general results obtained in this section to various shape constraints. In particular, we demonstrate how we are able to handle objectives and moment constraints beyond the existing Choquet-based technique, which is a key implication in this paper.

\section{Applications to Various Shape Constraints\label{sec:example}}

% In this section, we introduce several shape constraints widely-used in the literature and discuss whether they can be tractably solved by importance-weighted SAA and whether 
We apply Theorems \ref{consistency} and \ref{convergence_rate} to the shape constraints in Table \ref{table_conditions}, and also discuss the limitations of our approach to other remaining constraints in Table \ref{checkbox}.
% with the above statistical guarantees. 
% In the next section, we will apply to several common shape constraints and provide explicit settings where the assumptions needed in these theorems can be verified.
First of all, note that our shape constraints are imposed on the unknown density $f\in\mathcal{F}$ but Assumptions \ref{GC_class} and \ref{maximal_inequality} are imposed on the function classes $\mathcal{F}_{j},j=0,\ldots,l$. So we need a lemma that can verify Assumptions \ref{GC_class} and \ref{maximal_inequality} by the conditions on $\mathcal{F}$. We first introduce some notation (from \cite{van1996weak} Section 2.1.1). Let $(\mathcal{F},||\cdot||_p)$ be a normed function space, where $||\cdot||_p$ is the $L_{p}$-norm, i.e., $||f||_{p}=(P_{g}|f^{p}|)^{1/p}$. For two functions $l\leq u$ (not necessarily in $\mathcal{F}$), the bracket $[l,u]$ is the class of functions $f$ satisfying $l\leq f\leq u$. The bracket is called an $\varepsilon$-bracket if $||l||_p<\infty,||u||_p<\infty$ and $||u-l||_p<\varepsilon$. We define the bracketing number $N_{[~]}(\varepsilon,\mathcal{F},||\cdot||_p)$ as the minimum number of $\varepsilon$-brackets to cover $\mathcal{F}$. For $\gamma>0$, we define $\mathcal{F}^{\gamma}=\{f^{\gamma}(x):f\in\mathcal{F}\}$.

\begin{lemma}
\label{general_verification}Let $F(x)$ be an envelope of the function class $\mathcal{F}$ with $||F||_{\infty}:=\sup_{x\in\mathcal{X}}F(x)<\infty$. If for each $\phi_{j}(x),j=0,\ldots,l$, either (\ref{compact_applicable1}) or (\ref{unbounded_applicable1}) holds:%
\begin{equation}
\exists~\delta>0\text{ s.t. }N_{[~]}(\varepsilon,\mathcal{F},||\cdot||_{1+1/\delta})<\infty ~\forall\varepsilon>0,\text{and }||\phi_{j}/g||_{1+\delta}<\infty, \label{compact_applicable1}%
\end{equation}%
\begin{equation}
\exists~\gamma\in(0,1),\delta>0,M>0\text{ s.t. }\left\{
\begin{array}[c]{l}%
N_{[~]}(\varepsilon,\mathcal{F}^{\gamma},||\cdot||_{1+1/\delta})<\infty,\forall\varepsilon>0\\
F^{1-\gamma}(x)\leq Mg(x)~\forall x\in\mathcal{X},\text{and }||\phi_{j}||_{1+\delta}<\infty
\end{array}
\right.  , \label{unbounded_applicable1}%
\end{equation}
then Assumption \ref{GC_class} holds. If for each $\phi_{j}(x),j=0,\ldots,l$, either (\ref{compact_applicable2}) or (\ref{unbounded_applicable2}) holds:%
\begin{equation}
\exists~\delta>0\text{ s.t. }\int_{0}^{x}\sqrt{\log N_{[~]}(\varepsilon,\mathcal{F},||\cdot||_{2(2+\delta)/\delta})}d\varepsilon<\infty~\forall x\in\mathbb{R}_{+},\text{and }||\phi_{j}/g||_{2+\delta}<\infty,
\label{compact_applicable2}%
\end{equation}
\begin{equation}
\exists~\gamma\in(0,1),\delta>0,M>0\text{ s.t. }\left\{
\begin{array}[c]{l}%
\int_{0}^{x}\sqrt{\log N_{[~]}(\varepsilon,\mathcal{F}^{\gamma},||\cdot||_{2(2+\delta)/\delta})}d\varepsilon<\infty~\forall x\in\mathbb{R}_{+}\\
F^{1-\gamma}(x)\leq Mg(x)~\forall x\in\mathcal{X},\text{and }||\phi_{j}||_{2+\delta}<\infty
\end{array}
\right.  , \label{unbounded_applicable2}%
\end{equation}
then Assumption \ref{maximal_inequality} holds.
\end{lemma}

% , which can be viewed as importance sampling
Lemma \ref{general_verification} provides two types of conditions to verify Assumptions \ref{GC_class} and \ref{maximal_inequality}. One is based on the moment conditions on $\phi_{j}/g$, i.e., (\ref{compact_applicable1}) and (\ref{compact_applicable2}). Recall that in our formulation one of the functions $\phi_{j}$ must be the indicator function $I(x\in\mathcal{X})$ to represent the total mass constraint on $f$. Therefore, (\ref{compact_applicable1}) and (\ref{compact_applicable2}) can hold only when $\mathcal{X}$ is bounded. The other type, i.e., (\ref{unbounded_applicable1}) and (\ref{unbounded_applicable2}), is based on the moment conditions on $\phi_{j}$ but with an additional condition $F^{1-\gamma}(x)\leq Mg(x)~\forall x\in\mathcal{X}$. This type of condition can be applied when $\mathcal{X}$ is unbounded and the additional condition essentially says the sampling distribution $g$ should be heavier-tailed than any $f\in\mathcal{F}$. A common situation (also used in the following examples) of the second type is that the density $f\in\mathcal{F}$ is known to be lighter-tailed than some density $g_{0}$, i.e., $f^{1-\gamma}\leq Mg_{0}$ for any $f\in\mathcal{F}$ and some constant $M>0$. Then the envelope $F$ can be chosen as $F=(Mg_{0})^{1/(1-\gamma)}$ with the sampling distribution $g=g_{0}$. In data-driven settings, this can be conducted by estimating the tail index from the data and choosing $g_{0}$ that is heavier-tailed than the true distribution with high confidence.

% then construct the ambiguity set $\mathcal{F}$ with the envelope $F=(Mg_{0})^{1/(1-\gamma)}$ for some and discuss whether the corresponding distributional optimization problems allow a finite-dimensionally reducible importance-weighted SAA with statistical guarantees. 

Now we are ready to study the six shape constraints in Table \ref{checkbox}. In particular, all finite-dimensional reductions, if available, will be in the following form:
\begin{align*}
(\mathcal{P}^{\prime}_{n}):\sup_{L(X_i),i=1,\ldots,n}\text{ }  &  \frac{1}{n}\sum_{i=1}^{n}\phi_{0}(X_{i})L(X_{i})\\
\text{\rm subject to }  &  \frac{1}{n}\sum_{i=1}^{n}\phi_{j}(X_{i})L(X_{i})\leq\mu_{j},j=1,\ldots,m\\
&  \frac{1}{n}\sum_{i=1}^{n}\phi_{j}(X_{i})L(X_{i})=\mu_{j},j=m+1,\ldots,l\\
&  L(X_i),i=1,\ldots,n\text{ satisfy the discrete shape constraint.}
\end{align*}
To avoid repetition, we will only specify the ``discrete shape constraint" in the following subsections. As discussed above, we will distinguish between bounded domains and unbounded domains since they require different conditions.

\subsection{Compact Domain}\label{sec:compact}

We first consider three one-dimensional shapes: monotonicity, convexity and unimodality (a function $f$ is said to be unimodal about a mode ${c}$ if $f(x)$ is non-decreasing when $x\leq {c}$ and $f(x)$ is non-increasing when $x\geq {c}$). More precisely, letting $X_{(1)}\leq\cdots\leq X_{(n)}$ be the order statistics of $X_{1},\ldots,X_{n}$, we have the following:

\begin{corollary}[\textbf{One-dimensional shapes}]\label{1D_compact}Let $\mathcal{X}=[a,b]$, $M>0$ and the shape-constrained function class be one of the following:
\[
\left.
\begin{array}[c]{l}%
(\text{Monotonicity})\\
(\text{Convexity})\\
(\text{Unimodality})
\end{array}
\right.  \left.
\begin{array}[c]{l}%
\mathcal{F}=\{f:0\leq f(x)\leq M~\forall x\in\mathcal{X}\text{, }f\text{ is non-increasing on }\mathcal{X}\}\\
\mathcal{F}=\{f:0\leq f(x)\leq M~\forall x\in\mathcal{X}\text{, }f\text{ is convex on }\mathcal{X}\}\\
\mathcal{F}=\{f:0\leq f(x)\leq M~\forall x\in\mathcal{X}\text{, }f\text{ is unimodal about }c\text{ on }\mathcal{X}\}
\end{array}
\right.,
\]
where $c\in(a,b)$ is a fixed point. Suppose the sampling distribution $g$ satisfies $\inf_{x\in\mathcal{X}}g(x)>0$, the functions $\phi_{j}$ satisfy $\int_{\mathcal{X}}|\phi_{j}(x)|^{2+\delta}dx<\infty,j=0,\ldots,l$ for some $\delta>0$ and Assumption \ref{interior_point} holds. Then the IW-SAA problem $(\mathcal{P}_{n})$ with one of the above $\mathcal{F}$ is equivalent to the linear program $(\mathcal{P}^{\prime}_{n})$ with the corresponding discrete shape constraint:
\[
\left\{
\begin{array}[c]{l}
L(X_{(i+1)})g(X_{(i+1)})\leq L(X_{(i)})g(X_{(i)}),i=1,\ldots,n-1\\
0\leq L(X_{i})g(X_{i})\leq M,i=1,\ldots,n
\end{array}
\right.,
\]
\[
\left\{
\begin{array}[c]{l}
\frac{L(X_{(i+1)})g(X_{(i+1)})-L(X_{(i)})g(X_{(i)})}{X_{(i+1)}-X_{(i)}}\leq\frac{L(X_{(i+2)})g(X_{(i+2)})-L(X_{(i+1)})g(X_{(i+1)})}{X_{(i+2)}-X_{(i+1)}},i=0,\ldots,n-1\\
0\leq L(X_{i})g(X_{i})\leq M,i=0,1,\ldots,n+1
\end{array}
\right.,
\]
\[
\left\{
\begin{array}[c]{l}
L(X_{(i-1)})g(X_{(i-1)})\leq L(X_{(i)})g(X_{(i)}),\text{ if }X_{(i)}\leq c\\
L(X_{(i+1)})g(X_{(i+1)})\leq L(X_{(i)})g(X_{(i)}),\text{ if }X_{(i)}\geq c\\
0\leq L(X_{i})g(X_{i})\leq M,i=1,\ldots,n
\end{array}
\right.,
\]
i.e., $(\mathcal{P}_{n})$ is finite-dimensionally reducible. Additionally, $|val(\mathcal{P}_{n})-val(\mathcal{P})|\overset{\text{a.s.*}}{\rightarrow}0$ and $\sqrt{n}(val(\mathcal{P}_{n})-val(\mathcal{P}))=O_{\mathbb{P}^{\ast}}(1)$.
\end{corollary}

% \begin{corollary}
% [\textbf{Monotonicity}]\label{monotone_compact}Let $\mathcal{X}=[a,b]$, $M>0$ and the shape-constrained function class $\mathcal{F}=\{f:0\leq f(x)\leq M~\forall x\in\mathcal{X}\text{, }f\text{ is non-increasing on }\mathcal{X}\}$. Suppose the sampling distribution $g$ satisfies $\inf_{x\in\mathcal{X}}g(x)>0$, the functions $\phi_{j}$ satisfy $\int_{\mathcal{X}}|\phi_{j}(x)|^{2+\delta}dx<\infty,j=0,\ldots,l$ for some $\delta>0$ and Assumption \ref{interior_point} holds. Then the importance-weighted SAA problem $(\mathcal{P}_{n})$ is equivalent to the linear program $(\mathcal{P}^{\prime}_{n})$ with the following discrete shape constraint:
% \[
% \left.
% \begin{array}[c]{l}
% L(X_{(i+1)})g(X_{(i+1)})\leq L(X_{(i)})g(X_{(i)}),i=1,\ldots,n-1\\
% 0\leq L(X_{i})g(X_{i})\leq M,i=1,\ldots,n
% \end{array}
% \right..
% \]
% % \begin{align*}
% % \sup\text{ }  &  \frac{1}{n}\sum_{i=1}^{n}\phi_{0}(X_{i})L(X_{i})\\
% % \text{\rm subject to }  &  \frac{1}{n}\sum_{i=1}^{n}\phi_{j}(X_{i})L(X_{i})\leq\mu_{j},j=1,\ldots,m\\
% % &  \frac{1}{n}\sum_{i=1}^{n}\phi_{j}(X_{i})L(X_{i})=\mu_{j},j=m+1,\ldots,l\\
% % &  L(X_{(i+1)})g(X_{(i+1)})\leq L(X_{(i)})g(X_{(i)}),i=1,\ldots,n-1\\
% % &  0\leq L(X_{i})g(X_{i})\leq M,i=1,\ldots,n.
% % \end{align*}
% Additionally, $|val(\mathcal{P}_{n})-val(\mathcal{P})|\overset{\text{a.s.*}}{\rightarrow}0$ and $\sqrt{n}(val(\mathcal{P}_{n})-val(\mathcal{P}))=O_{\mathbb{P}^{\ast}}(1)$.
% \end{corollary} could be conservative because we would not like to see

Corollary \ref{1D_compact} concludes that under our assumptions, the three one-dimensional shapes elicit IW-SAAs that are finite-dimensionally reducible and exhibit consistency and $\sqrt{n}$-convergence. Corollary \ref{1D_compact}
can be also applied to the non-decreasing shape constraint and concave shape constraint with obvious modifications. For unimodality, the mode $c$ must be pre-specified. Otherwise the function spaces $\mathcal{F}$ and $\mathcal{L}$ are no longer convex, in which case Assumption \ref{convex_feasible_set} fails.

Note that the shape constraint $\mathcal{F}$ in Corollary \ref{1D_compact} contains an upper bound $M$. The impact of this bound can be divided into two cases. First is that the optimization problem without this boundedness constraint admits a bounded density as an optimal solution. In this case, selecting a large enough $M$ as the upper bound retains the optimal solution and optimal value. Another possibility is that either the problem without this boundedness constraint has no optimal solution or any optimal solution must be unbounded. In this situation, including an upper bound $M$ affects the optimal value. However, this situation itself indicates that the starting formulation can be problematic, as the worst-case scenario is achieved by an unbounded density that is typically unrealistic in practice. This discussion also carries over to the subsequent shape constraints that we study.
Next we discuss orthounimodality, which is a generalization of unimodality to the multidimensional case. A function $f$ on $\mathbb{R}^{d}$ is said to be orthounimodal about the mode $a=(a_1,\ldots,a_d)\in\mathbb{R}^{d}$, if for each $i=1,\ldots,d$ and any fixed $x_j\in\mathbb{R},j\neq i$, the function $x_i\mapsto f(x)$ is non-decreasing on $(-\infty,a_i]$ and non-increasing on $[a_i,\infty)$. Orthounimodality is recently studied in \cite{lam2021orthounimodal} as a multidimensional shape constraint well-suited to model tail distributional behaviors and thus is attractive for extreme event analysis. For two vectors $a,b\in\mathbb{R}^{d}$, $a\leq b$ and $a<b$ are interpreted in the component-wise sense and $[a,b]$ denotes the hyperrectangle $\{x\in\mathbb{R}^{d}:a\leq x\leq b\}$ (similar for $(a,b]$, $[a,b)$ and $(a,b)$). We show that orthounimodality entails a consistent IW-SAA. 
% (ZHENYUAN: NOW I DEFINE OU ON $\mathbb{R}^{d}$ NOT JUST ON $[a,\infty)$ AND DISCUSS THERE IS NO LOSS OF GENERALITY BELOW THIS COROLLARY.)
%For two vectors $a,b\in\mathbb{R}^{d}$, $a\leq b$ and $a<b$ are interpreted in the component-wise sense and $[a,b]$ denotes the hyperrectangle $\{x\in\mathbb{R}^{d}:a\leq x\leq b\}$ (similar for $(a,b]$, $[a,b)$ and $(a,b)$). A function $f$ is said to be orthounimodal about the mode $a\in\mathbb{R}^{d}$ on $[a,\infty)$, if $f(x)$ is non-increasing when any component of $x$ increases, i.e., $f(x)\geq f(x^{\prime})$ for any $a\leq x\leq x^{\prime}$. Orthounimodality is recently studied in \cite{lam2021orthounimodal} as a multidimensional shape constraint that is well-suited to model the tail behavior of densities and thus can be used for extreme event analysis. We show orthounimodality allows a finite-dimensionally reducible importance-weighted SAA with consistency.

\begin{corollary}
[\textbf{Orthounimodality}]\label{OU_compact}Let $\mathcal{X}=[a,b]\subset\mathbb{R}^{d}$ ($d\ge2$) be a hyperrectangle, $M>0$ and the shape-constrained function class $\mathcal{F}=\{f:0\leq f(x)\leq M~\forall x\in\mathcal{X}\text{, }f\text{ is orthounimodal about }a\text{ on }\mathcal{X}\}$. Suppose the sampling distribution $g$ satisfies $0<\inf_{x\in\mathcal{X}}g(x)\leq\sup_{x\in\mathcal{X}}g(x)<\infty$, the functions $\phi_{j}$ satisfy $\int_{\mathcal{X}}|\phi_{j}(x)|^{1+\delta}dx<\infty,j=0,\ldots,l$ for some $\delta>0$ and Assumption \ref{interior_point} holds. Then the IW-SAA problem $(\mathcal{P}_{n})$ is equivalent to the linear program $(\mathcal{P}^{\prime}_{n})$ with the following discrete shape constraint:
\[
\left.
\begin{array}[c]{l}
L(X_{i})g(X_{i})\leq L(X_{j})g(X_{j}),\text{ if }X_{i}\geq X_{j}\text{ component-wise}\\
0\leq L(X_{i})g(X_{i})\leq M,i=1,\ldots,n
\end{array}
\right.,
\]
% \begin{align*}
% \sup\text{ }  &  \frac{1}{n}\sum_{i=1}^{n}\phi_{0}(X_{i})L(X_{i})\\
% \text{\rm subject to }  &  \frac{1}{n}\sum_{i=1}^{n}\phi_{j}(X_{i})L(X_{i})\leq \mu_{j},j=1,\ldots,m\\
% &  \frac{1}{n}\sum_{i=1}^{n}\phi_{j}(X_{i})L(X_{i})=\mu_{j},j=m+1,\ldots,l\\
% &  L(X_{i})g(X_{i})\leq L(X_{j})g(X_{j}),\text{ if }X_{i}\geq X_{j}\text{ component-wise}\\
% &  0\leq L(X_{i})g(X_{i})\leq M,i=1,\ldots,n
% \end{align*}
i.e., $(\mathcal{P}_{n})$ is finite-dimensionally reducible. Additionally, $|val(\mathcal{P}_{n})-val(\mathcal{P})|\overset{\text{a.s.*}}{\rightarrow}0$.
\end{corollary}

In Corollary \ref{OU_compact}, for ease of exposition, we assume the mode $a$ is a vertex of the domain $\mathcal{X}=[a,b]$. When $a$ is in a general position, the discrete shape constraint can be modified accordingly and the same conclusion holds. This observation also carries over to the subsequent cases with unbounded domain. Note that, although $val(\mathcal{P}_{n})$ is consistent, unlike in Corollary \ref{1D_compact}, orthounimodality may not have a $\sqrt{n}$ convergence rate because Assumption \ref{maximal_inequality} may not hold. In fact, the lower bound in Theorem 1.1 of \cite{gao2007entropy} implies that if $\inf_{x\in\mathcal{X}}g(x)>0$, then for any fixed $p\geq1$, $\log N_{[~]}(\varepsilon,\mathcal{F},||\cdot||_{p})\geq C\varepsilon^{-2}$ for some constant $C>0$ when $\varepsilon\leq1$. Therefore, the sufficient condition (\ref{compact_applicable2}) for Assumption \ref{maximal_inequality} does not hold. Even if one would like to directly verify Assumption \ref{maximal_inequality}, one may still need to show the bracketing integral condition $\int_{0}^{x}\sqrt{\log N_{[~]}(\varepsilon,\mathcal{F}_{j},||\cdot||_{2})}d\varepsilon<\infty$ for some $x$ and all $j=0,\ldots,l$ (or the finiteness of the uniform entropy integral; see, e.g., \cite{van2000asymptotic} Chapter 19) since these conditions may give the easiest way to prove the maximal inequality (see, e.g., the results in \cite{van2000asymptotic} Chapter 19 starting from Lemma 19.34). However, recall that one of $\phi_{j}$, say $\phi_{1}$, is the indicator function $I(x\in\mathcal{X})$ which denotes the total mass condition. Therefore, when $\sup_{x\in\mathcal{X}}g(x)<\infty$, the lower bound in Theorem 1.1 of \cite{gao2007entropy} also implies $\log N_{[~]}(\varepsilon,\mathcal{F}_{1},||\cdot||_{2})\geq C\varepsilon^{-2}$, which makes the bracketing integral condition fail for the function class $\mathcal{F}_{1}$ associated with $\phi_{1}$. The uniform entropy integral condition\ is also likely not to hold because the lower bound in Theorem 1.1 of \cite{gao2007entropy} also applies to the covering number.

\subsection{Unbounded Domain} \label{sec:unbounded}

For unbounded domains, as we discussed previously following Lemma \ref{general_verification}, (\ref{unbounded_applicable1}) and (\ref{unbounded_applicable2}) should be used to justify Assumptions \ref{GC_class} and \ref{maximal_inequality}. Conditions (\ref{unbounded_applicable1}) and (\ref{unbounded_applicable2}) are naturally applied to the case where the densities $f\in\mathcal{F}$ are known to be lighter-tailed than some density $g_{0}$, i.e., $f^{1-\gamma}\leq Mg_{0}$ for any $f\in\mathcal{F}$ because the sampling density $g\equiv g_{0}$ will automatically satisfy the requirements. We will focus on this case in the following.

\begin{corollary}
[\textbf{Monotonicity with unbounded domain}]\label{monotone_unbounded}Let $\mathcal{X}=[a,\infty)$ and the shape-constrained function class $\mathcal{F}=\{f:0\leq f(x)\leq(Mg_{0}(x))^{1/(1-\gamma)}~\forall x\in\mathcal{X}\text{, }f\text{ is non-increasing on }\mathcal{X}\}$, where $M>0$, $\gamma\in(0,1)$ and $g_{0}$ is a non-increasing positive probability density on $\mathcal{X}$ with $g_{0}(a)<\infty$. Let $g_{0}$ be the sampling distribution. Suppose the functions $\phi_{j}$ satisfy $||\phi_{j}||_{2+\delta}<\infty,j=0,\ldots,l$ for some $\delta>0$ and Assumption \ref{interior_point} holds. Then the IW-SAA problem $(\mathcal{P}_{n})$ is equivalent to the linear program $(\mathcal{P}^{\prime}_{n})$ with the following discrete shape constraint:
\[
\left.
\begin{array}[c]{l}
L(X_{(i+1)})g_{0}(X_{(i+1)})\leq L(X_{(i)})g_{0}(X_{(i)}),i=1,\ldots,n-1\\
0\leq L(X_{i})g_{0}(X_{i})\leq(Mg_{0}(X_{i}))^{1/(1-\gamma)},i=1,\ldots,n
\end{array}
\right.,
\]
% \begin{align*}
% \sup\text{ }  &  \frac{1}{n}\sum_{i=1}^{n}\phi_{0}(X_{i})L(X_{i})\\
% \text{subject to }  &  \frac{1}{n}\sum_{i=1}^{n}\phi_{j}(X_{i})L(X_{i})\leq \mu_{j},j=1,\ldots,m\\
% &  \frac{1}{n}\sum_{i=1}^{n}\phi_{j}(X_{i})L(X_{i})=\mu_{j},j=m+1,\ldots,l\\
% &  L(X_{(i+1)})g_{0}(X_{(i+1)})\leq L(X_{(i)})g_{0}(X_{(i)}),i=1,\ldots,n-1\\
% &  0\leq L(X_{i})g_{0}(X_{i})\leq(Mg_{0}(X_{i}))^{1/(1-\gamma)},i=1,\ldots,n
% \end{align*}
i.e., $(\mathcal{P}_{n})$ is finite-dimensionally reducible. Additionally, $|val(\mathcal{P}_{n})-val(\mathcal{P})|\overset{\text{a.s.*}}{\rightarrow}0$ and $\sqrt{n}(val(\mathcal{P}_{n})-val(\mathcal{P}))=O_{\mathbb{P}^{\ast}}(1)$.
\end{corollary}

We comment that the requirement of $g_{0}$ being non-increasing and positive is without loss of generality. Suppose $g_{0}$ is not non-increasing. We define $g_{1}(x):=\inf_{y\in\lbrack a,x]}g_{0}(y)$. It is easy to see $\mathcal{F}$ is unchanged if $g_{0}$ is replaced by $g_{1}$. We can also normalize $g_{1}$ to make it a probability density with a corresponding change of $M$ ($g_{1}$ has finite integral since it is bounded by $g_0$). Moreover, if $g_{1}(b)=0$ for some $b\in(a,\infty)$, we can see the effective domain is actually $\mathcal{X}^{\prime}=[a,b]$ because $f(x)=0$ for any $x\in(b,\infty),f\in\mathcal{F}$. In this case, we can take $\mathcal{X}^{\prime}$ as the new domain and apply Corollary \ref{1D_compact}. Similar discussion also applies to the remaining shape constraints in this subsection.
% and we will not repeat this explanation.
% the assumptions on $g_{0}$ later.
% We comment that the requirement of $g_{0}$ being non-increasing and positive will not lose any generality. Suppose $g_{0}$ is not non-increasing. We define $g_{1}(x):=\inf_{y\in\lbrack a,x]}g_{0}(y)$ (it is the largest non-increasing function that is no larger than $g$ on $\mathcal{X}$). It is easy to see $\mathcal{F}$ is unchanged if $g_{0}$ is replaced by $g_{1}$. We can also normalize $g_{1}$ to make it a probability density with a corresponding change of $M$. Moreover, if $g_{1}(b)=0$ for some $b\in(a,\infty)$, we can see the effective domain is actually $\mathcal{X}^{\prime}=[a,b]$ because $f(x)=0$ for any $x\in(b,\infty),f\in\mathcal{F}$. In this case, we can take $\mathcal{X}^{\prime}$ as the new domain and apply Corollary \ref{monotone_compact}. This discussion also applies to the remaining shape constraints so we will not bother to explain the assumptions on $g_{0}$ later.

\begin{corollary}
[\textbf{Unimodality with unbounded domain}]\label{unimodal_unbounded}Let $\mathcal{X}=(-\infty,\infty)$ and the shape-constrained function class $\mathcal{F}=\{f:0\leq f(x)\leq(Mg_{0}(x))^{1/(1-\gamma)}~\forall x\in\mathcal{X}\text{, }f\text{ is unimodal about }c\text{ on }\mathcal{X}\}$, where $M>0$, $\gamma\in(0,1)$, $c\in\mathbb{R}$ and $g_{0}$ is a positive unimodal (about $c$) probability density on $\mathcal{X}$ with $g_{0}(c)<\infty$. Let $g_{0}$ be the sampling distribution. Suppose the functions $\phi_{j}$ satisfy $||\phi_{j}||_{2+\delta}<\infty,j=0,\ldots,l$ for some $\delta>0$ and Assumption \ref{interior_point} holds. Then the IW-SAA problem $(\mathcal{P}_{n})$ is equivalent to the linear program $(\mathcal{P}^{\prime}_{n})$ with the following discrete shape constraint:
\[
\left.
\begin{array}[c]{l}
L(X_{(i-1)})g_{0}(X_{(i-1)})\leq L(X_{(i)})g_{0}(X_{(i)}),\text{ if }X_{(i)}\leq c\\
L(X_{(i+1)})g_{0}(X_{(i+1)})\leq L(X_{(i)})g_{0}(X_{(i)}),\text{ if }X_{(i)}\geq c\\
0\leq L(X_{i})g_{0}(X_{i})\leq(Mg_{0}(X_{i}))^{1/(1-\gamma)},i=1,\ldots,n
\end{array}
\right.,
\]
% \begin{align*}
% \sup\text{ }  &  \frac{1}{n}\sum_{i=1}^{n}\phi_{0}(X_{i})L(X_{i})\\
% \text{subject to }  &  \frac{1}{n}\sum_{i=1}^{n}\phi_{j}(X_{i})L(X_{i})\leq \mu_{j},j=1,\ldots,m\\
% &  \frac{1}{n}\sum_{i=1}^{n}\phi_{j}(X_{i})L(X_{i})=\mu_{j},j=m+1,\ldots,l\\
% &  L(X_{(i-1)})g_{0}(X_{(i-1)})\leq L(X_{(i)})g_{0}(X_{(i)}),\text{ if }X_{(i)}\leq c\\
% &  L(X_{(i+1)})g_{0}(X_{(i+1)})\leq L(X_{(i)})g_{0}(X_{(i)}),\text{ if }X_{(i)}\geq c\\
% &  0\leq L(X_{i})g_{0}(X_{i})\leq(Mg_{0}(X_{i}))^{1/(1-\gamma)},i=1,\ldots,n
% \end{align*}
i.e., $(\mathcal{P}_{n})$ is finite-dimensionally reducible. Additionally, $|val(\mathcal{P}_{n})-val(\mathcal{P})|\overset{\text{a.s.*}}{\rightarrow}0$ and $\sqrt{n}(val(\mathcal{P}_{n})-val(\mathcal{P}))=O_{\mathbb{P}^{\ast}}(1)$.
\end{corollary}

Corollary \ref{unimodal_unbounded} still holds if $\mathcal{X}$ is the half line $[a,\infty)$ or $(-\infty,b]$.

%Next, we consider multi-dimensional shapes. Since star unimodality and block unimodality fail to produce a finite-dimensionally reducible importance-weighted SAA problem even in the compact domain, we only consider orthounimodality here.

\begin{corollary}
[\textbf{Orthounimodality with unbounded domain}]\label{OU_unbounded}Let $a\in\mathbb{R}^{d}$ ($d \ge2$), $\mathcal{X}=[a,\infty)$ and the shape-constrained function class $\mathcal{F}=\{f:0\leq f(x)\leq(Mg_{0}(x))^{1/(1-\gamma)}~\forall x\in\mathcal{X}\text{, }f\text{ is orthounimodal about }a\text{ on }\mathcal{X}\}$, where $M>0$, $\gamma\in(0,1)$ and $g_{0}$ is a positive orthounimodal (about $a$) probability density on $\mathcal{X}$ with $g_{0}(a)<\infty$. Let $g_{0}$ be the sampling distribution. Suppose the functions $\phi_{j}$ satisfy $||\phi_{j}||_{1+\delta}<\infty,j=0,\ldots,l$ for some $\delta>0$ and Assumption \ref{interior_point} holds. Then the IW-SAA problem $(\mathcal{P}_{n})$ is equivalent to the linear program $(\mathcal{P}^{\prime}_{n})$ with the following discrete shape constraint:
\[
\left.
\begin{array}[c]{l}
L(X_{i})g(X_{i})\leq L(X_{j})g(X_{j}),\text{ if }X_{i}\geq X_{j}\text{ component-wise}\\
0\leq L(X_{i})g_{0}(X_{i})\leq(Mg_{0}(X_{i}))^{1/(1-\gamma)},i=1,\ldots,n
\end{array}
\right.,
\]
% \begin{align*}
% \sup\text{ }  &  \frac{1}{n}\sum_{i=1}^{n}\phi_{0}(X_{i})L(X_{i})\\
% \text{subject to }  &  \frac{1}{n}\sum_{i=1}^{n}\phi_{j}(X_{i})L(X_{i})\leq \mu_{j},j=1,\ldots,m\\
% &  \frac{1}{n}\sum_{i=1}^{n}\phi_{j}(X_{i})L(X_{i})=\mu_{j},j=m+1,\ldots,l\\
% &  L(X_{i})g(X_{i})\leq L(X_{j})g(X_{j}),\text{ if }X_{i}\geq X_{j}\text{ component-wise}\\
% &  0\leq L(X_{i})g_{0}(X_{i})\leq(Mg_{0}(X_{i}))^{1/(1-\gamma)},i=1,\ldots,n
% \end{align*}
i.e., $(\mathcal{P}_{n})$ is finite-dimensionally reducible. Additionally, $|val(\mathcal{P}_{n})-val(\mathcal{P})|\overset{\text{a.s.*}}{\rightarrow}0$.
\end{corollary}

Corollaries \ref{monotone_unbounded}, \ref{unimodal_unbounded} and \ref{OU_unbounded} together stipulate that under the corresponding assumptions, monotonic and unimodal densities in single dimension, and orthounimodal densities in multiple dimension, elicit finite-dimensional reducibility. Moreover, the single-dimensional solutions exhibit consistency and canonical convergence rates, while the orthounimodal solutions exhibit only consistency and there is no  guarantee on the canonical rate, much like the situation in the bounded domain case presented in Section \ref{sec:compact}.

\subsection{Limitations} \label{sec:limitations}
While our approach works for the cases presented in Sections \ref{sec:compact} and \ref{sec:unbounded}, it bears limitations when applying to some other shape constraints, either due to a failure of finite-dimensional reducibility or consistency. We discuss these limitations as follows. Note that, while these are negative results at the moment, they also open up potential future developments along our current direction that can remedy these issues.
% displayed in Table \ref. 

\subsubsection{Convexity in unbounded domain.} \label{sec:convexity_unbounded}
Let $\mathcal{X}=[a,\infty)$. Note that a convex function can be a density on $\mathcal{X}$ only if it is non-increasing. As discussed in Lemma \ref{general_verification}, a shape-constrained function class in an unbounded domain should be chosen to be $\mathcal{F}=\{f:0\leq f(x)\leq(Mg_{0}(x))^{1/(1-\gamma)}~\forall x\in\mathcal{X}\text{, }f\text{ is non-increasing and convex on }\mathcal{X}\}$, where $M>0$, $\gamma\in(0,1)$ and $g_{0}$ is a positive probability density on $\mathcal{X}$. Unfortunately, in this case, convexity does not allow a finite-dimensional reduction of ($\mathcal{P}_{n}$) even though $val(\mathcal{P}_{n})$ is still consistent and has $\sqrt{n}$-convergence. Notice that for the discrete points $X_{1},\ldots,X_{n}$, the constraint $L\in\mathcal{L}$ is naturally expressed in the following way as in Corollary \ref{1D_compact}:%
\begin{equation}
\left\{
\begin{array}[c]{l}%
\frac{L(X_{(i+1)})g_{0}(X_{(i+1)})-L(X_{(i)})g_{0}(X_{(i)})}{X_{(i+1)}-X_{(i)}}\leq\frac{L(X_{(i+2)})g_{0}(X_{(i+2)})-L(X_{(i+1)})g_{0}(X_{(i+1)})}{X_{(i+2)}-X_{(i+1)}},i=0,\ldots,n-1\\
L(X_{(i+1)})g_{0}(X_{(i+1)})\leq L(X_{(i)})g_{0}(X_{(i)}),i=1,\ldots,n-1\\
0\leq L(X_{i})g_{0}(X_{i})\leq(Mg_{0}(X_{i}))^{1/(1-\gamma)},i=1,\ldots,n
\end{array}
\right.  \label{discrete_convex_unobunded}%
\end{equation}
However, (\ref{discrete_convex_unobunded}) is not an equivalent reformulation of $L\in\mathcal{L}$ since we may not be able to extend the discrete values $f(X_{i})=L(X_{i})g_{0}(X_{i})$ to a real function $f\in\mathcal{F}$. Although we can use, e.g., linear interpolation to get a convex function $f$, the constraint $f(x)\leq(Mg_{0}(x))^{1/(1-\gamma)}~\forall x\in\mathcal{X}$ could be violated. In other words, the requirement $f(x)\leq(Mg_{0}(x))^{1/(1-\gamma)}$ imposes a special structure on the convex functions $f\in\mathcal{F}$, but the simple discrete version (\ref{discrete_convex_unobunded}) cannot capture it. Therefore, ($\mathcal{P}_{n}$) may not admit a finite-dimensional reduction.
% to our best knowledge.
%unless $g_{0}^{1/(1-\gamma)}$ is a concave function (which is actually impossible since $g_{0}$ is a probability density). 

% \DSQ{[DS:The notion of linearly interpolating to obtain an actual feasible function hasn't really been discussed before this point, perhaps it should be mentioned earlier.]} 

\subsubsection{$\alpha$-unimodality.} \label{sec:alpha}
% \textcolor{red}{Yes you are right that $\alpha$-unimodality can also be discussed in the one-dimensional case. I have never thought about this case since I always treat this as the multidimensional generalization of unimodality. By a quick thought, it seems bounding its bracketing number (for any $\alpha$ and both bounded and unbounded domain) is not very easy using the results I know. So can we restrict our discussion to $d\ge2$ for $\alpha$-unimodality? An alternative way is only discussing star unimodality but not the general $\alpha$-unimodality. Then it suffices to consider $d\ge2$.} 
A function $f$ is said to be $\alpha$-unimodal about the mode $a\in\mathbb{R}^{d}$ (also called star unimodal if $\alpha=d$) if $t^{d-\alpha}f(tx+a)$ is non-increasing in $t\in(0,\infty)$ for any nonzero vector $x\in\mathbb{R}^{d}$. The problem with the function class of $\alpha$-unimodality is that it is too large to satisfy Assumptions \ref{GC_class} and \ref{maximal_inequality}. More precisely, $\alpha$-unimodality neither has a finite bracketing number nor satisfies Assumptions \ref{GC_class} or \ref{maximal_inequality} even if assumed to be bounded when $\alpha\geq d$ ($\alpha$-unimodal functions cannot be bounded when $\alpha<d$ unless it is zero almost everywhere with respect to the Lebesgue measure). 
% (ZHENYUAN: I NEED TO REPLACE THE FOLLOWING RESULT. FIND A COUNTEREXAMPLE WHICH DIRECTLY SHOW THE SAA OPTIMAL VALUE IS NOT CONSISTENT FOR $\alpha$-UNIMODALITY.)

\begin{proposition}
[\textbf{Failure of sufficient conditions for $\alpha$-unimodality}]\label{failure_alpha}Let $\mathcal{X}=[a,b]\subset\mathbb{R}^{d}$ ($d\ge2$) be a hyperrectangle and $M>0$ be an arbitrary constant. Define the $\alpha$-unimodal function class as $\mathcal{F}=\{f:f(x)\geq0~\forall x\in\mathcal{X}\text{, }f\text{ is }\alpha\text{-unimodal about }a\text{ on }\mathcal{X}\}$ if $\alpha<d$, and $\mathcal{F}=\{f:0\leq f(x)\leq M~\forall x\in\mathcal{X}\text{, }f\text{ is }\alpha\text{-unimodal about }a\text{ on }\mathcal{X}\}$ if $\alpha\geq d$. Let $g$ be any probability density on $\mathcal{X}$ that is equivalent to the Lebesgue measure, i.e., they are absolutely continuous with respect to each other. Then for any $p\geq1$, $N_{[~]}(\varepsilon,\mathcal{F},||\cdot||_{p})=\infty$ for all sufficiently small $\varepsilon>0$. Also, if $\phi_{j}$ is not almost everywhere zero (with respect to Lebesgue measure), then $\mathcal{F}_{j}$ does not satisfy Assumptions \ref{GC_class} or \ref{maximal_inequality}.
\end{proposition}

Proposition \ref{failure_alpha} stipulates that $\alpha$-unimodality fails to satisfy the sufficient conditions for consistency and $\sqrt{n}$-convergence. Next, we provide a counterexample to directly show the IW-SAA optimal value for $\alpha$-unimodality may not be consistent.

Consider the following two-dimensional $\alpha$-unimodal shape-constrained problem
\begin{align*}
(\mathcal{P}_{\alpha}):\quad\quad\quad\sup_{f}~  &  \mathbb{E}_{f}[I(X\in
S)]\\
\text{subject to }  &  \mathbb{E}_{f}[I(X\in\mathcal{X})]=1\\
&  f\in\mathcal{F}%
\end{align*}
where $\mathcal{X}=\{x\in\mathbb{R}^{2}:x\geq0,||x||\leq2\}$, $S=\{x\in
\mathbb{R}^{2}:x\geq0,1\leq||x||\leq2\}$, $||\cdot||$ denotes the usual Euclidean norm, and the shape constraint is $\mathcal{F}=\{f:f(x)\geq0~\forall x\in\mathcal{X},f\text{ is }\alpha
\text{-unimodal about }(0,0)\text{ on }\mathcal{X}\}$ if $\alpha<2$ and $\mathcal{F}=\{f:0\leq f(x)\leq M~\forall x\in\mathcal{X},f\text{ is }%
\alpha\text{-unimodal about }(0,0)\text{ on }\mathcal{X}\}$ for any $M>4/(3\pi)$ if $\alpha\geq2$. Given a sampling distribution $g$ on
$\mathcal{X}$, the corresponding IW-SAA problem is given by%
\begin{align*}
(\mathcal{P}_{\alpha.n}):\quad\quad\quad\sup_{L}~  &  \frac{1}{n}\sum
_{i=1}^{n}I(X_{i}\in S)L(X_{i})\\
\text{subject to }  &  \frac{1}{n}\sum_{i=1}^{n}I(X_{i}\in\mathcal{X}%
)L(X_{i})=1\\
&  L\in\mathcal{L}%
\end{align*}
where $\mathcal{L}=\{f/g:f\in\mathcal{F}\}$. For this example, we have:

\begin{proposition}[\textbf{Lack of consistency for $\alpha$-unimodality}]\label{2Dcounterexample}
The optimal value of $(\mathcal{P}_{\alpha})$ satisfies $val(\mathcal{P}%
_{\alpha})\leq(2^{\alpha}-1)/2^{\alpha}$. For any sampling distribution $g$,
the IW-SAA optimal value satisfies $\mathbb{P}_{\ast}(\exists N_{0}\in\mathbb{N}\text{ s.t. }val(\mathcal{P}_{\alpha
.n})=1,\forall n\geq N_{0})=1$. Consequently, the IW-SAA optimal value is not consistent.
\end{proposition}

\subsubsection{Block unimodality.}\label{sec:block}
A function $f$ is said to be block unimodal about the mode $a\in\mathbb{R}^{d}$ on $[a,\infty)$, if it is orthounimodal and satisfies a nonnegativity condition: for any vectors $b_{1},b_{2}$ with $a\leq b_{1}<b_{2}$, the alternating sum satisfies
\begin{equation}
\sum_{v\in V}(-1)^{\#\text{ of components of }b_{1}\text{ in }v}f(v)\geq0,
\label{alternating_sum}%
\end{equation}
where $V$ are the vertices of the hyperrectangle $[b_{1},b_{2}]$. This nonnegativity condition is similar to that for a cumulative distribution function, i.e., if $f$ denotes a cumulative distribution function, then this condition means that $f$ has nonnegative mass in the hyperrectangle $[b_{1},b_{2}]$. Let $\mathcal{X}=[a,b]\subset\mathbb{R}^{d}$ ($d\ge2$) and $\mathcal{F}=\{f:0\leq f(x)\leq M~\forall x\in\mathcal{X}$, $f$ is block unimodal about $a$ on $\mathcal{X}\}$. Since block unimodality contains fewer functions than orthounimodality, condition (\ref{compact_applicable1}) in Lemma \ref{general_verification} must hold for block unimodality when it holds for orthounimodality. Therefore, we can show $val(\mathcal{P}_{n})$ is consistent with proper conditions as in Corollary \ref{OU_compact}. However, the equivalent reformulation of $L\in\mathcal{L}$ using discrete values $X_{1},\ldots,X_{n}$ in ($\mathcal{P}_{n}$) is hard to obtain. In fact, condition (\ref{alternating_sum}) vanishes if it is restricted to discrete values $X_{1},\ldots,X_{n}$ because any subset of $\{X_{1},\ldots,X_{n}\}$ with cardinality $2^d$ cannot constitute the vertices of any hyperrectangle with probability 1. If this condition is ignored in the IW-SAA, then the problem ($\mathcal{P}_{n}$) for block unimodality reduces to the same linear program for orthounimodality in Corollary \ref{OU_compact}. However, it is not guaranteed that each discrete feasible solution $L(X_{i}),i=1,\ldots,n$ in this linear program can be extended to a function $L\in\mathcal{L}$ that satisfies block unimodality. Thus, there is no equivalent reformulation of $L\in\mathcal{L}$ for block unimodality that only involves $X_{1},\ldots,X_{n}$.

Finally, we note that in the multidimensional case, \cite{lam2021orthounimodal} explains how orthounimodality can be more well-suited for extreme event applications compared to $\alpha$-unimodality and block unimodality. This is due to the sensitivity of $\alpha$-unimodality with respect to its mode and the difficulty in interpretability for block unimodality. Orthounimodality, on the other hand, is demonstrably less sensitive to the mode and easy to intuit. From this perspective, while currently IW-SAA faces limitations in applying to $\alpha$-unimodality and block unimodality, it encouragingly applies to the important case of orthounimodality.

\subsection{Summary of Results for the Considered Shape Constraints}
We close this section by connecting all our derived results back to Table \ref{checkbox}. Corollaries \ref{1D_compact}-\ref{OU_unbounded} in Sections \ref{sec:compact} and \ref{sec:unbounded} show that IW-SAAs applied to monotonicity, convexity (bounded domain), unimodality and orthounimodality are finite-dimensionally reducible to linear programs, which are indicated as the checks in the second column in Table \ref{checkbox}. Moreover, these corollaries show that their IW-SAAs are consistent and, except for orthounimodality, exhibit canonical convergence rates, which are indicated as the checks in the third and fourth columns in Table \ref{checkbox}. The remaining cases are discussed in Section \ref{sec:limitations}. Although IW-SAA applied to convexity (unbounded domain) is consistent and has canonical convergence rate, we have not found its finite-dimensional reduction. For $\alpha$-unimodality, Proposition \ref{2Dcounterexample} provides a counterexample to show its IW-SAA may not be consistent, let alone have canonical convergence rate. Hence, there is no reason to study its finite-dimensional reduction. For block unimodality, the only property we can obtain is consistency as explained in Section \ref{sec:block}. All these results are indicated by the checks, crosses and blanks in the corresponding rows in Table \ref{checkbox}. 

\section{Complexity in the Number of Constraints\label{sec:SAA_DRO_complexity}}
We have so far focused on the convergence properties of our IW-SAA. 
On the other hand, as the simulated scenario size $n$ increases, IW-SAA requires an increasing number of decision variables and constraints, making the linear program harder to solve. In this section, we discuss computationally efficient ways to encode our IW-SAA and the corresponding complexity with respect to the number of constraints therein. 

First, for one-dimensional shape constraints, by using order statistics, the number of discrete shape constraints is $\Theta(n)$ as in Corollary \ref{1D_compact}, which is not improvable.

For orthounimodality, the above sorting method does not work because orthounimodality relies on a partial order. In this case, a naive approach is to construct an orthounimodality constraint $L(X_{i})g(X_{i})\leq L(X_{j})g(X_{j})$ for each pair $X_{i}\geq X_{j}$. To study the complexity of this approach, we assume that the components of $X_{i}$ are continuously distributed and mutually independent. Let $N_{1}(n,d)$ denote the number of orthounimodality constraints constructed from $n$ samples in $\mathbb{R}^{d}$ using this approach. We have%
\[
\mathbb{E}_{g}[N_{1}(n,d)]=\mathbb{E}_{g}\left[  \sum_{1\leq i\neq j\leq n}I(X_{i}\geq X_{j})\right]  =\sum_{1\leq i\neq j\leq n}\mathbb{P}_{g}(X_{i}\geq X_{j})=\frac{n(n-1)}{2^{d}},
\]
where $\mathbb{P}_{g}(X_{i}\geq X_{j})=1/2^{d}$ is due to the mutually independent and continuously distributed components. We observe that the complexity of this naive approach is $\Theta(n^{2})$ for a fixed dimension, which is larger than that of the one-dimensional shapes. Additionally, the complexity is decreasing in $d$ because the condition $X_{i}\geq X_{j}$ is harder to achieve for higher dimensions.

One clear drawback of the above naive approach is that it includes many redundant constraints. For example, if we have a chain $X_{1}\geq X_{2}\geq X_{3}$, then it suffices to specify the constraints for $X_{1}\geq X_{2}$ and $X_{2}\geq X_{3}$ but the naive approach still includes the redundant constraint $X_{1}\geq X_{3}$. In fact, if we have such an inequality chain with length $k$, then this approach will specify $k(k-1)/2$ orthounimodality constraints from the chain while only $k-1$ of them are non-redundant. Removing such redundant constraints can significantly reduce the complexity, especially for low dimensions when the chain tends to be long. This removal procedure is known as transitive reduction in graph theory. Here we explain how this can be done via adjacency matrix multiplication. We construct a direct graph $G=(V(G),E(G))$ where the vertices $V(G)=\{1,\ldots,n\}$ denote all samples and a directed edge $i\rightarrow j$ exists $(i\neq j)$ if and only if $X_{i}\geq X_{j}$. We notice that a orthounimodality constraint for $X_{i}\geq X_{j}$ is non-redundant if and only if there is no other sample that is in between them, or equivalently, there is no path of length two that is from $i$ to $j$. Let $A$ be the adjacency matrix (a boolean matrix) of graph $G$. Then we compute the ``two-step adjacency matrix'' $A^{2}$ via boolean matrix multiplication and notice that the $(i,j)$ element in $A^{2}$ is $0$ if and only if there is no path of length two that is from $i$ to $j$. Therefore, it suffices to pick all pairs $(i,j)$ whose element in $A$ is $1$ but $0$ in $A^{2}$ and they correspond to all non-redundant orthounimodality constraints. We summarize this algorithm in Algorithm \ref{alg:efficient_OU}. Its time complexity is $O(n^{2.3729})$ using the best exact algorithms for matrix multiplication (\cite{le2014powers}). However, if we leverage the sparse structure of $A$ and use a sparse matrix to represent $A$ and do matrix multiplication, the algorithm can be accelerated.

\begin{algorithm}[hbt!]
\caption{Construction of non-redundant orthounimodality constraints}
\label{alg:efficient_OU}
\textbf{Inputs:} Samples $X_i,i=1,\ldots,n$ and sampling density $g$
\begin{algorithmic}
\STATE Construct a sparse $n\times n$ boolean matrix $A$ where $A_{ij}=1$ if $X_i\geq X_j,i\neq j$ and $A_{ij}=0$ otherwise
\STATE Compute $A^2$ as a boolean matrix
\RETURN Select all $(i,j)$ pairs satisfying $A_{ij}=1$ and $A^2_{ij}=0$ and construct the orthounimodality constraint $L(X_{i})g(X_{i})\leq L(X_{j})g(X_{j})$ accordingly
\end{algorithmic}
\end{algorithm}

% Now let us analyze the complexity of the orthounimodality constraints after  on orthounimodality constraints
Let $N_{2}(n,d)$ denote the number of non-redundant orthounimodality constraints from $n$ samples in $\mathbb{R}^{d}$ after the transitive reduction. We have:
\begin{theorem}\label{complexity}
Suppose the IW-SAA samples $X_{i}\in\mathbb{R}^{d},i=1,\ldots,n$ are i.i.d. from $g$ which have continuously distributed and mutually independent components. Then 
\[
\mathbb{E}_{g}[N_{2}(n,d)]=n(n-1)\sum_{k=0}^{n-2}\frac{\binom{n-2}{k}(-1)^{k}}{(k+1)^{d}(k+2)^{d}}.
\]
In particular, when $d=2$, $\mathbb{E}_{g}[N_{2}(n,2)]$ can be simplified as%
\[
\mathbb{E}_{g}[N_{2}(n,2)]=(n+1)\sum_{i=1}^{n}\frac{1}{i}-2n=\Theta(n\log n).
\]
Additionally, we have%
\[
\frac{1}{2}\mathbb{E}_{g}[N_{2}(n,d-1)]\leq \mathbb{E}_{g}[N_{2}(n,d)]\leq\frac{n(n-1)}{2^{d}},~\forall d\geq2.
\]
\end{theorem}

Theorem \ref{complexity} shows how the dependence of the number of orthounimodality constraints on $n$ can be reduced via Algorithm \ref{alg:efficient_OU}. In particular, for $d=2$, the dependence is reduced from $\mathbb{E}_{g}[N_{1}(n,2)]=\Theta(n^2)$ to $\mathbb{E}_{g}[N_{2}(n,2)]=\Theta(n\log n)$.

% DISCUSS THE ABOVE THEOREM BRIEFLY

% ZHENYUAN: DONE.

\section{Numerical Experiments\label{sec:num}}
We present numerical results to validate our theory and demonstrate the performances of our IW-SAA. Section \ref{sec:test} first verifies the consistency and $\sqrt{n}$ convergence rate of the IW-SAA optimal value for the one-dimensional case, and suggests a similar rate even for the multidimensional case. Section \ref{sec:inf} investigates the influence of the choice of sampling distribution. Section \ref{sec:GMP} demonstrates the capability of our IW-SAA in solving a complicated multidimensional orthounimodal distributional optimization problem with more than 50 constraints. Finally, Section \ref{sec:rare} uses our IW-SAA to solve shape-constrained distributional optimization motivated from rare-event estimation. Each experiment is replicated 100 times without further clarification.

\subsection{Validation of Convergence Rates\label{sec:test}}
% REVISE THE FOCUS AND REWRITE. WE DON'T WANT TO COMPARE. INSTEAD, WE WANT TO FIRST VALIDATE OUR THEORY ON CONSISTENCY AND CONVERGENCE RATE FOR UNIMODAL. THIS IS AN IMPORTANT RESULT FIRST BY ITSELF. THEN WE WANT TO DEMONSTRATE THAT THE THEORY ALSO APPEARS TO MATCH FOR OU, AND THAT MORE SO THE RESULTS SUGGEST THAT THE CANONICAL RATE HOLDS FOR OU AS WELL EVEN THOUGH THERE'S NO THEORY.

% ZHENYUAN: REVISED THIS SECTION.  $\sqrt{n}$ convergence rate in Corollary \ref{1D_compact}  to see how fast the IW-SAA of orthounimodal distributional optimization may be

We first use a relatively simple unimodal distributional optimization problem to validate our theoretical consistency and convergence rate. Then we investigate these convergence behaviors for a 4-dimensional orthounimodal distributional optimization problem. Specifically, we consider the following two problems:%
\[%
\begin{array}[t]{r}%
\sup\\
\text{subject to}%
\end{array}%
\begin{array}[t]{l}%
\mathbb{P}_{f}(-3\leq X\leq3)\\
\mathbb{P}_{f}(-5\leq X\leq5)=1\\
\mathbb{P}_{f}(-1\leq X\leq1)=1/4\\
\mathbb{E}_{f}[XI(|X|\leq5)]=0\\
\mathbb{E}_{f}[X^{2}I(|X|\leq5)]\leq6\\
f\text{ is unimodal about }0\text{ on }\mathcal{X}_{1}\\
0\leq f\leq1
\end{array}%
\begin{array}[t]{r}%
\sup\\
\text{subject to}%
\end{array}%
\begin{array}[t]{l}%
\mathbb{P}_{f}(||X||\leq2,X\in\mathbb{R}_{+}^{4})\\
\mathbb{P}_{f}(||X||\leq2.5,X\in\mathbb{R}_{+}^{4})=1\\
\mathbb{P}_{f}(||X||\leq1.5,X\in\mathbb{R}_{+}^{4})=1/5\\
f\text{ is orthounimodal about }\mathbf{0}\text{ on }\mathcal{X}_{2}\\
0\leq f\leq1/4
\end{array}
\]
where $\mathcal{X}_{1}=[-5,5]$, $\mathcal{X}_{2}=\{x\in\mathbb{R}_{+}^{4}:||x||\leq2.5\}$ and $||\cdot||$ denotes the usual Euclidean norm. For these problems, by calculus and direct usage of the shape properties, we can obtain their optimal values, with the first one being $3/4$ attained by the uniform distribution on $[-4,4]$, and the second one being $256/405$ attained by the uniform distribution on $\{x\in\mathbb{R}_{+}^{4}:||x||\leq1.5\sqrt[4]{5}\}$. On the other hand, note that both problems cannot be handled by existing Choquet-based solutions. The first problem contains inequality moment constraints together with the density constraint $0\leq f\leq 1$, while the second problem violates both the epigraph requirements for the objective and the hyperrectangle requirements for the moment functions; see Table \ref{table_conditions}. 
% TRUE? YES.

% (IN FACT, IS IT REALLY THAT BIG A DEAL OR IT CAN BE EASILY EXTENSIBLE? IF IT CAN BE EASILY EXTENSIBLE, THEN WE SHOULD DISCUSS AND NOT COUNT IT AS OUT OF SCOPE OF THE PREVIOUS WORK)(ZHENYUAN: YOU ARE RIGHT. THEY SAY THEIR FORMULATION CAN EXTEND TO INEQUALITY CONSTRAINTS. I CHANGED THE TABLE BACK. BUT THIS FORMULATION IS STILL DIFFERENT FROM THEIRS SINCE THEY DON'T HAVE THE DENSITY CONSTRAINT $0\leq f\leq 1$), IS IT TRUE? THE "ESTIMATION ERROR" ISN'T DEFINED BEFORE SO IT'S AMBIGUOUS. (ZHENYUAN: YES.)

For IW-SAA, the sampling distributions are chosen as uniform distributions on $\mathcal{X}_{1}$ and $\mathcal{X}_{2}$ respectively. Figure \ref{SAA-DRO_comparison} shows the boxplots of the normalized estimation errors of our IW-SAA outputs against the true optimal values, i.e., $\sqrt{n}$ times the difference between the IW-SAA output and optimal value. We observe that the normalized estimation errors in both problems mostly lie in bounded regions as $n$ grows. These validate our consistency results in Corollaries \ref{1D_compact} and \ref{OU_compact} since they imply that the unnormalized estimation errors shrink to 0 as $n$ grows. More precisely, for the unimodal distributional optimization problem, bounded normalized estimation errors validate the $\sqrt{n}$ convergence rate, i.e., $\sqrt{n}(val(\mathcal{P}_n)-val(\mathcal{P})=O_{\mathbb{P}^{\ast}}(1)$ in Corollary \ref{1D_compact}. For the orthounimodal distributional optimization problem, the similar trend hints that this rate may continue to apply to this case even though no rate result has been established for this setting.

\begin{figure}[htp]
\centering

\subfloat[unimodal distributional optimization]{\includegraphics[width=0.45\textwidth]{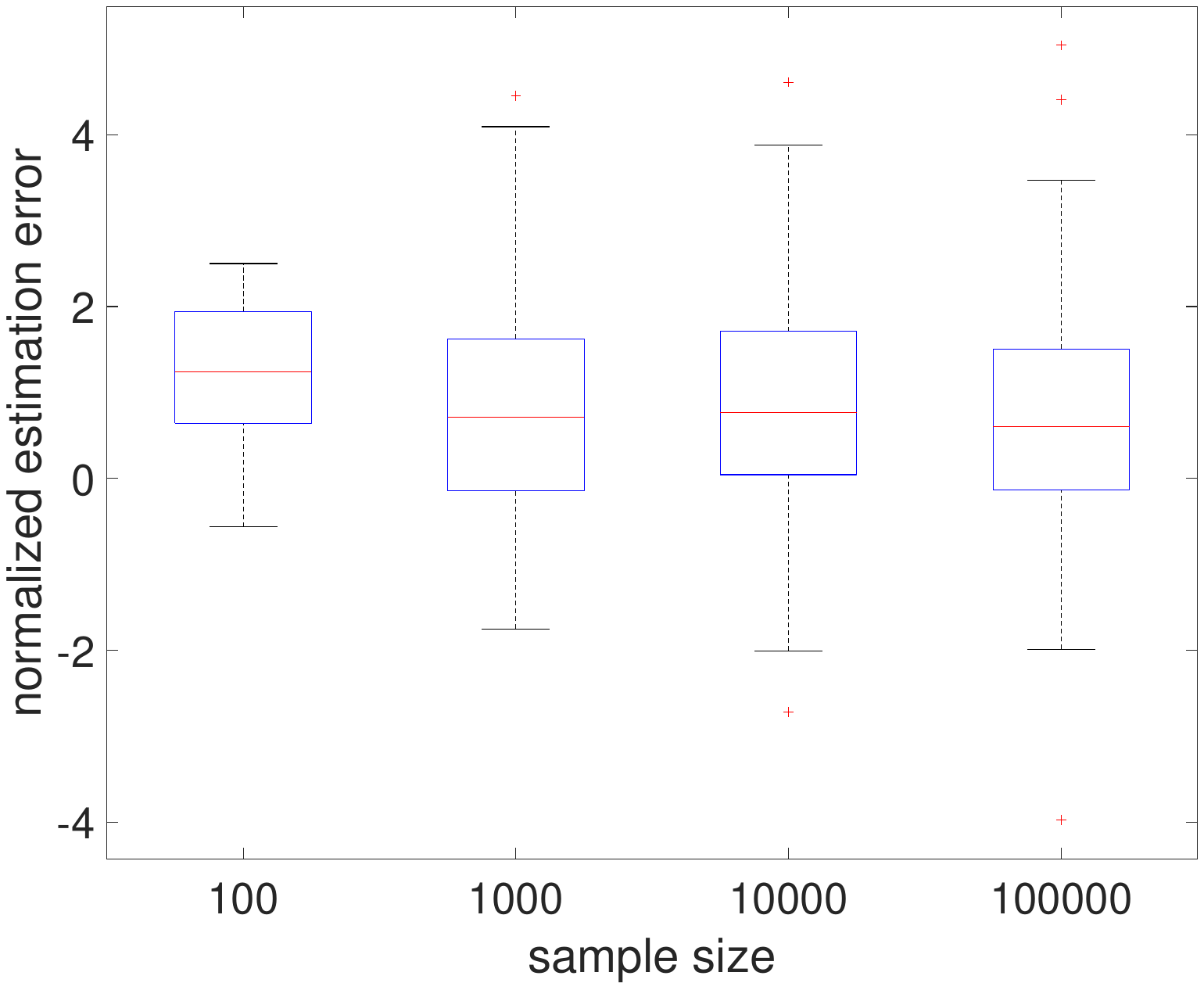}}
\subfloat[orthounimodal distributional optimization]{\includegraphics[width=0.45\textwidth]{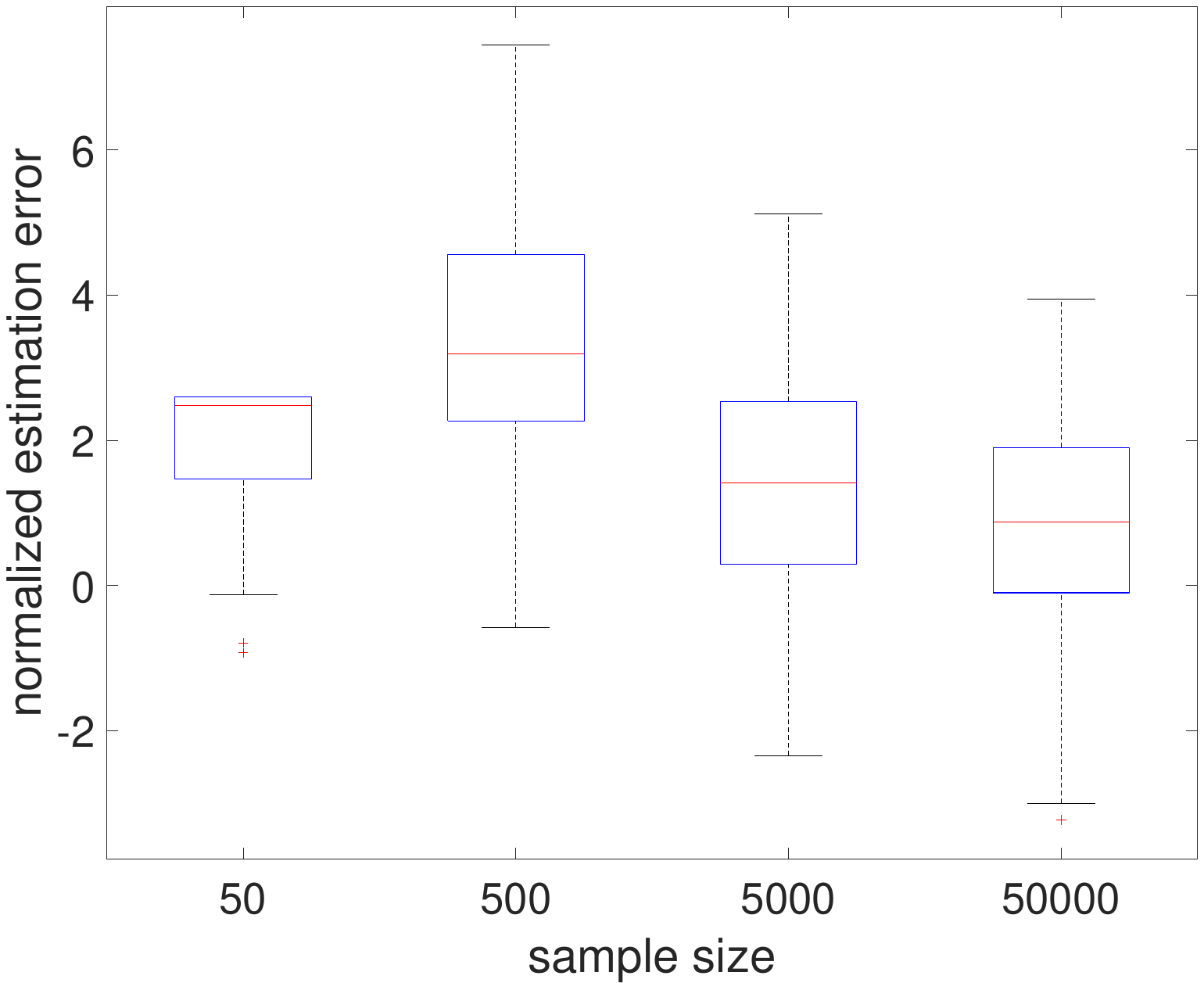}}

\caption{Normalized estimation errors of a single-dimensional unimodal and a 4-dimensional orthounimodal distributional optimization problem, for different sample sizes and over 100 experimental repetitions. 
% \DSQ{[Make axes labels and numbers bigger.]}\textcolor{magenta}{[]Zhenyuan: It seems there are no many distinctions. Maybe delete this section.]}
}
\label{SAA-DRO_comparison}
\end{figure}

\subsection{Influence of the Sampling Distribution\label{sec:inf}}

We examine the influence of the sampling distribution {$g$} on the quality of the IW-SAA optimal value. We consider the following unimodal distributional optimization problem%
\begin{align*}
\sup\text{ }  &  \mathbb{E}_{f}[e^{X}I(|X|\leq5)]\\
\text{subject to }  &  \mathbb{P}_{f}(-5\leq X\leq5)=1\\
&  \mathbb{E}_{f}[X^{i}I(|X|\leq5)]=\mu_{i},i=1,2,3,4\\
&  f\text{ is unimodal about }0\text{ on }\mathcal{X}\text{ and }0\leq f\leq2
\end{align*}
where $\mathcal{X}=[-5,5]$ and the values of $\mu_i$ are calibrated by $N(0,2)|\mathcal{X}$, with $N(0,2)|\mathcal{X}$ being the distribution of $N(0,2)$ conditional on that it is in $\mathcal{X}$ (same for the following without further clarification). We test four sampling distributions: uniform distribution on $\mathcal{X}$, $N(0,1)|\mathcal{X}$, $N(0,4)|\mathcal{X}$ and $N(0,9)|\mathcal{X}$. We use a sample size $5\times10^{4}$. Note that, like in Section \ref{sec:test}, this example cannot be handled by existing Choquet-based solutions as the objective function is not semi-algebraic; see Table \ref{table_conditions}.

Figure \ref{influence} displays the results. We observe that IW-SAA using the uniform distribution, $N(0,4)|\mathcal{X}$ and $N(0,9)|\mathcal{X}$ converge to almost the same value with the considered sample size, while IW-SAA using $N(0,1)|\mathcal{X}$ has not. On one hand, the similar convergences using $N(0,4)|\mathcal{X}$ and $N(0,9)|\mathcal{X}$ are coherent with the consistency of IW-SAA in Corollary \ref{1D_compact}. On the other hand, it appears that using $N(0,1)|\mathcal{X}$ underestimates the optimal value with as large as  $5\times10^{4}$ samples, since $N(0,1)|\mathcal{X}$ is too concentrated and cannot explore the whole domain $[-5,5]$ broadly enough.  This suggests that, in order to ensure IW-SAA results in a good-quality solution, the sampling distribution needs to be chosen suitably to be able to sufficiently explore the domain.

% (NOT SURE IF I UNDERSTAND THIS). (ZHENYUAN: THAT SENTENCE SHOULD BE REMOVED.)
% In particular, the objective value under this sampling distribution is around 2.67 (which is not $val(\mathcal{P})$ but just an objective value at one feasible solution) 

% \DSQ{I think this truncation needs more explanation.} 

\begin{figure}[htp]
\centering
\includegraphics[width=0.5\textwidth]{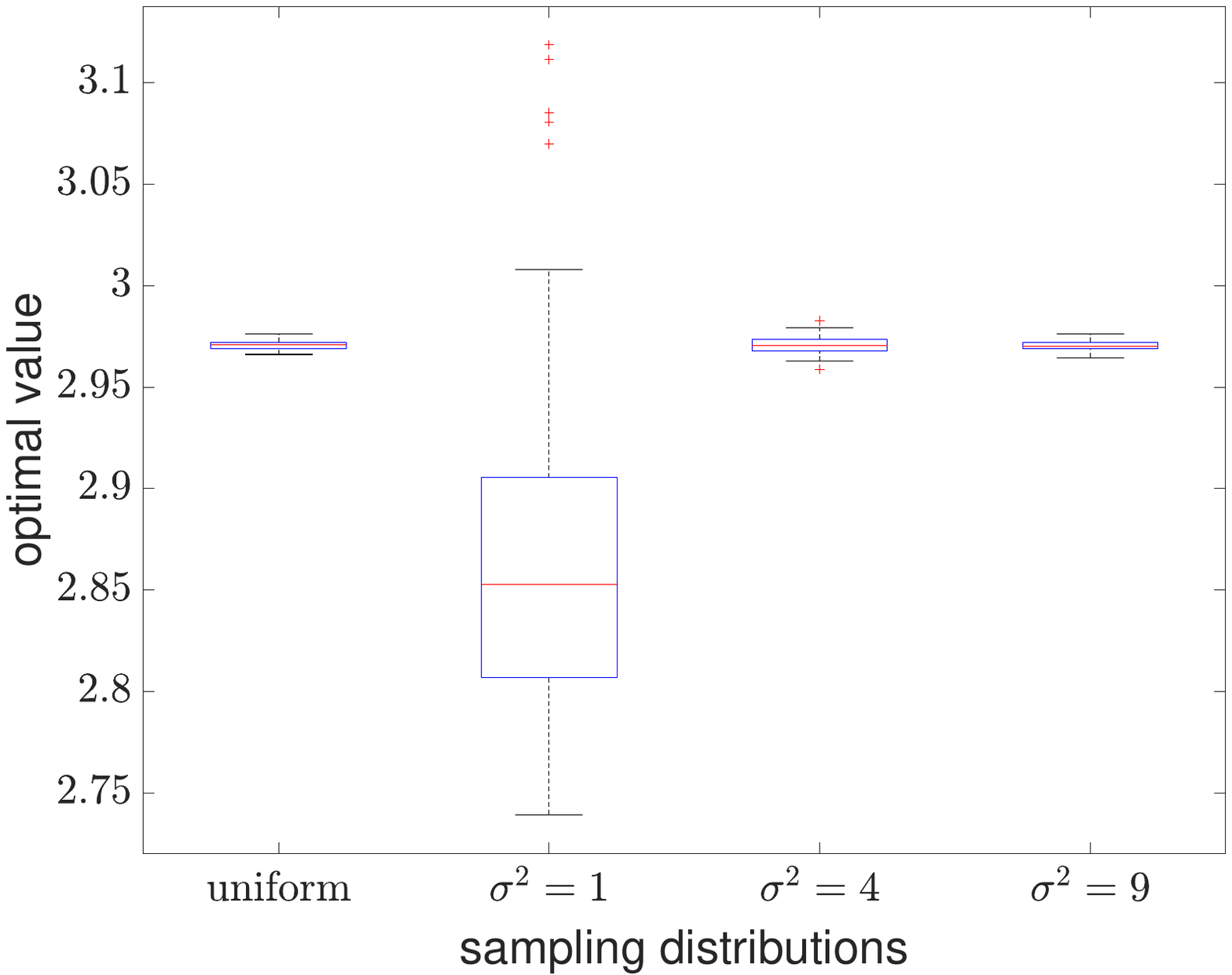}
\caption{Optimal values obtained from different sampling distributions.
% \DSQ{[Maybe include the true optimal here? Also I am not sure why each replication has a boxplot. This problem was a little confusing to me.]}\textcolor{magenta}{(Zhenyuan: Still waiting for additional numerical experiments. As explained at the beginning, we don't know the true optimal value. Why each replication has a boxplot: because there are two kinds of randomness in this problem. For each replication, we generate a new dataset. For each dataset, we calibrate the constraints and run the SAA. Since SAA method itself has randomness, we have a boxplot for each replication.)}
}
\label{influence}
\end{figure}

\subsection{High-Dimensional Many-Constraints Orthounimodal Distributional Optimization\label{sec:GMP}}
While the previous two subsections have exemplified problems that cannot be solved by existing Choquet-based methods, some of these problems could still be viewed as somewhat simple as we can find other means to obtain their true optimal values. In this section, we consider a significantly more complicated 10-dimensional orthounimodal distributional optimization problem with more than 50 moment constraints. We demonstrate how IW-SAA is capable of solving it. Specifically, consider the following problem:
\begin{align*}
\sup\text{ }  &  \mathbb{E}_{f}\left[  \left(  \sum_{i=1}^{10}X_{i}\right)I(X\in\mathcal{X})\right] \\
\text{subject to }  &  \mathbb{P}_{f}(X\in\mathcal{X})=\mu_{1}\\
&  \mathbb{P}_{f}(0\leq X_{i}\leq1/2,X\in\mathcal{X})=\mu_{2},i=1,\ldots,10\\
&  \mathbb{P}_{f}(0\leq X_{i}\leq1,X\in\mathcal{X})=\mu_{3},i=1,\ldots,10\\
&  \mathbb{E}_{f}[X_{i}^{j}I(X\in\mathcal{X})]=\mu_{3+j},i=1,\ldots,10,1\leq j\leq3\\
&  f\text{ is orthounimodal about }\mathbf{0}\text{ on }\mathcal{X}\text{ and }0\leq f\leq(Mg_{0})^{2}%
\end{align*}
where $X=(X_{1},\ldots,X_{10})$, $\mathcal{X}=[0,\infty)^{10}$, $\mu_{j}$ are calibrated by the true density $f_{0}\sim N(\mathbf{0},16I_{10\times10})$ ($I_{10\times10}$ is the identity matrix in $\mathbb{R}^{10}$), $g_{0}$ is the density of $N(\mathbf{0},32I_{10\times10})|\mathcal{X}$ and $M=2\max_{x\in\mathcal{X}}\sqrt{f_{0}(x)}/g_{0}(x)=2^{17/2}\pi^{5/2}$. We use $10^{5}$ samples from $g_{0}$ to drive the IW-SAA problem. The result is displayed in Figure \ref{GMP}. We see that the range of the boxplot is quite small, suggesting the convergence of our IW-SAA and coherence with Corollary \ref{OU_unbounded}. This result demonstrates that our IW-SAA is able to solve high-dimensional many-constraints distributional optimization problems. 
% WE CANNOT SHOW NUMERICS WITHOUT ANY DISCUSSION (ZHENYUAN: ADDED THE DISCUSSION.)
% We can see $f_{0}$ is a feasible solution with objective value around $4.17\times 10^{-3}$. 
% From the results, we can see the DRO optimal value indeed provides an upper bound of the objective value under $f_{0}$.

\begin{figure}[htp]
\centering
\includegraphics[width=0.5\textwidth]{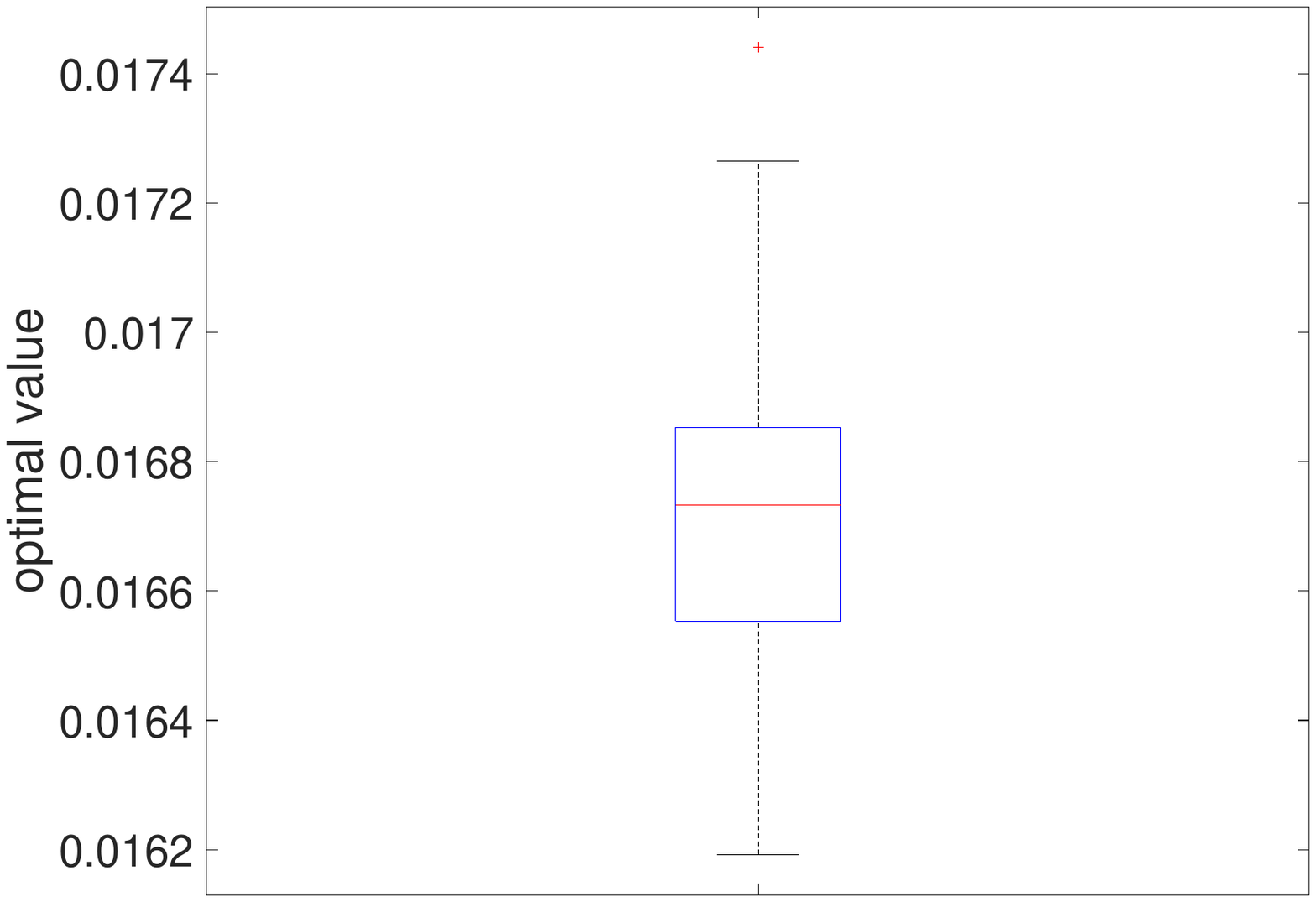}
\caption{Optimal values obtained in a 10-dimensional orthounimodal distributional optimization problem with more than 50 constraints.
% \DSQ{[Maybe include the true optimal here? Also I am not sure why each replication has a boxplot. This problem was a little confusing to me.]}\textcolor{magenta}{(Zhenyuan: Still waiting for additional numerical experiments. As explained at the beginning, we don't know the true optimal value. Why each replication has a boxplot: because there are two kinds of randomness in this problem. For each replication, we generate a new dataset. For each dataset, we calibrate the constraints and run the SAA. Since SAA method itself has randomness, we have a boxplot for each replication.)}
}
\label{GMP}
\end{figure}

\subsection{Extreme Event Analysis\label{sec:rare}}

In this experiment, we use our IW-SAA to solve a shape-constrained distributional optimization problem motivated from extreme event estimation. Suppose we have samples $(X_{1,k},\ldots X_{10,k}),k=1,\ldots,10^{7}$ from the unknown true distribution $N(\mathbf{0},16I_{10\times10})$. Suppose we want to estimate the 10-dimensional rare-event probability $\mathbb{P}(\min_{1\leq i\leq10}X_{i}\geq 3,\max_{1\leq i\leq10}X_{i}\geq12)$ whose true value is around $2.07\times 10^{-8}$. The empirical distribution cannot be used to estimate such a tiny probability because typically none of $10^{7}$ samples will fall into target region. Instead, following the argument in \cite{lam2021orthounimodal}, we find an upper bound on the true probability by imposing 
% shape-constrained distributional optimization to. We formulate 
the following orthounimodal distributional optimization problem%
\begin{align*}
\sup\text{ }  &  \mathbb{P}_{f}\left(  \min_{1\leq i\leq10}X_{i}\geq3,\max_{1\leq i\leq10}X_{i}\geq12\right) \\
\text{subject to }  &  l_{0}\leq \mathbb{P}_{f}(X\in\mathcal{X})\leq u_{0}\\
&  l_{i,j}\leq \mathbb{P}_{f}(1\leq X_{i}\leq\hat{q}_{i}(j/10),X\in\mathcal{X})\leq u_{i,j},i=1,\ldots,10,j=1,\ldots,9\\
&  l_{11,j}\leq \mathbb{P}_{f}\left(  1\leq\min_{1\leq i\leq10}X_{i}\leq\hat{q}_{11}(j/10),X\in\mathcal{X}\right)  \leq u_{11,j},j=1,\ldots,9\\
&  l_{12,j}\leq \mathbb{P}_{f}\left(  1\leq\max_{1\leq i\leq10}X_{i}\leq\hat{q}_{12}(j/10),X\in\mathcal{X}\right)  \leq u_{12,j},j=1,\ldots,9\\
&  f\text{ is orthounimodal about }\mathbf{1}\text{ on }\mathcal{X}\text{ and }0\leq f\leq(Mg_{0})^{2}%
\end{align*}
where $X=(X_{1},\ldots,X_{10})$, $\mathcal{X}=[1,\infty)^{10}$, $g_{0}$ is the density of $N(\mathbf{1},32I_{10\times10})|\mathcal{X}$ and $M=2^{17/2}\pi^{5/2}$. The moment constraints can be calibrated via the confidence intervals constructed from data. In particular, the first one $l_{0}\leq \mathbb{P}_{f}(X\in\mathcal{X})\leq u_{0}$ is directly obtained from the usual two-sided normal confidence interval, while the remaining ones are calibrated from the Kolmogorov--Smirnov (KS) two-sided statistic that gives a simultaneous confidence band for $\mathbb{P}_{f}(1\leq X_{1}\leq x|X\in\mathcal{X}),\forall x\geq1$ using the conditional empirical distribution of $X_{1,k}$ given $(X_{1,k},\ldots X_{10,k})\in\mathcal{X}$. Then taking $x=\hat{q}_{1}(j/10),j=1,\ldots,9$ and multiplying the confidence bounds of $\mathbb{P}_{f}(X\in\mathcal{X})$ leads to the desired confidence intervals in the DRO constraints, where $\hat{q}_{1}(p)$ is the $100p\%$-quantile of the conditional empirical distribution of $X_{1,k}$. Other constraints are calibrated in the same way where $\hat{q}_{i}(\cdot),i=1,\ldots,12$ are the quantiles of the conditional empirical distributions of $X_{i,k},i=1,\ldots,10$, $\min_{1\leq i\leq10}X_{i,k}$ and $\max_{1\leq i\leq10}X_{i,k}$. In particular, we calibrate each constraint with level $1-\alpha/13$ and the Bonferroni correction ensures that all moment constraints simultaneously hold with level $1-\alpha$ ($\alpha=0.05$ in our experiment). 
% As for the shape constraint, we can verify it by the kernel density estimation at least near the mode $\mathbf{1}$. The orthounimodality constraint can be checked by plotting the contours of the estimated density $\hat{f}$ and checking if the density is non-increasing as required by orthounimodality. $f\leq(Mg_{0})^{2}$ can be checked by seeing if $\hat{f}\leq(Mg_{0})^{2}$ holds.

% \DSQ{[Also need to define $\hat{q}_{11}$ and $\hat{q}_{12}$? Zhenyuan: Yes, define them in the next sentence.]} \ref{data_driven} 

We independently generate 20 different data sets. For each data set, we calibrate the constraints as in the above and run IW-SAA using $10^5$ samples from $g_{0}$ 20 times. Figure \ref{data_driven} shows the boxplot of the IW-SAA optimal values (the $x$-axis denotes different data sets and the $y$-axis denotes different IW-SAA runs for each data set). We see that the IW-SAA optimal values deliver an upper bound the true rare-event probability. Moreover, considering the tiny magnitude of the the true objective value, the estimated values appear reasonably tight.
% are arguably not overly conservative  .
% has such a tiny value.

%Once the constraints are calibrated, we run the SAA approach using $10^5$ samples from $g_{0}$. We repeat the procedure using 20 independently generated data sets. Figure \ref{data_driven} shows the boxplot of the SAA optimal values (the $x$-axis denotes different data sets). We can see the SAA optimal values are indeed upper bounds of the true rare event probability and they are not too conservative considering such a tiny true probability.

\begin{figure}[htp]
\centering
\includegraphics[width=0.75\textwidth]{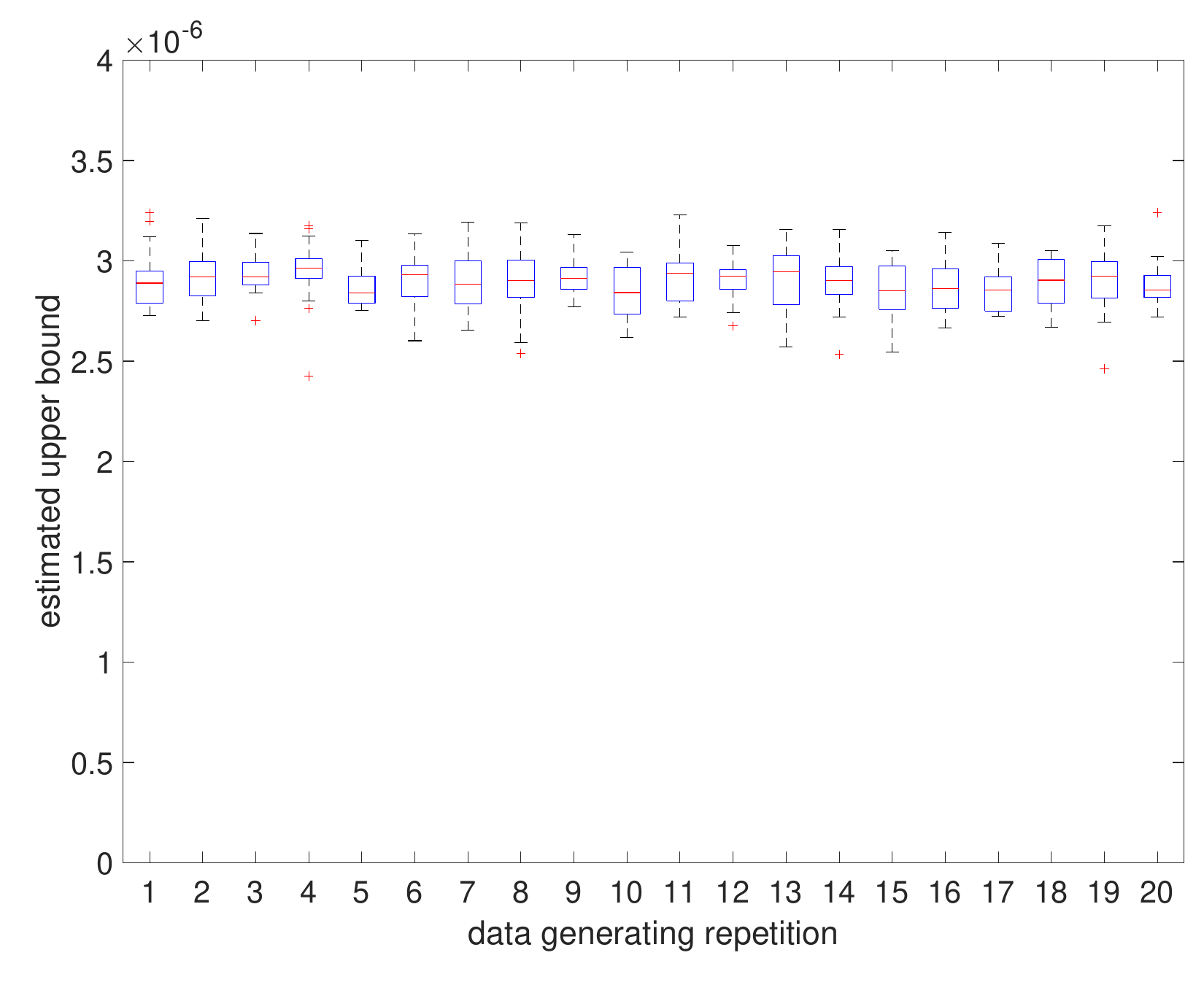}
\caption{Estimated upper bounds for a 10-dimensional extreme event estimation problem over 20 experimental repetitions.
% \DSQ{[Maybe include the true optimal here? Also I am not sure why each replication has a boxplot. This problem was a little confusing to me.]}\textcolor{magenta}{(Zhenyuan: Still waiting for additional numerical experiments. As explained at the beginning, we don't know the true optimal value. Why each replication has a boxplot: because there are two kinds of randomness in this problem. For each replication, we generate a new dataset. For each dataset, we calibrate the constraints and run the SAA. Since SAA method itself has randomness, we have a boxplot for each replication.)}
}
\label{data_driven}
\end{figure}

% ADD CONCLUSION AND DISCUSSION. FOR DISCUSSION, IT WOULD BE APPEALING TO AT LEAST GIVE SOME HINTS ON HOW TO GENERALIZE OUR APPROACH TO THE SHAPES THAT WE CANNOT HANLDE AT THE MOMENT

% ZHENYUAN: DONE.

\section{Conclusion and Discussion}

This paper proposes a new approach, which we call IW-SAA, to solve shape-constrained distributional optimization that arises prominently in DRO. Our approach is motivated from the tractability challenges faced by the existing Choquet-based method, which uses a mixture representation whose resulting solvability is confined by the geometries and the compatibility between the shape and the objective function and moment constraints. Instead of using such a mixture idea, IW-SAA uses a change of measure to convert the distributional decision variable into likelihood ratio with respect to a known sampling distribution, and uses SAA to approximate the objective and constraints. In order to analyze the consistency and canonical convergence rate of our approach, we develop a strong duality theory to remedy the mismatch of feasibility between original and sampled problems, and couple the duality with empirical processes, where the geometries of shape constraints are shown to play a critical role in controlling the corresponding function class complexities.  We apply IW-SAA to various one-dimensional shape-constrained distributional optimization and a recent multidimensional orthounimodal distributional optimization, and show that these formulations can be reduced to finite-dimensional linear programs and exhibit desirable convergence behaviors. 
% Compared to the existing literature, our approach has much fewer requirements on the problem formulation and thus significantly enlarges the scope of solvable shape-constrained distributional optimization. Consistency and canonical convergence rate of IW-SAA are established by means of strong duality and empirical process theory.

 % sampling from this distribution and replacing the expectations in the problem formulation by the sample averages  to better understand the behavior of the normalized estimation error
 
We view our study as a starting point in researching this new IW-SAA approach to tackle shape-constrained optimization. While our IW-SAA applies to a range of shapes, it also incurs limitations in some contexts, which prompt several immediate future directions. First, convexity (unbounded domain) and block unimodality lack finite-dimensional reducibility although they have statistical guarantees. To make IW-SAA applicable, an approach is to develop tractable reduction of the current infinite-dimensional sample problems by possibly leveraging specific function structure and shapes. Second, the convergence rate of IW-SAA for orthounimodality is still open, even though the numerical results in Section \ref{sec:test} hint that it could continue to be canonical. To either prove or disprove this conjecture, a more refined maximal inequality for empirical processes needs to be developed. Third, we have proved that IW-SAA for $\alpha$-unimodality is not consistent. Hence, solving $\alpha$-unimodal distributional optimization may still rely on the Choquet representation theorem like the existing literature. However, it is possible that by injecting further suitable constraints into the formulation that reduces the corresponding function class complexity, we can obtain consistency or even canonical convergence rates for such shape-constrained problems.
% , or even a new technique that has not been found out yet.

%\section{Acknowledgements}
% \textcolor{red}{Acknowledgements}

% Acknowledgments here
\section*{Acknowledgements}
We gratefully acknowledge support from the InnoHK initiative, the Government of the HKSAR, Laboratory for AI-Powered Financial Technologies, and Columbia Innovation Hub grant. We would also like to thank Raghu Pasupathy and Johannes Royset for their very kind suggestions that have helped improve this paper.

\bibliographystyle{plainnat}
\bibliography{references.bib}

\clearpage

\newpage

\begin{appendix}

\section{Measurability}\label{sec:remedy}
% \DSQ{[Moved all the measurability definitions to this section.]} \DS{
We employ the outer and inner probability and expectation to handle the measurability issue (see \cite{van1996weak} Section 1.2). Let $(\Omega,\mathcal{G},\mathbb{P})$ be a probability space. The outer probability $\mathbb{P}^{\ast}$ and the inner probability $\mathbb{P}_{\ast}$ of an arbitrary subset $B\subset\Omega$ are defined as%
\[
\mathbb{P}^{\ast}(B)=\inf\{\mathbb{P}(A):A\supset B\text{, }A\in\mathcal{G}\},\quad \mathbb{P}_{\ast}(B)=\sup\{\mathbb{P}(A):A\subset B\text{, }A\in\mathcal{G}\}.
\]
The inner and outer probability are related via $\mathbb{P}_{\ast}(B)=1-\mathbb{P}^{\ast}(B^{c})$, where $B^{c}$ denotes the complement of $B$. Similarly, for an arbitrary map $T:\Omega\mapsto\mathbb{\bar{R}}:=\mathbb{R}\cup\{\pm\infty\}$, its outer expectation is defined as%
\[
\mathbb{E}^{\ast}[T]=\inf\{\mathbb{E}[U]:U\geq T,U:\Omega\mapsto\mathbb{\bar{R}}\text{ measurable and }\mathbb{E}[U]\text{ is well-defined on }\mathbb{\bar{R}}\}.
\]
The inner expectation can be defined similarly. For the map $T$, there is a measurable map (unique up to a null set) $T^{\ast}:\Omega\mapsto\mathbb{\bar{R}}$ called the minimal measurable majorant of $T$ satisfying (i) $T^{\ast}\geq T$; (ii) $T^{\ast}\leq U$ a.s. for any measurable map $U:\Omega\mapsto\mathbb{\bar{R}}$ with $U\geq T$ a.s.; (iii) $\mathbb{E}^{\ast}[T]=\mathbb{E}[T^{\ast}]$ provided $\mathbb{E}[T^{\ast}]$ exists. The existence of $T^{\ast}$ can be found in \cite{van1996weak} Lemma 1.2.1. It also satisfies $\mathbb{P}^{\ast}(T>x)=\mathbb{P}(T^{\ast}>x)$ for any $x\in\mathbb{R}$. 

We next introduce stochastic convergence under the outer expectation (\cite{van1996weak} Section 1.9). Let $X_{n},X:\Omega\mapsto\mathbb{\bar{R}}$ be arbitrary maps. We say $X_{n}$ converges outer almost surely to $X$, denoted $X_{n}\overset{\text{a.s.*}}{\rightarrow}X$, if $|X_{n}-X|^{\ast}\rightarrow0$ a.s., where $|X_{n}-X|^{\ast}$ is the minimal measurable majorant of $|X_{n}-X|$. We say $X_{n}$ converges in outer probability to $X$, denoted $X_{n}\overset{\mathbb{P}^*}{\rightarrow}X$, if $|X_{n}-X|^{\ast}\rightarrow0$ in probability, which is equivalent to $\mathbb{P}(|X_{n}-X|^{\ast}>\varepsilon)=\mathbb{P}^{\ast}(|X_{n}-X|>\varepsilon)\rightarrow0$ for any $\varepsilon>0$. Clearly, $X_{n}\overset{\text{a.s.*}}{\rightarrow}X$ implies $X_{n}\overset{\mathbb{P}^*}{\rightarrow}X$.

\section{Proofs}\label{sec:SAA proofs}

% AT END OF PROOFS, ADD  FOR THE SAKE OF READABILITY.

% ZHENYUAN: DONE.

\begin{proof}[Proof of Theorem \ref{strong_duality}.]
We first prove weak duality $val(\mathcal{P})\leq val(\mathcal{D})$. For any feasible solution $L\in\mathcal{L}$ of the problem {($\mathcal{P}$)} (\ref{DRO_g}), we have
\begin{equation}
\mathbb{E}_{g}[\phi_{j}(X)L(X)]\leq\mu_{j},j=1,\ldots,m,\quad \mathbb{E}_{g}[\phi_{j}(X)L(X)]=\mu_{j},j=m+1,\ldots,l. \label{constraints_in_proof}%
\end{equation}
Thus, for any $\lambda\in\mathbb{R}_{+}^{m}\times\mathbb{R}^{l-m}$, we have%
\begin{align*}
\mathbb{E}_{g}[\phi_{0}(X)L(X)]
&  \leq \mathbb{E}_{g}[\phi_{0}(X)L(X)]-\sum_{j=1}^{l}\lambda_{j}(\mathbb{E}_{g}[\phi_{j}(X)L(X)]-\mu_{j})\\
&  \leq\sup_{L\in\mathcal{L}}\text{ }\left\{  \mathbb{E}_{g}[\phi_{0}(X)L(X)]-\sum_{j=1}^{l}\lambda_{j}(\mathbb{E}_{g}[\phi_{j}(X)L(X)]-\mu_{j})\right\}  .
\end{align*}
Taking the supremum over the feasible solutions of the problem {($\mathcal{P}$)} (\ref{DRO_g}), we have
\[
val(\mathcal{P})\leq\sup_{L\in\mathcal{L}}\text{ }\left\{  \mathbb{E}_{g}[\phi_{0}(X)L(X)]-\sum_{j=1}^{l}\lambda_{j}(\mathbb{E}_{g}[\phi_{j}(X)L(X)]-\mu_{j})\right\}  .
\]
Then we take the infimum over $\lambda\in\mathbb{R}_{+}^{m}\times\mathbb{R}^{l-m}$ on the right hand side and get $val(\mathcal{P})\leq val(\mathcal{D})$.

Next, we prove strong duality. Notice that if $val(\mathcal{P})=\infty$, then weak duality implies strong duality. Also, $val(\mathcal{P})>-\infty$ must hold because Assumption \ref{interior_point} ensures the existence of a feasible solution. So in the following, we assume $val(\mathcal{P})\in\mathbb{R}$. We will prove $val(\mathcal{P})\geq val(\mathcal{D})$. We define a set%
\begin{align*}
\mathcal{C}=  &  \{(r,s,t)\in\mathbb{R}\times\mathbb{R}^{m}\times\mathbb{R}^{l-m}:\exists~L\in\mathcal{L}\text{ s.t. }\mathbb{E}_{g}[\phi_{0}(X)L(X)]\geq r,\mathbb{E}_{g}[\phi_{j}(X)L(X)]-\mu_{j}\leq s_{j},\\
&  j=1,\ldots,m,\mathbb{E}_{g}[\phi_{j+m}(X)L(X)]-\mu_{j+m}=t_{j},j=1,\ldots,l-m\}.
\end{align*}
By Assumption \ref{convex_feasible_set}, we can see $\mathcal{C}$ is a convex set. Note that $(val(\mathcal{P}),\mathbf{0},\mathbf{0})$ is a boundary point of $\mathcal{C}$. By the supporting hyperplane theorem, there exists $(\lambda_{0},\lambda_{1},\lambda_{2})\neq\mathbf{0}$ in $\mathbb{R}\times\mathbb{R}^{m}\times\mathbb{R}^{l-m}$ such that%
\begin{equation}
\lambda_{0}val(\mathcal{P})\geq\sup_{(r,s,t)\in\mathcal{C}}(\lambda_{0}r-\lambda_{1}^{\top}s-\lambda_{2}^{\top}t). \label{supporting_hyperplane}%
\end{equation}
We claim that $\lambda_{1}\geq\mathbf{0}$ (component-wise) and $\lambda_{0}>0$. First, if $\lambda_{1}\geq\mathbf{0}$ doesn't hold, we assume without loss of generality that $\lambda_{1,1}<0$. We choose any feasible solution $L\in\mathcal{L}$ satisfying (\ref{constraints_in_proof}), then by setting $r=\mathbb{E}_{g}[\phi_{0}(X)L(X)]$, $s=(s_{1},\mathbf{0})$ (with $s_{1}>0$), $t=\mathbf{0}$ we have $(r,s,t)\in\mathcal{C}$. Letting $s_{1}\rightarrow\infty$, we can see $\sup_{(r,s,t)\in\mathcal{C}}(\lambda_{0}r-\lambda_{1}^{\top}s-\lambda_{2}^{\top}t)=\infty$, which contradicts (\ref{supporting_hyperplane}). Second, if $\lambda_{0}>0$ doesn't hold, we have either $\lambda_{0}<0$ or $\lambda_{0}=0$. If $\lambda_{0}<0$, we know that $(r,\mathbf{0},\mathbf{0})\in\mathcal{C}$ for $r<val(\mathcal{P})$. Letting $r\rightarrow-\infty$, we obtain $\sup_{(r,s,t)\in\mathcal{C}}(\lambda_{0}r-\lambda_{1}^{\top}s-\lambda_{2}^{\top}t)=\infty$ again, which contradicts (\ref{supporting_hyperplane}). If $\lambda_{0}=0$, we consider two cases depending on $\lambda_{1}$: $\lambda_{1}=\mathbf{0}$ or $\lambda_{1}\geq\mathbf{0}$ with at least one positive component. If $\lambda_{0}=0$ and $\lambda_{1}=\mathbf{0}$, we must have $\lambda_{2}\neq\mathbf{0}$ because $(\lambda_{0},\lambda_{1},\lambda_{2})\neq\mathbf{0}$. By Assumption \ref{interior_point}, we know that $t=\mathbf{0}$ is an interior point of
\[
\{(\mathbb{E}_{g}[\phi_{m+1}(X)L(X)]-\mu_{m+1},\ldots,\mathbb{E}_{g}[\phi_{l}(X)L(X)]-\mu_{l}):L\in\mathcal{L\}.}%
\]
Thus, there exists $(r_{1},s_{1},t_{1})\in\mathcal{C}$ s.t. $\lambda_{2}^{\top}t_{1}<0$. In this case (recall that $\lambda_{0}=0,\lambda_{1}=\mathbf{0}$), we have $\lambda_{0}val(\mathcal{P})=0$ but $\sup_{(r,s,t)\in\mathcal{C}}(\lambda_{0}r-\lambda_{1}^{\top}s-\lambda_{2}^{\top}t)=\sup_{(r,s,t)\in\mathcal{C}}(-\lambda_{2}^{\top}t)\geq-\lambda_{2}^{\top}t_{1}>0$, which contradicts (\ref{supporting_hyperplane}). If $\lambda_{0}=0$ and $\lambda_{1}\geq\mathbf{0}$ with at least one positive component, by Assumption \ref{interior_point}, there exists $L_{0}\in\mathcal{L}$ s.t.
\[
\mathbb{E}_{g}[\phi_{j}(X)L_{0}(X)]<\mu_{j},j=1,\ldots,m,\quad \mathbb{E}_{g}[\phi_{j}(X)L_{0}(X)]=\mu_{j},j=m+1,\ldots,l.
\]
Therefore, we have $\lambda_{0}val(\mathcal{P})=0$ (by $\lambda_{0}=0$) but
\[
\sup_{(r,s,t)\in\mathcal{C}}(\lambda_{0}r-\lambda_{1}^{\top}s-\lambda_{2}^{\top}t)=\sup_{(r,s,t)\in\mathcal{C}}(-\lambda_{1}^{\top}s-\lambda_{2}^{\top}t)\geq-\sum_{j=1}^{m}\lambda_{1,j}(\mathbb{E}_{g}[\phi_{j}(X)L_{0}(X)]-\mu_{j})>0,
\]
which contradicts (\ref{supporting_hyperplane}). Therefore, we must have $\lambda_{0}>0$.

Now by (\ref{supporting_hyperplane}), we have%
\begin{align*}
val(\mathcal{P})
&  \geq\sup_{(r,s,t)\in\mathcal{C}}\left(  r-\frac{\lambda_{1}^{\top}}{\lambda_{0}}s-\frac{\lambda_{2}^{\top}}{\lambda_{0}}t\right) \\
&  =\sup_{L\in\mathcal{L}}\left(  \mathbb{E}_{g}[\phi_{0}(X)L(X)]-\sum_{j=1}^{m}\frac{\lambda_{1,j}}{\lambda_{0}}(\mathbb{E}_{g}[\phi_{j}(X)L(X)]-\mu_{j})\right.\\
&-\left.\sum_{j=1}^{l-m}\frac{\lambda_{2,j}}{\lambda_{0}}(\mathbb{E}_{g}[\phi_{j+m}(X)L(X)]-\mu_{j+m})\right) \\
&  \geq\inf_{\lambda\in\mathbb{R}_{+}^{m}\times\mathbb{R}^{l-m}}\sup_{L\in\mathcal{L}}\text{ }\left\{  \mathbb{E}_{g}[\phi_{0}(X)L(X)]-\sum_{j=1}^{l}\lambda_{j}(\mathbb{E}_{g}[\phi_{j}(X)L(X)]-\mu_{j})\right\} \\
&  =val(\mathcal{D}).
\end{align*}
Combining it with weak duality, we have strong duality, i.e., $val(\mathcal{P})=val(\mathcal{D})$.
\end{proof}

\begin{proof}[Proof of Theorem \ref{strong_duality_SAA}.]
We notice that the proof of Theorem \ref{strong_duality} does not rely on a specific choice of $g$. Therefore, to prove {Theorem \ref{strong_duality_SAA}}, it suffices to verify the empirical version of Assumptions \ref{integrability_condition}-\ref{interior_point} where the distribution $g$ is replaced by the empirical distribution of $X_{1},\ldots,X_{n}$. Since the empirical version of Assumptions \ref{integrability_condition}-\ref{convex_feasible_set} automatically holds, weak duality automatically holds. To prove  strong duality, it remains to show the empirical version of Assumption \ref{interior_point} holds with probability approaching $1$.

We first assume that in the distributional optimization formulation, $1\leq m<l$, i.e., both the equality and inequality moment constraints exist. {The challenge will be to show there exists a feasible solution meeting the equality constraints.} By Assumption \ref{interior_point}, there exists a small $\delta>0$ s.t.%
\[
\mathbb{E}_{g}[\phi_{j}(X)L_{0}(X)]\leq\mu_{j}-\delta,j=1,\ldots,m,\quad \mathbb{E}_{g}[\phi_{j}(X)L_{0}(X)]=\mu_{j},j=m+1,\ldots,l,
\]
and the $2^{l-m}$ different vectors $(\mu_{m+1},\ldots,\mu_{l})+(\pm\delta,\ldots,\pm\delta)$ are contained in the set 
\[
\{(\mathbb{E}_{g}[\phi_{m+1}(X)L(X)],\ldots,\mathbb{E}_{g}[\phi_{l}(X)L(X)]):L\in\mathcal{L}\},
\]
say, they are achieved by $L_{1},\ldots,L_{2^{l-m}}\in\mathcal{L}$. We define $M_{1}=\max_{j=1,\ldots,m,k=1,\ldots,2^{l-m}}|\mathbb{E}_{g}[\phi_{j}(X)L_{k}(X)]-\mu_{j}|$ and $M_{2}=2+4M_1/\delta$. Now we consider the event%
\[
A_{n}=\left\{
\begin{array}[c]{l}%
\sum_{i=1}^{n}\phi_{j}(X_{i})L_{0}(X_{i})/n\leq\mu_{j}-\delta/2,j=1,\ldots,m,\\
|\sum_{i=1}^{n}\phi_{j}(X_{i})L_{0}(X_{i})/n-\mu_{j}|\leq\delta/(2M_{2}),j=m+1,\ldots,l,\\
|\sum_{i=1}^{n}\phi_{j}(X_{i})L_{k}(X_{i})/n-\mu_{j}|\leq2M_{1},j=1,\ldots,m,k=1,\ldots,2^{l-m},\\
|\sum_{i=1}^{n}\phi_{j}(X_{i})L_{k}(X_{i})/n-\mathbb{E}_{g}[\phi_{j}(X)L_{k}(X)]|<\delta/2,j=m+1,\ldots,l,k=1,\ldots,2^{l-m}.
\end{array}
\right\}  ,
\]
 By the strong law of large numbers, for almost all samples $X_{1},X_{2},\ldots$, $A_{n}$ happens when $n$ is large enough. In other words,%
\begin{equation}
\mathbb{P}\left(  \liminf_{n\rightarrow\infty}A_{n}\right)  =\mathbb{P}\left(\bigcup_{n=1}^{\infty}\bigcap_{k=n}^{\infty}A_{k}\right)  =1. \label{liminf_An}%
\end{equation}
Next, we will show on the event $A_{n}$, the empirical version of Assumption \ref{interior_point} holds. By the last requirement in $A_{n}$ and the choice of $L_{k}$, we know that there is one and only one vector
\[
\left(\sum_{i=1}^{n}\phi_{m+1}(X_{i})L_{k}(X_{i})/n,\ldots,\sum_{i=1}^{n}\phi_{l}(X_{i})L_{k}(X_{i})/n\right)
\]
among $k=1,\ldots,2^{l-m}$ in each orthant with respect to the center $(\mu_{m+1},\ldots,\mu_{l})$. Additionally, they are all outside the hyperrectangle $(\mu_{m+1},\ldots,\mu_{l})+[-\delta/2,\delta/2]^{l-m}$, that is,%
\[
\left\vert \frac{1}{n}\sum_{i=1}^{n}\phi_{j}(X_{i})L_{k}(X_{i})-\mu_{j}\right\vert >\frac{\delta}{2},j=m+1,\ldots,l,k=1,\ldots,2^{l-m}.
\]
Consequently, we know that%
\begin{align}
&  (\mu_{m+1},\ldots,\mu_{l})+[-\delta/2,\delta/2]^{l-m}\nonumber\\
&  \subset\left\{  \left(  \frac{1}{n}\sum_{i=1}^{n}\phi_{m+1}(X_{i})L(X_{i}),\ldots,\frac{1}{n}\sum_{i=1}^{n}\phi_{l}(X_{i})L(X_{i})\right):L=\sum_{k=1}^{2^{l-m}}\lambda_{k}L_{k},\sum_{k=1}^{2^{l-m}}\lambda_{k}=1,\lambda_{k}\geq0\right\} \label{inclusion_hyperrectangle}\\
&  \subset\left\{  \left(  \frac{1}{n}\sum_{i=1}^{n}\phi_{m+1}(X_{i})L(X_{i}),\ldots,\frac{1}{n}\sum_{i=1}^{n}\phi_{l}(X_{i})L(X_{i})\right):L\in\mathcal{L}\right\}  ,\nonumber
\end{align}
where the last inclusion is due to the convexity in Assumption \ref{convex_feasible_set}. This proves the interior point condition in the empirical version of Assumption \ref{interior_point}. Then it remains to show the existence of a strictly feasible solution to the IW-SAA problem. We define%
\[
v=\left(  \frac{1}{n}\sum_{i=1}^{n}\phi_{m+1}(X_{i})L_{0}(X_{i})-\mu_{m+1},\ldots,\frac{1}{n}\sum_{i=1}^{n}\phi_{l}(X_{i})L_{0}(X_{i})-\mu_{l}\right)  .
\]
According to the second requirement in $A_n$, $v\in\lbrack-\delta/(2M_{2}),\delta/(2M_{2})]^{l-m}$, which implies $-(M_{2}-1)v\in\lbrack-\delta/2,\delta/2]^{l-m}$. By (\ref{inclusion_hyperrectangle}), there exists a convex combination $\tilde{L}_{0}=\sum_{k=1}^{2^{l-m}}\lambda_{k}L_{k}$ such that%
\[
\left(  \frac{1}{n}\sum_{i=1}^{n}\phi_{m+1}(X_{i})\tilde{L}_{0}(X_{i}),\ldots,\frac{1}{n}\sum_{i=1}^{n}\phi_{l}(X_{i})\tilde{L}_{0}(X_{i})\right)=(\mu_{m+1},\ldots,\mu_{l})-(M_{2}-1)v.
\]
We claim that the convex combination $(1-1/M_{2})L_{0}+\tilde{L}_{0}/M_{2}\in\mathcal{L}$ is a strictly feasible solution for the IW-SAA problem. First, it satisfies the inequality constraints since for $j=1,\ldots,m$%
\begin{align*}
&  \frac{1}{n}\sum_{i=1}^{n}\phi_{j}(X_{i})\left(  \left(  1-\frac{1}{M_{2}}\right)  L_{0}(X_{i})+\frac{1}{M_{2}}\tilde{L}_{0}(X_{i})\right)  -\mu_{j}\\
&  =\left(  1-\frac{1}{M_{2}}\right)  \left(  \frac{1}{n}\sum_{i=1}^{n}\phi_{j}(X_{i})L_{0}(X_{i})-\mu_{j}\right)  +\frac{1}{M_{2}}\left(  \frac{1}{n}\sum_{i=1}^{n}\phi_{j}(X_{i})\tilde{L}_{0}(X_{i})-\mu_{j}\right) \\
&  =\left(  1-\frac{1}{M_{2}}\right)  \left(  \frac{1}{n}\sum_{i=1}^{n}\phi_{j}(X_{i})L_{0}(X_{i})-\mu_{j}\right)  +\frac{1}{M_{2}}\sum_{k=1}^{2^{l-m}}\lambda_{k}\left(  \frac{1}{n}\sum_{i=1}^{n}\phi_{j}(X_{i})L_{k}(X_{i})-\mu_{j}\right) \\
&  \leq-\frac{\delta}{2}\left(  1-\frac{1}{M_{2}}\right)  +\frac{1}{M_{2}}2M_{1}=\frac{-\delta}{2M_2}<0,
\end{align*}
by the first and third requirements in $A_{n}$ and the definition of $M_{2}$. It also satisfies the equality constraints since for $j=m+1,\ldots,l$
\begin{align*}
&  \frac{1}{n}\sum_{i=1}^{n}\phi_{j}(X_{i})\left(  \left(  1-\frac{1}{M_{2}}\right)  L_{0}(X_{i})+\frac{1}{M_{2}}\tilde{L}_{0}(X_{i})\right)  -\mu_{j}\\
&  =\left(  1-\frac{1}{M_{2}}\right)  \left(  \frac{1}{n}\sum_{i=1}^{n}\phi_{j}(X_{i})L_{0}(X_{i})-\mu_{j}\right)  +\frac{1}{M_{2}}\left(  \frac{1}{n}\sum_{i=1}^{n}\phi_{j}(X_{i})\tilde{L}_{0}(X_{i})-\mu_{j}\right) \\
&  =\left(  1-\frac{1}{M_{2}}\right)  {v_{j-m}}-\frac{M_{2}-1}{M_{2}}{v_{j-m}}=0.
\end{align*}
Hence, $(1-1/M_{2})L_{0}+\tilde{L}_{0}/M_{2}\in\mathcal{L}$ is a strictly feasible solution for the IW-SAA problem, which proves the empirical version of Assumption \ref{interior_point}. Therefore, the above argument and the proof of Theorem \ref{strong_duality} shows the inclusion $A_{n}\subset\{\text{empirical version of Assumptions \ref{integrability_condition}-\ref{interior_point} holds}\}\subset\{val(\mathcal{P}_{n})=val(\mathcal{D}_{n})\}$. Then we deduce that%
\[
\bigcup_{n=1}^{\infty}\bigcap_{k=n}^{\infty}A_{k}\subset\bigcup_{n=1}^{\infty}\bigcap_{k=n}^{\infty}\{val(\mathcal{P}_{k})=val(\mathcal{D}_{k})\}\equiv\{\exists N_{0}\in\mathbb{N}\text{ such that }val(\mathcal{P}_{n})=val(\mathcal{D}_{n}),\forall n\geq N_{0}\},
\]
which implies (because $val(\mathcal{P}_{n})$ may be non-measurable) $\mathbb{P}_{\ast}(\exists N_{0}\in\mathbb{N}\text{ such that }val(\mathcal{P}_{n})=val(\mathcal{D}_{n}),\forall n\geq N_{0})=1$ by (\ref{liminf_An}). Finally, by the weak law of large numbers, $\mathbb{P}(A_{n})\rightarrow1$ which implies $\mathbb{P}_{\ast}(val(\mathcal{P}_{n})=val(\mathcal{D}_{n}))\geq \mathbb{P}(A_{n})\rightarrow1$.

If there is no inequality constraint in the distributional optimization formulation, the above argument still holds by removing the part regarding the inequality constraint. If there is no equality constraint in the distributional optimization formulation, then $L_{0}$ is a strictly feasible solution (when $n$ is large) for the IW-SAA problem by the law of large numbers, which again proves the empirical version of Assumption \ref{interior_point}. This completes our proof.
\end{proof}

Now we prove Theorem \ref{consistency}. By strong duality, it suffices to consider $val(\mathcal{D}_{n})$ and $val(\mathcal{D})$. However, the outer minimization over $\lambda$ in ($\mathcal{D}_{n}$) and ($\mathcal{D}$) is conducted on a unbounded region on which the uniform law of large numbers usually fails. Thus, we first establish the following technical lemma which ensures the minimization over $\lambda$ can be conducted on a compact set under certain conditions and thus quantifies the difference between two Lagrangian dual problems.

\begin{lemma}
\label{deviation_optimal_value}Let $C(\lambda)=\sup_{i\in I}(a_{i}-\lambda^{\top}b_{i})$ and $\tilde{C}(\lambda)=\sup_{i\in I}(c_{i}-\lambda^{\top}d_{i})$ be two functions defined on $\lambda\in\mathbb{R}_{+}^{m}\times\mathbb{R}^{l-m}$ and taking values in $(-\infty,\infty]$, where $I$ is an index set with arbitrary cardinality and $a_{i},c_{i}\in\mathbb{R},b_{i},d_{i}\in\mathbb{R}^{l}$. Assume that (i) $\lambda^{\ast}\in\argmin_{\lambda\in\mathbb{R}_{+}^{m}\times\mathbb{R}^{l-m}}C(\lambda)$ with $C(\lambda^{\ast})\in(-\infty,\infty)$, (ii) $\exists~\delta>0$ s.t. for any $(j_{1},\ldots,j_{l-m})\in\{1,-1\}^{l-m}$, there exists an index $i(j_{1},\ldots,j_{l-m})\in I$ satisfying $b_{i(j_{1},\ldots,j_{l-m}),k}<0,k=1,\ldots,m$ and $b_{i(j_{1},\ldots,j_{l-m}),k}=j_{k-m}\delta,k=m+1,\ldots,l$. (iii) $\sup_{i\in I}|a_{i}-c_{i}|<\infty$ and $\sup_{i\in I}|b_{i}-d_{i}|<\eta_{1}$, where 
\[
\eta_{1}=\min_{(j_{1},\ldots,j_{l-m})\in\{1,-1\}^{l-m}}\min_{1\leq k\leq l}|b_{i(j_{1},\ldots,j_{l-m}),k}|.
\]
Then we have $\inf_{\lambda\in\mathbb{R}_{+}^{m}\times\mathbb{R}^{l-m}}\tilde{C}(\lambda)\in(-\infty,\infty)$ and%
\begin{align*}
\left\vert \inf_{\lambda\in\mathbb{R}_{+}^{m}\times\mathbb{R}^{l-m}}C(\lambda)-\inf_{\lambda\in\mathbb{R}_{+}^{m}\times\mathbb{R}^{l-m}}\tilde{C}(\lambda)\right\vert &\leq\sup_{i\in I}|a_{i}-c_{i}|+M\sup_{i\in I}|b_{i}-d_{i}|\\
&\leq\sup_{i\in I}|a_{i}-c_{i}|+M\sum_{j=1}^{l}\sup_{i\in I}|b_{ij}-d_{ij}|,
\end{align*}
where $M$ is defined as%
\[
M=\max\left(  |\lambda^{\ast}|,\frac{C(\lambda^{\ast})-\eta_{2}+2\sup_{i\in I}|a_{i}-c_{i}|+|\lambda^{\ast}|\sup_{i\in I}|b_{i}-d_{i}|}{\eta_{1}-\sup_{i\in I}|b_{i}-d_{i}|}\right)  ,
\]
and $\eta_{2}=\min_{(j_{1},\ldots,j_{l-m})\in\{1,-1\}^{l-m}}a_{i(j_{1},\ldots,j_{l-m})}.$
\end{lemma}

\begin{proof}[Proof of Lemma \ref{deviation_optimal_value}.]
The proof includes two steps. The first step is to bound $|C(\lambda)-\tilde{C}(\lambda)|$ and $|\inf C(\lambda)-\inf\tilde{C}(\lambda)|$ when the infimum is taken over compact sets of $\lambda$. The second step is to show the optimization problem $\inf_{\lambda\in\mathbb{R}_{+}^{m}\times\mathbb{R}^{l-m}}\tilde{C}(\lambda)$ can be restricted over a compact set.

For any $i\in I$, we have $|(a_{i}-\lambda^{\top}b_{i})-(c_{i}-\lambda^{\top}d_{i})|\leq\sup_{i\in I}|a_{i}-c_{i}|+|\lambda|\sup_{i\in I}|b_{i}-d_{i}|<\infty$, i.e.,%
\begin{align*}
&-\left(  \sup_{i\in I}|a_{i}-c_{i}|+|\lambda|\sup_{i\in I}|b_{i}-d_{i}|\right)  +a_{i}-\lambda^{\top}b_{i}\\
&\leq c_{i}-\lambda^{\top}d_{i}\\
&\leq a_{i}-\lambda^{\top}b_{i}+\left(  \sup_{i\in I}|a_{i}-c_{i}|+|\lambda|\sup_{i\in I}|b_{i}-d_{i}|\right)  .
\end{align*}
By taking the supremum over $i\in I$, we have
\begin{equation}
-\left(  \sup_{i\in I}|a_{i}-c_{i}|+|\lambda|\sup_{i\in I}|b_{i}-d_{i}|\right)  +C(\lambda)\leq\tilde{C}(\lambda)\leq C(\lambda)+\left(\sup_{i\in I}|a_{i}-c_{i}|+|\lambda|\sup_{i\in I}|b_{i}-d_{i}|\right)  .
\label{upper_bound_difference_C's}%
\end{equation}
We can see the inequality (\ref{upper_bound_difference_C's}) holds even if $C(\lambda)=\infty$ or $\tilde{C}(\lambda)=\infty$ (actually it tells us $C(\lambda)=\infty\Leftrightarrow\tilde{C}(\lambda)=\infty$). Applying (\ref{upper_bound_difference_C's}) to $\lambda^{\ast}$, we get%
\begin{equation}
\tilde{C}(\lambda^{\ast})\leq C(\lambda^{\ast})+\sup_{i\in I}|a_{i}-c_{i}|+|\lambda^{\ast}|\sup_{i\in I}|b_{i}-d_{i}|<\infty.
\label{upper_bound_C_tilde_star}%
\end{equation}
When $M\geq|\lambda^{\ast}|$, (\ref{upper_bound_difference_C's}) implies that%
\begin{equation}
\left\vert \inf_{\lambda\in\mathbb{R}_{+}^{m}\times\mathbb{R}^{l-m},|\lambda|\leq M}\tilde{C}(\lambda)-\inf_{\lambda\in\mathbb{R}_{+}^{m}\times\mathbb{R}^{l-m},|\lambda|\leq M}C(\lambda)\right\vert \leq\sup_{i\in I}|a_{i}-c_{i}|+M\sup_{i\in I}|b_{i}-d_{i}|.
\label{upper_bound_difference_infC's}%
\end{equation}
By assumption (i) in this lemma, we know that $\inf_{\lambda\in\mathbb{R}_{+}^{m}\times\mathbb{R}^{l-m},|\lambda|\leq M}C(\lambda)=C(\lambda^{\ast})\in(-\infty,\infty)$. Then it follows from (\ref{upper_bound_difference_infC's}) that $\inf_{\lambda\in\mathbb{R}_{+}^{m}\times\mathbb{R}^{l-m},|\lambda|\leq M}\tilde{C}(\lambda)\in(-\infty,\infty).$

Now, let's find the lower bound for $\tilde{C}(\lambda)$ when $\lambda$ is large. For any $\lambda=(\lambda_{1},\ldots,\lambda_{l})\in\mathbb{R}_{+}^{m}\times\mathbb{R}^{l-m}$, consider $(j_{1},\ldots,j_{l-m})=(-\sign(\lambda_{m+1}),\ldots,-\sign(\lambda_{l}))$, where $\sign(x)$ is the sign function.
% \[
% \sign(x)=\left\{
% \begin{array}[c]{r}%
% 1,\\
% -1,
% \end{array}
% \left.
% \begin{array}[c]{c}%
% \text{if }x\geq0\\
% \text{if }x<0
% \end{array}
% \right.  \right.  .
% \]
By assumption (ii), there exists an index $i(j_{1},\ldots,j_{l-m})\in I$ satisfying $b_{i(j_{1},\ldots,j_{l-m}),k}<0,k=1,\ldots,m$ and $b_{i(j_{1},\ldots,j_{l-m}),k}=j_{k-m}\delta,k=m+1,\ldots,l$. Thus, we have%
\begin{align}
C(\lambda)  &  =\sup_{i\in I}(a_{i}-\lambda^{\top}b_{i})\geq a_{i(j_{1},\ldots,j_{l-m})}-\lambda^{\top}b_{i(j_{1},\ldots,j_{l-m})}\nonumber\\
&  =a_{i(j_{1},\ldots,j_{l-m})}+\sum_{k=1}^{l}|\lambda_{k}b_{i(j_{1},\ldots,j_{l-m}),k}|\geq\eta_{2}+\eta_{1}\sum_{k=1}^{l}|\lambda_{k}|\geq\eta_{2}+\eta_{1}|\lambda|, \label{lower_bound_C}%
\end{align}
where%
\[
\eta_{1}=\min_{(j_{1},\ldots,j_{l-m})\in\{1,-1\}^{l-m}}\min_{1\leq k\leq l}|b_{i(j_{1},\ldots,j_{l-m}),k}|>0,
\]%
\[
\eta_{2}=\min_{(j_{1},\ldots,j_{l-m})\in\{1,-1\}^{l-m}}a_{i(j_{1},\ldots,j_{l-m})}.
\]
By (\ref{upper_bound_difference_C's}) and (\ref{lower_bound_C}), we have
\begin{equation}
\tilde{C}(\lambda)\geq C(\lambda)-\left(  \sup_{i\in I}|a_{i}-c_{i}|+|\lambda|\sup_{i\in I}|b_{i}-d_{i}|\right)  \geq\left(  \eta_{2}-\sup_{i\in I}|a_{i}-c_{i}|\right)  +\left(  \eta_{1}-\sup_{i\in I}|b_{i}-d_{i}|\right)|\lambda|. \label{lower_bound_C_tilde}%
\end{equation}
Combining the bounds in (\ref{upper_bound_C_tilde_star}) and (\ref{lower_bound_C_tilde}), we can see if $\lambda\in\mathbb{R}_{+}^{m}\times\mathbb{R}^{l-m}$ satisfies%
\[
|\lambda|\geq\frac{C(\lambda^{\ast})-\eta_{2}+2\sup_{i\in I}|a_{i}-c_{i}|+|\lambda^{\ast}|\sup_{i\in I}|b_{i}-d_{i}|}{\eta_{1}-\sup_{i\in I}|b_{i}-d_{i}|},
\]
then $\tilde{C}(\lambda)\geq\tilde{C}(\lambda^{\ast})$. Thus, by choosing%
\[
M=\max\left(  |\lambda^{\ast}|,\frac{C(\lambda^{\ast})-\eta_{2}+2\sup_{i\in I}|a_{i}-c_{i}|+|\lambda^{\ast}|\sup_{i\in I}|b_{i}-d_{i}|}{\eta_{1}-\sup_{i\in I}|b_{i}-d_{i}|}\right)  ,
\]
we will get%
\[
\inf_{\lambda\in\mathbb{R}_{+}^{m}\times\mathbb{R}^{l-m},|\lambda|\leq M}C(\lambda)=\inf_{\lambda\in\mathbb{R}_{+}^{m}\times\mathbb{R}^{l-m}}C(\lambda),\quad\inf_{\lambda\in\mathbb{R}_{+}^{m}\times\mathbb{R}^{l-m},|\lambda|\leq M}\tilde{C}(\lambda)=\inf_{\lambda\in\mathbb{R}_{+}^{m}\times\mathbb{R}^{l-m}}\tilde{C}(\lambda),
\]
where the first equality follows from assumption (i) of this lemma and the second one follows from $\tilde{C}(\lambda)\geq\tilde{C}(\lambda^{\ast})$ when $|\lambda|\geq M$. Moreover, (\ref{upper_bound_difference_infC's}) gives us a bound for the difference of the optimal values:%
\begin{align*}
&  \left\vert \inf_{\lambda\in\mathbb{R}_{+}^{m}\times\mathbb{R}^{l-m}}C(\lambda)-\inf_{\lambda\in\mathbb{R}_{+}^{m}\times\mathbb{R}^{l-m}}\tilde{C}(\lambda)\right\vert \\
&  =\left\vert \inf_{\lambda\in\mathbb{R}_{+}^{m}\times\mathbb{R}^{l-m},|\lambda|\leq M}\tilde{C}(\lambda)-\inf_{\lambda\in\mathbb{R}_{+}^{m}\times\mathbb{R}^{l-m},|\lambda|\leq M}\tilde{C}(\lambda)\right\vert \\
&  \leq\sup_{i\in I}|a_{i}-c_{i}|+M\sup_{i\in I}|b_{i}-d_{i}|.
\end{align*}
Finally, note that $\sup_{i\in I}|b_{i}-d_{i}|\leq\sum_{j=1}^{l}\sup_{i\in I}|b_{ij}-d_{ij}|$. This concludes our proof.
\end{proof}

\begin{proof}[Proof of Theorem \ref{consistency}.]
We make use of Lemma \ref{deviation_optimal_value} by taking
\[
C(\lambda)=\sup_{L\in\mathcal{L}}\text{ }\left\{  \mathbb{E}_{g}[\phi_{0}(X)L(X)]-\sum_{j=1}^{l}\lambda_{j}(\mathbb{E}_{g}[\phi_{j}(X)L(X)]-\mu_{j})\right\}  ,
\]%
\[
\tilde{C}(\lambda)=\sup_{L\in\mathcal{L}}\text{ }\frac{1}{n}\sum_{i=1}^{n}\left(  \phi_{0}(X_{i})L(X_{i})-\sum_{j=1}^{l}\lambda_{j}(\phi_{j}(X_{i})L(X_{i})-\mu_{j})\right)  .
\]
By Assumption \ref{GC_class}, $\mathcal{F}_{j}$ have integrable envelopes (say $F_{j}(x)$) and thus $C(\lambda)$ is a real-valued convex function. Moreover, $C(\lambda)$ is a Lipschitz continuous function by the triangular inequality $|C(\lambda)-C(\lambda^{\prime})|\leq\sum_{j=1}^{l}|\lambda_{j}-\lambda_{j}^{\prime}|\mathbb{E}_{g}[F_{j}(X)]$. Now we verify the conditions (i)-(iii) in Lemma \ref{deviation_optimal_value}.

We first verify condition (ii). By Assumption \ref{interior_point}, there exists a feasible function $L_{0}\in\mathcal{L}$ s.t.
\[
\mathbb{E}_{g}[\phi_{j}(X)L_{0}(X)]<\mu_{j},j=1,\ldots,m,\quad \mathbb{E}_{g}[\phi_{j}(X)L_{0}(X)]=\mu_{j},j=m+1,\ldots,l.
\]
Additionally, $\mathbf{0}$ is an interior point of the following set%
\[
\{(\mathbb{E}_{g}[\phi_{m+1}(X)L(X)]-\mu_{m+1},\ldots,\mathbb{E}_{g}[\phi_{l}(X)L(X)]-\mu_{l}):L\in\mathcal{L\}.}%
\]
Therefore, $\exists~\delta_{1}>0$ s.t. for any $(j_{1},\ldots,j_{l-m})\in\{1,-1\}^{l-m}$, there exists $L_{(j_{1},\ldots,j_{l-m})}\in\mathcal{L}$ with $\mathbb{E}_{g}[\phi_{k}(X)L_{(j_{1},\ldots,j_{l-m})}(X)]-\mu_{k}=j_{k-m}\delta_{1},k=m+1,\ldots,l$. So we can choose $\delta_{2}>0$ small enough s.t. for any $(j_{1},\ldots,j_{l-m})\in\{1,-1\}^{l-m}$, the $k$th component of the vector%
% \begin{equation}
% (\mathbb{E}_{g}[\phi_{1}(X)[(1-\delta_{2})L_{0}(X)+\delta_{2}L_{(j_{1},\ldots,j_{l-m})}(X)]]-\mu_{1},\ldots,\mathbb{E}_{g}[\phi_{l}(X)[(1-\delta_{2})L_{0}(X)+\delta_{2}L_{(j_{1},\ldots,j_{l-m})}(X)]]-\mu_{l}) \label{vector_for_assumption2}%
% \end{equation}
\begin{equation}
(\mathbb{E}_{g}[\phi_{i}(X)[(1-\delta_{2})L_{0}(X)+\delta_{2}L_{(j_{1},\ldots,j_{l-m})}(X)]]-\mu_{i})_{i=1}^{l} \label{vector_for_assumption2}%
\end{equation}
is negative for $k=1,\ldots,m$ and equal to $j_{k-m}\delta_{1}\delta_{2}$ for $k=m+1,\ldots,l$. Additionally, $(1-\delta_{2})L_{0}(X)+\delta_{2}L_{(j_{1},\ldots,j_{l-m})}(X)\in\mathcal{L}$ since $\mathcal{L}$ is a convex set. This proves condition (ii).

Next, we verify condition (i). As in the statement of Lemma \ref{deviation_optimal_value}, we define%
\[
\eta_{1}=\min_{(j_{1},\ldots,j_{l-m})\in\{1,-1\}^{l-m}}\min_{1\leq k\leq l}|\mathbb{E}_{g}[\phi_{k}(X)[(1-\delta_{2})L_{0}(X)+\delta_{2}L_{(j_{1},\ldots,j_{l-m})}(X)]]-\mu_{k}|>0,
\]%
\[
\eta_{2}=\min_{(j_{1},\ldots,j_{l-m})\in\{1,-1\}^{l-m}}\mathbb{E}_{g}[\phi_0(X)[(1-\delta_{2})L_{0}(X)+\delta_{2}L_{(j_{1},\ldots,j_{l-m})}(X)]].
\]
Then according to (\ref{lower_bound_C}), we have $C(\lambda)\geq\eta_{2}+\eta_{1}|\lambda|$, which implies $\lim_{|\lambda|\rightarrow\infty}C(\lambda)=\infty$. By strong duality and Assumption \ref{finite_opt}, we know that $\inf_{\lambda\in\mathbb{R}_{+}^{m}\times\mathbb{R}^{l-m}}C(\lambda)\in\mathbb{R}$, which implies the existence of a minimizer $\lambda^{\ast}$ due to $\lim_{|\lambda|\rightarrow\infty}C(\lambda)=\infty$ and continuity of $C(\lambda)$.

Finally, we verify condition (iii) along with the main argument for this proposition. Notice that under our choice of $C(\lambda)$ and $\tilde{C}(\lambda)$, we actually have%
\[
\sup_{i\in I}|a_{i}-c_{i}|=||P_{n}-P_{g}||_{\mathcal{F}_{0}}\leq||P_{n}-P_{g}||_{\mathcal{F}_{0}}^{\ast},
\]%
\[
\sup_{i\in I}|b_{i}-d_{i}|\leq\sum_{j=1}^{l}\sup_{i\in I}|b_{ij}-d_{ij}|=\sum_{j=1}^{l}||P_{n}-P_{g}||_{\mathcal{F}_{j}}\leq\sum_{j=1}^{l}||P_{n}-P_{g}||_{\mathcal{F}_{j}}^{\ast}.
\]
We consider the event%
\[
\left\{  \max_{0\leq j\leq l}||P_{n}-P_{g}||_{\mathcal{F}_{j}}^{\ast}\rightarrow0\right\}\bigcap\{\exists N_{0}\in\mathbb{N}\text{ s.t. }val(\mathcal{P}_{n})=val(\mathcal{D}_{n}),\forall n\geq N_{0}\}.
\]
On this event, when $n$ is large enough, we have $val(\mathcal{P}_{n})=val(\mathcal{D}_{n})$ and $\max_{0\leq j\leq l}||P_{n}-P_{g}||_{\mathcal{F}_{j}}^{\ast}\leq\eta_{1}/(l+1)$, which implies condition (iii) in Lemma \ref{deviation_optimal_value}. Then by Lemma \ref{deviation_optimal_value}, we obtain that, on the above event and for $n$ large enough,
% (Zhenyuan: make some changes)}
\begin{align}
&  |val(\mathcal{P}_{n})-val(\mathcal{P})|\nonumber\\
&  =|val(\mathcal{D}_{n})-val(\mathcal{D})|\nonumber\\
&  \leq||P_{n}-P_{g}||_{\mathcal{F}_{0}}^{\ast}+M\sum_{j=1}^{l}||P_{n}-P_{g}||_{\mathcal{F}_{j}}^{\ast}\nonumber\\
&  \leq||P_{n}-P_{g}||_{\mathcal{F}_{0}}^{\ast}\nonumber\\
&+\max\left(  |\lambda^{\ast}|,\frac{val(\mathcal{P})-\eta_{2}+2||P_{n}-P_{g}||_{\mathcal{F}_{0}}^{\ast}+|\lambda^{\ast}|\sum_{j=1}^{l}||P_{n}-P_{g}||_{\mathcal{F}_{j}}^{\ast}}{\eta_{1}-\sum_{j=1}^{l}||P_{n}-P_{g}||_{\mathcal{F}_{j}}^{\ast}}\right)\sum_{j=1}^{l}||P_{n}-P_{g}||_{\mathcal{F}_{j}}^{\ast}.
\label{bound_opt_difference}%
\end{align}
Since the quantity in (\ref{bound_opt_difference}) is measurable and converges to $0$, we obtain that $|val(\mathcal{P}_{n})-val(\mathcal{P})|^{\ast}\rightarrow0$. Hence, we have proved that%
\begin{align*}
&\left\{  \max_{0\leq j\leq l}||P_{n}-P_{g}||_{\mathcal{F}_{j}}^{\ast}\rightarrow0\right\}\bigcap\{\exists N_{0}\in\mathbb{N}\text{ s.t. }val(\mathcal{P}_{n})=val(\mathcal{D}_{n}),\forall n\geq N_{0}\}\\
\subset&\{|val(\mathcal{P}_{n})-val(\mathcal{P})|^{\ast}\rightarrow0\}.
\end{align*}
Finally, we notice that
\begin{align*}
&  \mathbb{P}_{\ast}\left(  \left\{  \max_{0\leq j\leq l}||P_{n}-P_{g}||_{\mathcal{F}_{j}}^{\ast}\rightarrow0\right\}\bigcap\{\exists N_{0}\in\mathbb{N}\text{ s.t. }val(\mathcal{P}_{n})=val(\mathcal{D}_{n}),\forall n\geq N_{0}\}\right) \\
&  =1-\mathbb{P}^{\ast}\left(  \left\{  \max_{0\leq j\leq l}||P_{n}-P_{g}||_{\mathcal{F}_{j}}^{\ast}\rightarrow0\right\}  ^{c}\bigcup\{\exists N_{0}\in\mathbb{N}\text{ s.t. }val(\mathcal{P}_{n})=val(\mathcal{D}_{n}),\forall n\geq N_{0}\}^{c}\right) \\
&  \geq1-\mathbb{P}^{\ast}\left(  \left\{  \max_{0\leq j\leq l}||P_{n}-P_{g}||_{\mathcal{F}_{j}}^{\ast}\rightarrow0\right\}  ^{c}\right)  -\mathbb{P}^{\ast}(\{\exists N_{0}\in\mathbb{N}\text{ s.t. }val(\mathcal{P}_{n})=val(\mathcal{D}_{n}),\forall n\geq N_{0}\}^{c})\\
&  =1,
\end{align*}
where we make use of $\max_{0\leq j\leq l}||P_{n}-P_{g}||_{\mathcal{F}_{j}}^{\ast}\rightarrow0$ almost surely by Assumption \ref{GC_class} and the strong duality result in Theorem \ref{strong_duality_SAA}. Therefore, we obtain $|val(\mathcal{P}_{n})-val(\mathcal{P})|^{\ast}\rightarrow0$ almost surely, i.e., $|val(\mathcal{P}_{n})-val(\mathcal{P})|\overset{\text{a.s.*}}{\rightarrow}0.$ 
\end{proof}

\begin{proof}[Proof of Theorem \ref{convergence_rate}.]
We follow the proof of Theorem \ref{consistency} but modify the argument starting from the verification of condition (iii) of Lemma \ref{deviation_optimal_value}. We define the set $A_{n}=\{  \max_{0\leq j\leq l}||P_{n}-P||_{\mathcal{F}_{j}}<\eta_{1}/l+1\}\bigcap\{val(\mathcal{P}_{n})=val(\mathcal{D}_{n})\}$. Assumption \ref{GC_class} and Theorem \ref{strong_duality_SAA} ensure $\mathbb{P}^{\ast}(A_{n}^{c})\rightarrow0$ as $n\rightarrow\infty$. Now condition (iii) holds on $A_{n}$. By Lemma \ref{deviation_optimal_value}, the following holds on $A_{n}$:
\[
|val(\mathcal{P}_{n})-val(\mathcal{P})|=|val(\mathcal{D}_{n})-val(\mathcal{D})|\leq||P_{n}-P_{g}||_{\mathcal{F}_{0}}+M\sum_{j=1}^{l}||P_{n}-P_{g}||_{\mathcal{F}_{j}},
\]
where $M$ in Lemma \ref{deviation_optimal_value} is upper bounded by
\[
M\leq\max\left(  |\lambda^{\ast}|,\frac{(l+1)(val(\mathcal{P})-\eta_{2})+2\eta_{1}+|\lambda^{\ast}|l\eta_{1}}{\eta_{1}}\right)  :=\tilde{M}<\infty.
\]
So on $A_{n}$, we will have
\[
\sqrt{n}|val(\mathcal{P}_{n})-val(\mathcal{P})|\leq\sqrt{n}||P_{n}-P_{g}||_{\mathcal{F}_{0}}+\tilde{M}\sum_{j=1}^{l}\sqrt{n}||P_{n}-P_{g}||_{\mathcal{F}_{j}}=||G_{n}||_{\mathcal{F}_{0}}+\tilde{M}\sum_{j=1}^{l}||G_{n}||_{\mathcal{F}_{j}},
\]
i.e.,%
\[
A_{n}\subset\left\{  \sqrt{n}|val(\mathcal{P}_{n})-val(\mathcal{P})|\leq||G_{n}||_{\mathcal{F}_{0}}+\tilde{M}\sum_{j=1}^{l}||G_{n}||_{\mathcal{F}_{j}}\right\}  .
\]
Now, for any $\varepsilon>0$, we choose $K=C(1+\tilde{M}l)/\varepsilon$, where $C$ is the constant in Assumption \ref{maximal_inequality}. Then we can see%
\begin{align*}
&  \mathbb{P}^{\ast}(\sqrt{n}|val(\mathcal{P}_{n})-val(\mathcal{P})|\geq K)\\
&  \leq \mathbb{P}^{\ast}(\{\sqrt{n}|val(\mathcal{P}_{n})-val(\mathcal{P})|\geq K\}\cap A_{n})+\mathbb{P}^{\ast}(\{\sqrt{n}|val(\mathcal{P}_{n})-val(\mathcal{P})|\geq K\}\cap A_{n}^{c})\\
&  \leq \mathbb{P}^{\ast}\left(  \{\sqrt{n}|val(\mathcal{P}_{n})-val(\mathcal{P})|\geq K\}\cap\left\{  \sqrt{n}|val(\mathcal{P}_{n})-val(\mathcal{P})|\leq ||G_{n}||_{\mathcal{F}_{0}}+\tilde{M}\sum_{j=1}^{l}||G_{n}||_{\mathcal{F}_{j}}\right\}  \right) \\
& +\mathbb{P}^{\ast}(A_{n}^{c})\\
&  \leq \mathbb{P}^{\ast}\left(  ||G_{n}||_{\mathcal{F}_{0}}+\tilde{M}\sum_{j=1}^{l}||G_{n}||_{\mathcal{F}_{j}}\geq K\right)  +\mathbb{P}^{\ast}(A_{n}^{c})\\
&  \leq\frac{\mathbb{E}_{P}^{\ast}||G_{n}||_{\mathcal{F}_{0}}+\tilde{M}\sum_{j=1}^{l}\mathbb{E}_{P}^{\ast}||G_{n}||_{\mathcal{F}_{j}}}{K}+\mathbb{P}^{\ast}(A_{n}^{c})\\
&  \leq\varepsilon+\mathbb{P}^{\ast}(A_{n}^{c}).
\end{align*}
By taking the limsup, we finally get%
\[
\limsup_{n\rightarrow\infty}\mathbb{P}^{\ast}(\sqrt{n}|val(\mathcal{P}_{n})-val(\mathcal{P})|\geq K)\leq \varepsilon+\limsup_{n\rightarrow\infty}\mathbb{P}^{\ast}(A_{n}^{c})=\varepsilon.
\]

\end{proof}

\begin{proof}[Proof of Lemma \ref{general_verification}.]
We first consider the verification of Assumption \ref{GC_class}. It suffices to show each $\mathcal{F}_{j}$ has an integrable envelope and $N_{[~]}(\varepsilon,\mathcal{F}_{j},||\cdot||_{1})<\infty$ for any $\varepsilon>0$ by \cite{van1996weak} Theorem 2.4.1. Suppose (\ref{compact_applicable1}) holds. Then $\mathcal{F}_{j}$ has an integrable envelope $||F||_{\infty}|\phi_{j}(x)|/g(x)$. Next, we show $N_{[~]}(\varepsilon,\mathcal{F}_{j},||\cdot||_{1})<\infty$ for any $\varepsilon>0$. We take $\varepsilon/||\phi_{j}/g||_{1+\delta}$-brackets $[l_{i},u_{i}],i=1,\ldots,N=N_{[~]}(\varepsilon/||\phi_{j}/g||_{1+\delta},\mathcal{F},||\cdot||_{1+1/\delta})$
% \DSQ{[Not sure what distinguishes/defines the brackets with respect to $i$? Zhenyuan: This sentence means that we take $N_{[~]}(\varepsilon/||\phi_{j}/g||_{1+\delta},\mathcal{F},||\cdot||_{1+1/\delta})$ brackets indexed by $i$]} 
that cover $(\mathcal{F},||\cdot||_{1+1/\delta})$. Without loss of generality, we can assume $0\leq l_{i}\leq u_{i}\leq||F||_{\infty}$. For each $f\in\mathcal{F}$, $\exists~[l_{i},u_{i}]$ s.t. $l_{i}\leq f\leq u_{i}$. Let $\phi_{j}^{+}$ and $\phi_{j}^{-}$ be the positive and negative part of $\phi_{j}$, i.e., $\phi_{j}^{+}=\max\{\phi_{j},0\}$ and $\phi_{j}^{-}=\max\{-\phi_{j},0\}$. Then we can see%
\[
\frac{l_{i}}{g}\phi_{j}^{+}-\frac{u_{i}}{g}\phi_{j}^{-}\leq\frac{f}{g}\phi_{j}\leq\frac{u_{i}}{g}\phi_{j}^{+}-\frac{l_{i}}{g}\phi_{j}^{-}.
\]
Therefore, $[l_{i}\phi_{j}^{+}/g-u_{i}\phi_{j}^{-}/g,u_{i}\phi_{j}^{+}/g-l_{i}\phi_{j}^{-}/g],i=1,\ldots,N$ forms a cover of $(\mathcal{F}_{j},||\cdot||_{1})$. Moreover, they satisfy%
\[
\left\vert \frac{u_{i}}{g}\phi_{j}^{+}-\frac{l_{i}}{g}\phi_{j}^{-}\right\vert\leq\frac{||F||_{\infty}}{g}\phi_{j}^{+}+\frac{||F||_{\infty}}{g}\phi_{j}^{-}=||F||_{\infty}\frac{|\phi_{j}|}{g}\Rightarrow\left\Vert \frac{u_{i}}{g}\phi_{j}^{+}-\frac{l_{i}}{g}\phi_{j}^{-}\right\Vert _{1}<\infty,
\]%
\[
\left\vert \frac{l_{i}}{g}\phi_{j}^{+}-\frac{u_{i}}{g}\phi_{j}^{-}\right\vert\leq\frac{||F||_{\infty}}{g}\phi_{j}^{+}+\frac{||F||_{\infty}}{g}\phi_{j}^{-}=||F||_{\infty}\frac{|\phi_{j}|}{g}\Rightarrow\left\Vert \frac{l_{i}}{g}\phi_{j}^{+}-\frac{u_{i}}{g}\phi_{j}^{-}\right\Vert _{1}<\infty,
\]%
\begin{align*}
&  \left\vert \frac{u_{i}}{g}\phi_{j}^{+}-\frac{l_{i}}{g}\phi_{j}^{-}-\frac{l_{i}}{g}\phi_{j}^{+}+\frac{u_{i}}{g}\phi_{j}^{-}\right\vert=(u_{i}-l_{i})\frac{|\phi_{j}|}{g}\\
&  \Rightarrow\left\Vert \frac{u_{i}}{g}\phi_{j}^{+}-\frac{l_{i}}{g}\phi_{j}^{-}-\frac{l_{i}}{g}\phi_{j}^{+}+\frac{u_{i}}{g}\phi_{j}^{-}\right\Vert_{1}\leq||u_{i}-l_{i}||_{1+1/\delta}||\phi_{j}/g||_{1+\delta}<\varepsilon.
\end{align*}
So they are $\varepsilon$-brackets for $(\mathcal{F}_{j},||\cdot||_{1})$, which proves%
\[
N_{[~]}(\varepsilon,\mathcal{F}_{j},||\cdot||_{1})\leq N_{[~]}(\varepsilon/||\phi_{j}/g||_{1+\delta},\mathcal{F},||\cdot||_{1+1/\delta})<\infty,\forall\varepsilon>0.
\]

Next we assume (\ref{unbounded_applicable1}) holds and prove $\mathcal{F}_{j}$ has an integrable envelope and $N_{[~]}(\varepsilon,\mathcal{F}_{j},||\cdot||_{1})<\infty$ for any $\varepsilon>0$. Note that
\[
\left\vert \frac{f(x)}{g(x)}\phi_{j}(x)\right\vert \leq||F||_{\infty}^{\gamma}\frac{F^{1-\gamma}(x)}{g(x)}|\phi_{j}(x)|\leq||F||_{\infty}^{\gamma}M|\phi_{j}(x)|,\forall f\in\mathcal{F}.
\]
So $\mathcal{F}_{j}$ has an integrable envelope $||F||_{\infty}^{\gamma}M|\phi_{j}(x)|$. To show $N_{[~]}(\varepsilon,\mathcal{F}_{j},||\cdot||_{1})<\infty$ for any $\varepsilon>0$, we take $\varepsilon\gamma/(M||\phi_{j}||_{1+\delta})$-brackets $[l_{i},u_{i}],i=1,\ldots,N=N_{[~]}(\varepsilon\gamma/(M||\phi_{j}||_{1+\delta}),\mathcal{F}^{\gamma},||\cdot||_{1+1/\delta})$ that covers $(\mathcal{F}^{\gamma},||\cdot||_{1+1/\delta})$. Without loss of generality, we can assume $0\leq l_{i}\leq u_{i}\leq F^{\gamma}$. For each $f\in\mathcal{F}$, $\exists~[l_{i},u_{i}]$ s.t. $l_{i}\leq f^{\gamma}\leq u_{i}\Leftrightarrow l_{i}^{1/\gamma}\leq f\leq u_{i}^{1/\gamma}$. Then we obtain%
\[
\frac{l_{i}^{1/\gamma}}{g}\phi_{j}^{+}-\frac{u_{i}^{1/\gamma}}{g}\phi_{j}^{-}\leq\frac{f}{g}\phi_{j}\leq\frac{u_{i}^{1/\gamma}}{g}\phi_{j}^{+}-\frac{l_{i}^{1/\gamma}}{g}\phi_{j}^{-}.
\]
Therefore, $[l_{i}^{1/\gamma}\phi_{j}^{+}/g-u_{i}^{1/\gamma}\phi_{j}^{-}/g,u_{i}^{1/\gamma}\phi_{j}^{+}/g-l_{i}^{1/\gamma}\phi_{j}^{-}/g],i=1,\ldots,N$ forms a cover of $(\mathcal{F}_{j},||\cdot||_{1})$. Moreover, these brackets satisfy%
\[
\left\vert \frac{u_{i}^{1/\gamma}}{g}\phi_{j}^{+}-\frac{l_{i}^{1/\gamma}}{g}\phi_{j}^{-}\right\vert \leq\frac{F}{g}\phi_{j}^{+}+\frac{F}{g}\phi_{j}^{-}\leq||F||_{\infty}^{\gamma}M|\phi_{j}|\Rightarrow\left\Vert \frac{u_{i}^{1/\gamma}}{g}\phi_{j}^{+}-\frac{l_{i}^{1/\gamma}}{g}\phi_{j}^{-}\right\Vert _{1}<\infty,
\]%
\[
\left\vert \frac{l_{i}^{1/\gamma}}{g}\phi_{j}^{+}-\frac{u_{i}^{1/\gamma}}{g}\phi_{j}^{-}\right\vert \leq\frac{F}{g}\phi_{j}^{+}+\frac{F}{g}\phi_{j}^{-}\leq||F||_{\infty}^{\gamma}M|\phi_{j}|\Rightarrow\left\Vert \frac{l_{i}^{1/\gamma}}{g}\phi_{j}^{+}-\frac{u_{i}^{1/\gamma}}{g}\phi_{j}^{-}\right\Vert _{1}<\infty,
\]%
\begin{align*}
&  \left\vert \frac{u_{i}^{1/\gamma}}{g}\phi_{j}^{+}-\frac{l_{i}^{1/\gamma}}{g}\phi_{j}^{-}-\frac{l_{i}^{1/\gamma}}{g}\phi_{j}^{+}+\frac{u_{i}^{1/\gamma}}{g}\phi_{j}^{-}\right\vert =\frac{(u_{i}^{1/\gamma}-l_{i}^{1/\gamma})|\phi_{j}|}{g}\leq\frac{(u_{i}-l_{i})F^{1-\gamma}|\phi_{j}|}{\gamma g}\leq\frac{M(u_{i}-l_{i})|\phi_{j}|}{\gamma}\\
&  \Rightarrow\left\Vert \frac{u_{i}^{1/\gamma}}{g}\phi_{j}^{+}-\frac{l_{i}^{1/\gamma}}{g}\phi_{j}^{-}-\frac{l_{i}^{1/\gamma}}{g}\phi_{j}^{+}+\frac{u_{i}^{1/\gamma}}{g}\phi_{j}^{-}\right\Vert _{1}\leq\frac{M}{\gamma}||u_{i}-l_{i}||_{1+1/\delta}||\phi_{j}||_{1+\delta}<\varepsilon,
\end{align*}
where the third one uses the mean value theorem $u_{i}^{1/\gamma}-l_{i}^{1/\gamma}=(\xi^{1/\gamma-1}/\gamma)(u_{i}-l_{i})\leq(F^{1-\gamma}/\gamma)(u_{i}-l_{i})$ for some $\xi$ with $l_{i}\leq\xi\leq u_{i}\leq F^{\gamma}$. So they are $\varepsilon$-brackets for $(\mathcal{F}_{j},||\cdot||_{1})$, which proves%
\[
N_{[~]}(\varepsilon,\mathcal{F}_{j},||\cdot||_{1})\leq N_{[~]}(\varepsilon\gamma/(M||\phi_{j}||_{1+\delta}),\mathcal{F}^{\gamma},||\cdot||_{1+1/\delta})<\infty,\forall\varepsilon>0.
\]

Now we prove Assumption \ref{maximal_inequality}. By Corollary 19.35 in \cite{van2000asymptotic}, it suffices to show for each $\mathcal{F}_{j}$, it has an envelope $F_{j}$ with $||F_{j}||_{2}<\infty$ and
\begin{equation}
\int_{0}^{||F_{j}||_{2}}\sqrt{\log N_{[~]}(\varepsilon,\mathcal{F}_{j},||\cdot||_{2})}d\varepsilon<\infty. \label{bracketing_entropy_condition}%
\end{equation}
When (\ref{compact_applicable2}) holds, by slightly modifying the first proof of Assumption \ref{GC_class}, we can see $\mathcal{F}_{j}$ has an envelope with finite $L_{2}$-norm and
\[
N_{[~]}(\varepsilon,\mathcal{F}_{j},||\cdot||_{2})\leq N_{[~]}(\varepsilon/||\phi_{j}/g||_{2+\delta},\mathcal{F},||\cdot||_{2(2+\delta)/\delta})<\infty,\forall\varepsilon>0,
\]
which proves (\ref{bracketing_entropy_condition}). When (\ref{unbounded_applicable2}) holds, by slightly modifying the second proof of Assumption \ref{GC_class}, we can still see $\mathcal{F}_{j}$ has an envelope with finite $L_{2}$-norm and%
\[
N_{[~]}(\varepsilon,\mathcal{F}_{j},||\cdot||_{2})\leq N_{[~]}(\varepsilon\gamma/(M||\phi_{j}||_{2+\delta}),\mathcal{F}^{\gamma},||\cdot||_{2(2+\delta)/\delta})<\infty,\forall\varepsilon>0,
\]
which again proves (\ref{bracketing_entropy_condition}).
\end{proof}

\begin{proof}[Proof of Corollary \ref{1D_compact}.]
We first prove the statistical guarantee. Assumptions \ref{integrability_condition}, \ref{convex_feasible_set} and \ref{finite_opt} clearly hold. Notice that the functions in all of the three classes have variation bounded by $2M$. By the bracketing number bound for functions of bounded variation (see \cite{van2000asymptotic} Example 19.11 where is a small typo: $L_{2}(P)$ should be $L_{r}(P)$) and assumptions on $\phi_{j}$ and $g$, we can see conditions (\ref{compact_applicable1}) and (\ref{compact_applicable2}) hold, which implies Assumptions \ref{GC_class} and \ref{maximal_inequality} by Lemma \ref{general_verification}. Therefore, the desired statistical guarantee is ensured by Theorems \ref{consistency} and \ref{convergence_rate}.
%By the bound of the bracketing number for the bounded monotone functions (see \cite{van1996weak} Theorem 2.7.5) 

Now we show ($\mathcal{P}_{n}$) is equivalent to the linear program. It suffices to show a finite-dimensional reduction in each case, i.e., the values $f(X_{i})=L(X_{i})g(X_{i})$ satisfying the discrete shape constraint can be extended to a function $f\in\mathcal{F}$. For monotonicity, this is achieved by the following step function
\[
f(x)=\left\{
\begin{array}[c]{l}%
L(X_{(1)})g(X_{(1)})\\
L(X_{(i+1)})g(X_{(i+1)})\\
0
\end{array}
\left.
\begin{array}[c]{l}%
\text{if }a\leq x\leq X_{(1)}\\
\text{if }X_{(i)}<x\leq X_{(i+1)},i=1,\ldots,n-1\\
\text{if }X_{(n)}<x\leq b.
\end{array}
\right.  \right.  
\]
For convexity, this is achieved by linear interpolation. For unimodality, this is achieved by the following step function
\[
f(x)=\left\{
\begin{array}[c]{l}%
L(X_{(i)})g(X_{(i)})\\
L(X_{(i_{0})})g(X_{(i_{0})})\\
\max\{L(X_{(i_{0})})g(X_{(i_{0})}),L(X_{(i_{0}+1)})g(X_{(i_{0}+1)})\}\\
L(X_{(i_{0}+1)})g(X_{(i_{0}+1)})\\
L(X_{(i+1)})g(X_{(i+1)})
\end{array}
\left.
\begin{array}[c]{l}%
\text{if }0\leq i<i_{0}\text{ and }X_{(i)}\leq x<X_{(i+1)}\\
\text{if }X_{(i_{0})}\leq x<c\\
\text{if }x=c\\
\text{if }c<x\leq X_{(i_{0}+1)}\\
\text{if }i_{0}<i\leq n\text{ and }X_{(i)}<x\leq X_{(i+1),}%
\end{array}
\right.  \right.  
\]
where $X_{(0)}:=a$, $X_{(n+1)}:=b$, $L(X_{(0)}):=0$, $L(X_{(n+1)}):=0$ and $i_{0}\in\{0,\ldots,n\}$ is the unique index s.t. $X_{(i_{0})}\leq c<X_{(i_{0}+1)}$ (note that with probability 1, the order statistics satisfy the strict inequality $X_{(0)}<X_{(1)}<\cdots<X_{(n)}<X_{(n+1)}$). This concludes our proof.
\end{proof}

\begin{proof}[Proof of Corollary \ref{OU_compact}.]
We first prove the statistical guarantee. Assumptions \ref{integrability_condition}, \ref{convex_feasible_set} and \ref{finite_opt} are obvious. By the finiteness of the bracketing number for the orthounimodal functions (see \cite{gao2007entropy} Corollary 1.3) and assumptions on $\phi_{j}$ and $g$, we can see condition (\ref{compact_applicable1}) holds, which implies Assumptions \ref{GC_class} by Lemma \ref{general_verification}. Therefore, consistency is ensured by Theorems \ref{consistency}.

Next we show ($\mathcal{P}_{n}$) is equivalent to the linear program. It suffices to show the discrete values $f(X_{i})=L(X_{i})g(X_{i})$ satisfying
\begin{align*}
& L(X_{i})g(X_{i})\leq L(X_{j})g(X_{j}),\text{ if }X_{i}\geq X_{j}\text{ component-wise}\\
& 0\leq L(X_{i})g(X_{i})\leq M,i=1,\ldots,n
\end{align*}
can be extended to a function $f\in\mathcal{F}$. We define a function $f$ on $\mathcal{X}$\ by%
\[
f(x)=\left\{
\begin{array}[c]{l}%
M\\
\min_{\{i:X_{i}\in\lbrack a,x]\}}L(X_{i})g(X_{i})
\end{array}
\left.
\begin{array}[c]{l}%
\text{if }[a,x]\text{ does not contain any }X_{i}\\
\text{otherwise.}%
\end{array}
\right.  \right.  
\]
We can see $f\in\mathcal{F}$ and it has the value $L(X_{i})g(X_{i})$ at $x=X_{i}$. This concludes our proof.
\end{proof}

\begin{proof}[Proof of Corollary \ref{monotone_unbounded}.]
We take $F(x)=(Mg_{0}(x))^{1/(1-\gamma)}$ as the envelope of $\mathcal{F}$. Notice that $\mathcal{F}^{\gamma}$ are a subset of non-increasing functions taking values in $[0,g_{0}^{\gamma}(a)]$. By the bound of the bracketing number for bounded monotone functions (see \cite{van1996weak} Theorem 2.7.5) as well as the assumptions in this corollary, we can see (\ref{unbounded_applicable1}) and (\ref{unbounded_applicable2}) hold. Therefore, Assumptions \ref{GC_class} and \ref{maximal_inequality} hold. Assumptions \ref{integrability_condition}, \ref{convex_feasible_set} and \ref{finite_opt} are trivial to check. So the desired statistical guarantee is ensured by Theorems \ref{consistency} and \ref{convergence_rate}. Additionally, we notice that the function $f$ defined in the proof of Corollary \ref{1D_compact} (with $b=\infty$) extends the discrete values to a function in $\mathcal{F}$ due to the non-increasing property of $g_{0}$. This proves the equivalence of ($\mathcal{P}_{n}$) and the linear program.
\end{proof}

\begin{proof}[Proof of Corollary \ref{unimodal_unbounded}.]
With the bound of the bracketing number for the functions of bounded variation (see \cite{van2000asymptotic} Example 19.11 where is a small typo: $L_{2}(P)$ should be $L_{r}(P)$), the proof of the statistical guarantee in Corollary \ref{monotone_unbounded} still applies. Additionally, we notice that the function $f$ defined in the proof of Corollary \ref{1D_compact} (with $a=-\infty,b=\infty$) extends the discrete values to a function in $\mathcal{F}$ due to the unimodality of $g_{0}$. This proves the equivalence of ($\mathcal{P}_{n}$) and the linear program.
\end{proof}

\begin{proof}[Proof of Corollary \ref{OU_unbounded}.]
We first show $N_{[~]}(\varepsilon,\mathcal{F}^{\gamma},||\cdot||_{1+1/\delta})<\infty,\forall\varepsilon>0$. Fix $\varepsilon>0$. We take $b\geq a,b\in\mathbb{R}^{d}$ such that $\int_{[a,\infty)\backslash\lbrack a,b]}g_{0}(x)dx\leq\varepsilon$. We write $\mathcal{F}^{\gamma}|_{[a,b]}$ as the class of functions in $\mathcal{F}^{\gamma}$ that are restricted to the domain $[a,b]$. The functions in $\mathcal{F}^{\gamma}|_{[a,b]}$ are still bounded orthounimodal functions. By Theorem 1.1 in \cite{gao2007entropy}, there are brackets $[l_{i},u_{i}],i=1,\ldots,N(\varepsilon)<\infty$ on $[a,b]$ such that they are $\varepsilon$-brackets covering $\mathcal{F}^{\gamma}|_{[a,b]}$ under the Lebesgue $L_{1+1/\delta}$-norm. For each $l_{i}$ and $u_{i}$, we extend them to $[a,\infty)$ by defining%
\[
\tilde{l}_{i}(x)=\left\{
\begin{array}[c]{l}%
l_{i}(x)\\
0
\end{array}
\left.
\begin{array}[c]{l}%
\text{if }x\in\lbrack a,b]\\
\text{otherwise}%
\end{array}
\right.  \right.  ,\quad\tilde{u}_{i}(x)=\left\{
\begin{array}[c]{l}%
u_{i}(x)\\
g_{0}(a)
\end{array}
\left.
\begin{array}[c]{l}%
\text{if }x\in\lbrack a,b]\\
\text{otherwise.}%
\end{array}
\right.  \right.
\]
We can see $[\tilde{l}_{i},\tilde{u}_{i}],i=1,\ldots,N(\varepsilon)$ cover $\mathcal{F}^{\gamma}$ and the size of the brackets satisfy%
\begin{align*}
&  \left(  \int_{\lbrack a,\infty)}|\tilde{u}_{i}-\tilde{l}_{i}|^{1+1/\delta}g_{0}dx\right)  ^{1/(1+1/\delta)}\\
&  =\left(  \int_{[a,b]}|u_{i}-l_{i}|^{1+1/\delta}g_{0}dx+g_{0}^{1+1/\delta}(a)\int_{[a,\infty)\backslash\lbrack a,b]}g_{0}dx\right)  ^{1/(1+1/\delta)}\\
&  \leq\left(  g_{0}(a)\int_{[a,b]}|u_{i}-l_{i}|^{1+1/\delta}dx+g_{0}^{1+1/\delta}(a)\varepsilon\right)  ^{1/(1+1/\delta)}\\
&  \leq\left(  g_{0}(a)\varepsilon^{1+1/\delta}+g_{0}^{1+1/\delta}(a)\varepsilon\right)  ^{1/(1+1/\delta)}\rightarrow0
\end{align*}
as $\varepsilon\rightarrow0$. This proves $N_{[~]}(\varepsilon,\mathcal{F}^{\gamma},||\cdot||_{1+1/\delta})<\infty,\forall\varepsilon>0$. Therefore, condition (\ref{unbounded_applicable1}) holds, which implies Assumption \ref{GC_class}. Then consistency follows from Theorem \ref{consistency} (the other assumptions are easy to check).

Next we show ($\mathcal{P}_{n}$) is equivalent to the linear program. It suffices to show the discrete values $f(X_{i})=L(X_{i})g(X_{i})$ satisfying
\begin{align*}
& L(X_{i})g(X_{i})\leq L(X_{j})g(X_{j}),\text{ if }X_{i}\geq X_{j}\text{ component-wise}\\
& 0\leq L(X_{i})g(X_{i})\leq(Mg_{0}(X_{i}))^{1/(1-\gamma)},i=1,\ldots,n
\end{align*}
can be extended to a function $f\in\mathcal{F}$. We define a function $f$ on $\mathcal{X}$\ by%
\[
f(x)=\left\{
\begin{array}[c]{l}%
(Mg_{0}(x))^{1/(1-\gamma)}\\
\min\{(Mg_{0}(x))^{1/(1-\gamma)},\min_{\{i:X_{i}\in\lbrack a,x]\}}L(X_{i})g(X_{i})\}
\end{array}
\left.
\begin{array}[c]{l}
\text{if }[a,x]\text{ does not contain any }X_{i}\\
\text{otherwise}%
\end{array}
\right.  \right.  .
\]
We can see $f\in\mathcal{F}$ and it has the value $L(X_{i})g(X_{i})$ at $x=X_{i}$. This concludes our proof.
\end{proof}

\begin{proof}[Proof of Proposition \ref{failure_alpha}.]
We first consider $\alpha<d$. To show for any $p\geq1$, $N_{[~]}(\varepsilon,\mathcal{F},||\cdot||_{p})=\infty$ for all sufficiently small $\varepsilon>0$, it suffices to show $\mathcal{F}$ is not $P_{g}$-GC by \cite{van1996weak} Theorem 2.4.1. Fix a realization of $n$ i.i.d. samples $X_{1},\ldots,X_{n}$ from $g$. We define a function $f$ on $\mathcal{X}$ as%
\begin{equation}
f(x)=\left\{
\begin{array}[c]{l}%
1/|x-a|^{d-\alpha}\\
0
\end{array}
\left.
\begin{array}[c]{l}%
\text{if }x\neq a\text{ and }a,x,X_{i}\text{ are collinear for some }X_{i}\\
\text{otherwise,}%
\end{array}
\right.  \right.   \label{counterexample1}%
\end{equation}
where $|\cdot|$ is the usual Euclidean norm. We can see $f\in\mathcal{F}$ but $f=0$ almost everywhere so that $P_{g}f=0$. Therefore,%
\[
||P_{n}-P_{g}||_{\mathcal{F}}\geq\left\vert \frac{1}{n}\sum_{i=1}^{n}f(X_{i})\right\vert =\frac{1}{n}\sum_{i=1}^{n}\frac{1}{|X_{i}-a|^{d-\alpha}}\overset{\text{a.s.}}{\rightarrow}\mathbb{E}_{g}\left[  \frac{1}{|X-a|^{d-\alpha}}\right]  >0,
\]
which proves $\mathcal{F}$ is not $P_{g}$-GC. Next, we show $\mathcal{F}_{j}$ does not satisfy Assumptions \ref{GC_class} or \ref{maximal_inequality} if $\phi_{j}$ is not almost everywhere zero. Without loss of generality, assume $\{x\in\mathcal{X}:\phi_{j}(x)>0\}$ has a positive Lebesgue measure. It suffices to disprove $||P_{n}-P_{g}||_{\mathcal{F}_{j}}\overset{\text{P*}}{\rightarrow}0$ since both assumptions imply this conclusion. We modify the above function $f$ as%
\begin{equation}
f(x)=\left\{
\begin{array}[c]{l}%
1/|x-a|^{d-\alpha}\\
0
\end{array}
\left.
\begin{array}[c]{l}%
\text{if }x\neq a\text{ and }a,x,X_{i}\text{ are collinear for some }X_{i}\text{ with }\phi_{j}(X_{i})>0\\
\text{otherwise.}%
\end{array}
\right.  \right.   \label{counterexample2}%
\end{equation}
Then we still have $f\in\mathcal{F}$ and $f=0$ almost everywhere. Therefore,%
\begin{align*}
||P_{n}-P_{g}||_{\mathcal{F}_{j}}\geq\left\vert \frac{1}{n}\sum_{i=1}^{n}\phi_{j}(X_{i})\frac{f(X_{i})}{g(X_{i})}\right\vert &=\frac{1}{n}\sum_{i=1}^{n}\frac{\phi_{j}(X_{i})I(\phi_{j}(X_{i})>0)}{|X_{i}-a|^{d-\alpha}g(X_{i})}\\
&\overset{\text{a.s.}}{\rightarrow}\int_{\mathcal{X}}\frac{\phi_{j}(x)I(\phi_{j}(x)>0)}{|x-a|^{d-\alpha}}dx>0,
\end{align*}
which disproves $||P_{n}-P_{g}||_{\mathcal{F}_{j}}\overset{\text{P*}}{\rightarrow}0$.

When $\alpha\geq d$, we simply change the nonzero function value in (\ref{counterexample1}) and (\ref{counterexample2}) into $M$ and the same proof can be applied.
\end{proof}

\begin{proof}[Proof of Proposition \ref{2Dcounterexample}.]
We first prove that $val(\mathcal{P}_{\alpha})\leq(2^{\alpha}-1)/2^{\alpha}$.
Consider any feasible function $f$ in ($\mathcal{P}_{\alpha}$). By changing
into polar coordinates, the objective function can be written as%
\[
\mathbb{E}_{f}[I(X\in S)]=\int_{0}^{\pi/2}\int_{1}^{2}f(r\cos\theta
,r\sin\theta)rdrd\theta.
\]
The definition of $\alpha$-unimodality implies $r^{2-\alpha}f(r\cos
\theta,r\sin\theta)$ is non-increasing in $r\in(0,\infty)$. Therefore, the
objective function can be bounded by%
\begin{align*}
\mathbb{E}_{f}[I(X\in S)]  &  =\int_{0}^{\pi/2}\int_{1}^{2}f(r\cos\theta
,r\sin\theta)r^{2-\alpha}\frac{1}{r^{1-\alpha}}drd\theta\\
&  \leq\int_{0}^{\pi/2}\int_{1}^{2}f(\cos\theta,\sin\theta)\frac
{1}{r^{1-\alpha}}drd\theta\\
&  =\int_{0}^{\pi/2}f(\cos\theta,\sin\theta)\int_{1}^{2}\frac{1}{r^{1-\alpha}%
}drd\theta\\
&  =\int_{0}^{\pi/2}f(\cos\theta,\sin\theta)\frac{2^{\alpha}-1}{\alpha}%
d\theta.
\end{align*}
On the other hand, by a similar argument, we have%
\begin{align*}
\mathbb{E}_{f}[I(X\geq0,||X||\leq1)]  &  =\int_{0}^{\pi/2}\int_{0}^{1}%
f(r\cos\theta,r\sin\theta)r^{2-\alpha}\frac{1}{r^{1-\alpha}}drd\theta\\
&  \geq\int_{0}^{\pi/2}\int_{0}^{1}f(\cos\theta,\sin\theta)\frac
{1}{r^{1-\alpha}}drd\theta\\
&  =\int_{0}^{\pi/2}f(\cos\theta,\sin\theta)\frac{1}{\alpha}d\theta\\
&  \geq\frac{1}{2^{\alpha}-1}\mathbb{E}_{f}[I(X\in S)].
\end{align*}
Since $1=\mathbb{E}_{f}[I(X\in\mathcal{X})]=\mathbb{E}_{f}[I(X\geq0,||X||\leq
1)]+\mathbb{E}_{f}[I(X\in S)]$, we obtain $\mathbb{E}_{f}[I(X\in S)]\leq(2^{\alpha}-1)/2^{\alpha}$. Since the above arguments hold for any feasible $f$, we know that
$val(\mathcal{P}_{\alpha})\leq(2^{\alpha}-1)/2^{\alpha}$.

Next, we will prove that $\mathbb{P}_{\ast}(\exists N_{0}\in\mathbb{N}\text{ s.t. }val(\mathcal{P}_{\alpha
.n})=1,\forall n\geq N_{0})=1$.

We first consider the case $\alpha<2$. Since the sampling distribution is
continuously distributed, any three of $(0,0),X_{1},\ldots X_{n}$ are not
collinear a.s. For any $C>0$ and any realization of $X_{1},\ldots,X_{n}$, we
define the following function
\[
f(x;C,X_{1},\ldots,X_{n})=\left\{
\begin{array}
[c]{l}%
C/||x||^{2-\alpha}\\
0
\end{array}
\left.
\begin{array}
[c]{l}%
\text{if }x,(0,0)\text{ and }X_{i}\text{ are collinear for some }X_{i}\in S\\
\text{otherwise.}%
\end{array}
\right.  \right.  
\]
All of them are in $\mathcal{F}$ because%
\[
r^{2-\alpha}f(rx;C,X_{1},\ldots,X_{n})=\left\{
\begin{array}
[c]{l}%
C/||x||^{2-\alpha}\\
0
\end{array}
\left.
\begin{array}
[c]{l}%
\text{if }x,(0,0)\text{ and }X_{i}\text{ are collinear for some }X_{i}\in S\\
\text{otherwise}%
\end{array}
\right.  \right.
\]
is non-increasing in $r$. Therefore, the functions $L(x;C,X_{1},\ldots
,X_{n})=f(x;C,X_{1},\ldots,X_{n})/g(x)$ are in $\mathcal{L}$. Then, we will
prove that the constraint $(1/n)\sum_{i=1}^{n}I(X_{i}\in\mathcal{X})L(X_{i};C,X_{1},\ldots,X_{n})=1$ is satisfied for some $C$. We notice that
\begin{align*}
&  \frac{1}{n}\sum_{i=1}^{n}I(X_{i}\in\mathcal{X})L(X_{i};C,X_{1},\ldots
,X_{n})\\
&  =\frac{1}{n}\sum_{i=1}^{n}I(X_{i}\in S)\frac{C}{||X_{i}||^{2-\alpha}%
g(X_{i})}\\
&  \overset{a.s.}{\rightarrow}\int_{S}\frac{C}{||x||^{2-\alpha}g(x)}g(x)dx\\
&  =\int_{0}^{\pi/2}\int_{1}^{2}\frac{C}{r^{2-\alpha}}rdrd\theta\\
&  =\frac{\pi(2^{\alpha}-1)}{2\alpha}C
  =\left\{
\begin{array}
[c]{l}%
2\\
1/2
\end{array}
\left.
\begin{array}
[c]{l}%
\text{if }C=4\alpha/(\pi(2^{\alpha}-1))\\
\text{if }C=\alpha/(\pi(2^{\alpha}-1)).
\end{array}
\right.  \right.  
\end{align*}
Since
\[
\frac{1}{n}\sum_{i=1}^{n}I(X_{i}\in\mathcal{X})L(X_{i};C,X_{1},\ldots
,X_{n})=\frac{1}{n}\sum_{i=1}^{n}I(X_{i}\in S)\frac{C}{||X_{i}||^{2-\alpha
}g(X_{i})}%
\]
is continuous in $C$, when the above limit holds, we know that $\exists
N_{0}\in\mathbb{N}$ such that for all $n\geq N_{0}$ there must be some
$C_{0}\in(\alpha/(\pi(2^{\alpha}-1)),4\alpha/(\pi(2^{\alpha}-1)))$ (depending
on $n$ and sampling realizations) such that $(1/n)\sum_{i=1}^{n}I(X_{i}\in\mathcal{X})L(X_{i};C_{0},X_{1}%
,\ldots,X_{n})=1$. Moreover, the objective function is also 1 because $I(X_{i}\in S)L(X_{i};C_{0},X_{1},\ldots,X_{n})=I(X_{i}\in\mathcal{X}%
)L(X_{i};C_{0},X_{1},\ldots,X_{n})$. Therefore, we prove that $\mathbb{P}_{\ast}(\exists N_{0}\in\mathbb{N}\text{ s.t. }val(\mathcal{P}_{\alpha
.n})=1,\forall n\geq N_{0})=1.$

Then we consider the case $\alpha\geq2$. For any $C\in(0,M]$ and any
realization of $X_{1},\ldots,X_{n}$, we define the following function (still
denoted by $f(x;C,X_{1},\ldots,X_{n})$)
\[
f(x;C,X_{1},\ldots,X_{n})=\left\{
\begin{array}
[c]{l}%
C\\
0
\end{array}
\left.
\begin{array}
[c]{l}%
\text{if }x,(0,0)\text{ and }X_{i}\text{ are collinear for some }X_{i}\in S\\
\text{otherwise.}%
\end{array}
\right.  \right.  
\]
All of them are in $\mathcal{F}$ because%
\[
r^{2-\alpha}f(rx;C,X_{1},\ldots,X_{n})=\left\{
\begin{array}
[c]{l}%
r^{2-\alpha}C\\
0
\end{array}
\left.
\begin{array}
[c]{l}%
\text{if }x,(0,0)\text{ and }X_{i}\text{ are collinear for some }X_{i}\in S\\
\text{otherwise}%
\end{array}
\right.  \right.
\]
is non-increasing in $r$. Therefore, the functions $L(x;C,X_{1},\ldots
,X_{n})=f(x;C,X_{1},\ldots,X_{n})/g(x)$ are in $\mathcal{L}$. Like the
previous case, we will prove that the constraint $(1/n)\sum_{i=1}^{n}I(X_{i}\in\mathcal{X})L(X_{i};C,X_{1},\ldots,X_{n})=1$ is satisfied for some $C$. We notice that
\begin{align*}
&  \frac{1}{n}\sum_{i=1}^{n}I(X_{i}\in\mathcal{X})L(X_{i};C,X_{1},\ldots
,X_{n})\\
&  =\frac{1}{n}\sum_{i=1}^{n}I(X_{i}\in S)\frac{C}{g(X_{i})}\\
&  \overset{a.s.}{\rightarrow}\int_{S}\frac{C}{g(x)}g(x)dx\\
&  =\frac{3\pi C}{4}
  =\left\{
\begin{array}
[c]{l}%
3\pi M/4>1\\
1/2
\end{array}
\left.
\begin{array}
[c]{l}%
\text{if }C=M\\
\text{if }C=2/(3\pi).
\end{array}
\right.  \right.  
\end{align*}
Therefore, when the above limit holds, we know that $\exists N_{0}%
\in\mathbb{N}$ such that for all $n\geq N_{0}$ there must be some $C_{0}%
\in(2/(3\pi),M)$ (depending on $n$ and sampling realizations) such that $(1/n)\sum_{i=1}^{n}I(X_{i}\in\mathcal{X})L(X_{i};C_{0},X_{1}%
,\ldots,X_{n})=1$. Moreover, the objective function is also 1 because $I(X_{i}\in S)L(X_{i};C_{0},X_{1},\ldots,X_{n})=I(X_{i}\in\mathcal{X}%
)L(X_{i};C_{0},X_{1},\ldots,X_{n})$. Therefore, we prove that $\mathbb{P}_{\ast}(\exists N_{0}\in\mathbb{N}\text{ s.t. }val(\mathcal{P}_{\alpha
.n})=1,\forall n\geq N_{0})=1.$

Now for any $\alpha$, we have proved $\mathbb{P}_{\ast}(\exists N_{0}\in\mathbb{N}\text{ s.t. }val(\mathcal{P}_{\alpha
.n})=1,\forall n\geq N_{0})=1$. Recall that we have also proved $val(\mathcal{P}_{\alpha})\leq(2^{\alpha}-1)/2^{\alpha}$, which means the SAA optimal value is not consistent.
\end{proof}%

\begin{proof}[Proof of Theorem \ref{complexity}.]
We have%
\begin{align*}
\mathbb{E}_{g}[N_{2}(n,d)]  & =\mathbb{E}_{g}\left[  \sum_{1\leq i\neq j\leq n}I(X_{i}\geq X_{j}\text{ and it's non-redundant})\right]  \\
& =n(n-1)\mathbb{P}_{g}(X_{1}\geq X_{2}\text{ and it's non-redundant}).
\end{align*}
Let $F_{i},i=1,\ldots d$ be the marginal distributions of $X_{1}$. We can see%
\begin{align*}
& \mathbb{P}_{g}(X_{1}\geq X_{2}\text{ and it is non-redundant})\\
& =\int_{-\infty}^{\infty}\int_{x_{1}}^{\infty}\cdots\int_{-\infty}^{\infty}\int_{x_{d}}^{\infty}\mathbb{P}_{g}(X_{i}\text{ violates }y=X_{1}\geq X_{i}\geq X_{2}=x,\forall i\geq3)dF_{d}(y_{d})dF_{d}(x_{d})\cdots dF_{1}(y_{1})dF_{1}(x_{1})\\
& =\int_{-\infty}^{\infty}\int_{x_{1}}^{\infty}\cdots\int_{-\infty}^{\infty}\int_{x_{d}}^{\infty}\left(  1-{\prod_{i=1}^{d}}(F_{i}(y_{i})-F_{i}(x_{i}))\right)  ^{n-2}dF_{d}(y_{d})dF_{d}(x_{d})\cdots dF_{1}(y_{1})dF_{1}(x_{1})\\
& =\int_{-\infty}^{\infty}\int_{x_{1}}^{\infty}\cdots\int_{-\infty}^{\infty}\int_{x_{d}}^{\infty}\sum_{k=0}^{n-2}\binom{n-2}{k}(-1)^{k}{\prod_{i=1}^{d}}(F_{i}(y_{i})-F_{i}(x_{i}))^{k}dF_{d}(y_{d})dF_{d}(x_{d})\cdots dF_{1}(y_{1})dF_{1}(x_{1})\\
& =\sum_{k=0}^{n-2}\binom{n-2}{k}(-1)^{k}\int_{-\infty}^{\infty}\int_{x_{1}}^{\infty}\cdots\int_{-\infty}^{\infty}\int_{x_{d}}^{\infty}{\prod_{i=1}^{d}}(F_{i}(y_{i})-F_{i}(x_{i}))^{k}dF_{d}(y_{d})dF_{d}(x_{d})\cdots dF_{1}(y_{1})dF_{1}(x_{1})\\
& =\sum_{k=0}^{n-2}\binom{n-2}{k}(-1)^{k}{\prod_{i=1}^{d}}\int_{-\infty}^{\infty}\int_{x_{i}}^{\infty}(F_{i}(y_{i})-F_{i}(x_{i}))^{k}dF_{i}(y_{i})dF_{i}(x_{i})\\
& =\sum_{k=0}^{n-2}\frac{\binom{n-2}{k}(-1)^{k}}{(k+1)^{d}(k+2)^{d}},
\end{align*}
which implies that
\[
\mathbb{E}_{g}[N_{2}(n,d)]=n(n-1)\sum_{k=0}^{n-2}\frac{\binom{n-2}{k}(-1)^{k}}{(k+1)^{d}(k+2)^{d}}.
\]
When $d=2$, we can further simplify $\mathbb{E}_{g}[N_{2}(n,d)]$. First note that%
\begin{align*}
\mathbb{E}_{g}[N_{2}(n,2)]  & =n(n-1)\sum_{k=0}^{n-2}\frac{\binom{n-2}{k}(-1)^{k}}{(k+1)^{2}(k+2)^{2}}\\
& =\sum_{k=0}^{n-2}\frac{\binom{n}{k+2}(-1)^{k}}{(k+1)(k+2)}\\
& =\sum_{k=2}^{n}\frac{\binom{n}{k}(-1)^{k}}{k(k-1)}\\
& =\sum_{k=2}^{n}\frac{\binom{n}{k}(-1)^{k}}{k-1}-\sum_{k=2}^{n}\frac{\binom{n}{k}(-1)^{k}}{k}.
\end{align*}
We define two functions%
\[
f_{1}(x)=\sum_{k=2}^{n}\frac{\binom{n}{k}(-1)^{k}}{k-1}x^{k-1},\quad f_{2}(x)=\sum_{k=2}^{n}\frac{\binom{n}{k}(-1)^{k}}{k}x^{k}%
\]
with derivatives%
\[
f_{1}^{\prime}(x)=\sum_{k=2}^{n}\binom{n}{k}(-1)^{k}x^{k-2}=\frac{(1-x)^{n}-1+nx}{x^{2}}%
\]
and%
\[
f_{2}^{\prime}(x)=\sum_{k=2}^{n}\binom{n}{k}(-1)^{k}x^{k-1}=\frac{(1-x)^{n}-1+nx}{x}.
\]
Therefore, we can see%
\begin{align*}
f_{1}(1)
& =f_{1}(0)+\int_{0}^{1}f_{1}^{\prime}(x)dx\\
& =\int_{0}^{1}\frac{(1-x)^{n}-1+nx}{x^{2}}dx\\
& =\int_{0}^{1}\frac{x^{n}-1+n(1-x)}{(1-x)^{2}}dx\\
& =\int_{0}^{1}\frac{(x-1)(x^{n-1}+\cdots+x+1-n)}{(1-x)^{2}}dx\\
& =\int_{0}^{1}\frac{(x-1)^{2}(x^{n-2}+2x^{n-3}+\cdots+(n-2)x+n-1)}{(1-x)^{2}%
}dx\\
& =\int_{0}^{1}(x^{n-2}+2x^{n-3}+\cdots+(n-2)x+n-1)dx\\
& =\sum_{i=1}^{n-1}\frac{n-i}{i}%
\end{align*}
and%
\begin{align*}
f_{2}(1)
& =f_{2}(0)+\int_{0}^{1}f_{2}^{\prime}(x)dx\\
& =\int_{0}^{1}\frac{(1-x)^{n}-1+nx}{x}dx\\
& =\int_{0}^{1}\frac{x^{n}-1+n(1-x)}{1-x}dx\\
& =\int_{0}^{1}(n-x^{n-1}-\cdots-x-1)dx\\
& =n-\sum_{i=1}^{n}\frac{1}{i}.
\end{align*}
Plugging $f_{1}(1)$ and $f_{2}(1)$ into $\mathbb{E}_{g}[N_{2}(n,2)]$, we obtain%
\[
\mathbb{E}_{g}[N_{2}(n,2)]=\sum_{i=1}^{n-1}\frac{n-i}{i}-\left(  n-\sum_{i=1}^{n}\frac{1}{i}\right)  =\sum_{i=1}^{n}\frac{n-i}{i}-\left(  n-\sum_{i=1}^{n}\frac{1}{i}\right)  =(n+1)\sum_{i=1}^{n}\frac{1}{i}-2n.
\]

Finally, we prove the inequality
\[
\frac{1}{2}\mathbb{E}_{g}[N_{2}(n,d-1)]\leq \mathbb{E}_{g}[N_{2}(n,d)]\leq\frac{n(n-1)}{2^{d}}.
\]
The second inequality directly follows from $\mathbb{E}_{g}[N_{2}(n,d)]\leq \mathbb{E}_{g}[N_{1}(n,d)]=n(n-1)/2^{d}$. For the first one, according to the previous calculation,%
\begin{align*}
& \mathbb{P}_{g}(X_{1}\geq X_{2}\text{ and it is non-redundant})\\
& =\int_{-\infty}^{\infty}\int_{x_{1}}^{\infty}\cdots\int_{-\infty}^{\infty}\int_{x_{d}}^{\infty}\left(  1-{\prod_{i=1}^{d}}(F_{i}(y_{i})-F_{i}(x_{i}))\right)  ^{n-2}dF_{d}(y_{d})dF_{d}(x_{d})\cdots dF_{1}(y_{1})dF_{1}(x_{1})\\
& \geq\int_{-\infty}^{\infty}\int_{x_{1}}^{\infty}\cdots\int_{-\infty}^{\infty}\int_{x_{d}}^{\infty}\left(  1-{\prod_{i=1}^{d-1}}(F_{i}(y_{i})-F_{i}(x_{i}))\right)  ^{n-2}dF_{d}(y_{d})dF_{d}(x_{d})\cdots dF_{1}(y_{1})dF_{1}(x_{1})\\
& =\int_{-\infty}^{\infty}\int_{x_{1}}^{\infty}\cdots\int_{-\infty}^{\infty}\int_{x_{d-1}}^{\infty}\left(  1-{\prod_{i=1}^{d-1}}(F_{i}(y_{i})-F_{i}(x_{i}))\right)  ^{n-2}dF_{d-1}(y_{d-1})dF_{d-1}(x_{d-1})\cdots dF_{1}(y_{1})dF_{1}(x_{1})\\
& \times\int_{-\infty}^{\infty}\int_{x_{d}}^{\infty}1dF_{d}(y_{d})dF_{d}(x_{d})\\
& =\frac{1}{2}\int_{-\infty}^{\infty}\int_{x_{1}}^{\infty}\cdots\int_{-\infty}^{\infty}\int_{x_{d-1}}^{\infty}\left(  1-{\prod_{i=1}^{d-1}} (F_{i}(y_{i})-F_{i}(x_{i}))\right)  ^{n-2}dF_{d-1}(y_{d-1})dF_{d-1}(x_{d-1})\cdots dF_{1}(y_{1})dF_{1}(x_{1}).
\end{align*}
which gives us $\mathbb{E}_{g}[N_{2}(n,d)]\geq \mathbb{E}_{g}[N_{2}(n,d-1)]/2$.
\end{proof}

\end{appendix}

\end{document}